\documentclass[a4paper,reqno,oneside,12pt]{amsart}
\usepackage{fullpage}

\usepackage{amsmath}
\usepackage{amsfonts}
\usepackage{amssymb}
\usepackage{verbatim,enumitem,graphics}
\usepackage[colorlinks]{hyperref}
\usepackage{dsfont}

\usepackage{tabls}
\usepackage{graphicx}
\usepackage[utf8]{inputenc}

\usepackage{tikz}
\tikzset{font=\large}
\usetikzlibrary{calc,math,decorations.pathreplacing,patterns.meta}

\usepackage{enumitem}
\usepackage{MnSymbol}
\usepackage{caption}

\setenumerate{leftmargin=*} % ,labelindent=\parindent}

\newtheorem{theorem}{Theorem}[section]
\newtheorem{lemma}[theorem]{Lemma}
\newtheorem{proposition}[theorem]{Proposition}

\newtheorem{conjecture}[theorem]{Conjecture}
\newtheorem{definition}[theorem]{Definition}

\newtheorem{claim}[theorem]{Claim}

\newtheorem{question}[theorem]{Question}

\newcommand{\oldqed}{}
\def\endofClaim{\hfill\scalebox{.6}{$\Box$}}

\newenvironment{claimproof}[1][Proof]{
\renewcommand{\oldqed}{\qedsymbol}
\renewcommand{\qedsymbol}{\endofClaim}
\begin{proof}[#1]
}{
\end{proof}
\renewcommand{\qedsymbol}{\oldqed}
}

\newcommand{\Gnp}{G(n,p)}

\newcommand{\cH}{\mathcal{H}}
\newcommand{\cM}{\mathcal{M}}

\newcommand{\cF}{\mathcal{F}}

\newcommand{\cP}{\mathcal{P}}
\newcommand{\cJ}{\mathcal{J}}

\newcommand{\cK}{\mathcal{K}}
\newcommand{\cD}{\mathcal{D}}
\newcommand{\cI}{\mathcal{I}}
\newcommand{\cT}{\mathcal{T}}

\newcommand{\PP}{\mathbb{P}}
\newcommand{\EE}{\mathbb{E}}

\newcommand{\Var}{\mathrm{Var}}

\newcommand{\eps}{\varepsilon}

\newcommand{\cB}{\mathcal{B}}

\newcommand{\hp}{\hat{p}_{\alpha}(n)}

\renewcommand{\subset}{\subseteq}

\title{The square of a Hamilton cycle in randomly perturbed graphs}

\author[J.~B\"{o}ttcher]{Julia B\"{o}ttcher$^\ast$}
\author[O.~Parczyk]{Olaf Parczyk$^{\dag}$}
\author[A.~Sgueglia]{Amedeo Sgueglia$^\natural$}
\author[J.~Skokan]{Jozef Skokan$^{\ast,\S}$}

\thanks{$^\ast$ Department of Mathematics, London School of Economics, London, WC2A 2AE, UK. \\
\textit{E-mail:} \{\texttt{j.boettcher|j.skokan}\}\texttt{@lse.ac.uk}}
\thanks{$^\dag$ 
Institute of Mathematics, Freie Universität Berlin, Arnimallee 3, 14195 Berlin, Germany. \\
OP is supported by the Deutsche Forschungsgemeinschaft (DFG, German Research Foundation) under Germany's Excellence Strategy – The Berlin Mathematics Research Center MATH+ (EXC-2046/1, project ID: 390685689).
Part of this research was conducted while OP was a visiting fellow at the London School of Economics and supported by the Deutsche Forschungsgemeinschaft (DFG, Grant PA 3513/1-1).\\
\textit{E-mail:} \texttt{parczyk@mi.fu-berlin.de}}
\thanks{$^\natural$ Department of Mathematics, University College London, London WC1E 6BT, UK. \\
This research was conducted while AS was a PhD student at the London School of Economics.\\
\textit{E-mail:} \texttt{a.sgueglia@ucl.ac.uk}}
\thanks{$^\S$ Department of Mathematics, University of Illinois at Urbana-Champaign,  Urbana, IL 61801, USA}

\begin{document}

\begin{abstract}
	We investigate the appearance of the square of a Hamilton cycle in the model of
	randomly perturbed graphs, which is, for a given $\alpha \in (0,1)$, the union
	of any $n$-vertex graph with minimum degree $\alpha n$ and the binomial random
	graph $G(n,p)$.  This is known when $\alpha > 1/2$ and we determine the exact
	perturbed threshold probability in all the remaining cases, i.e., for each
	$\alpha \le 1/2$. We demonstrate that, as $\alpha$ ranges over the interval
	$(0,1)$, the threshold performs a countably infinite number of `jumps'.  Our result has
	implications on the perturbed threshold for $2$-universality, where we also
	fully address all open cases.
\end{abstract}

\maketitle

\thispagestyle{empty}
\section{Introduction}

In extremal graph theory, the square of a Hamilton cycle often serves as a good
concrete but reasonably complex special case when results about the appearance
of a more general class of structures are still out of reach. Here, a
\emph{Hamilton cycle} is a cycle through all the vertices of a graph and the
\emph{square} of a graph~$H$ is obtained from~$H$ by adding edges between all
vertices of distance~two in~$H$.

For example, a well-known conjecture by P\'osa from the 1960s
(see~\cite{Posa:conj}) states that any $n$-vertex graph~$G$ of minimum degree at
least $2n/3$ contains the square of a Hamilton cycle. This was solved in
the 1990s for large~$n$ by Koml\'os, Sark\"ozy and
Szemer\'edi~\cite{KSSsquare}, demonstrating the power of the then new and by now
celebrated Blow-Up Lemma. Only more than 10 years later the analogous problem
was settled for a more general class of spanning subgraphs~\cite{BSTthree},
using the result for squares of Hamilton cycles as a fundamental stepping stone.

In random graphs, on the other hand, the threshold for the
containment of the square of a Hamilton cycle was determined only very recently
and, surprisingly, this proved to be much harder than the corresponding problem
for higher powers of Hamilton cycles (see also Section~\ref{sec:relatedWork}).
Here, the model of random graphs considered is the \emph{binomial random graph}
$G(n,p)$, which is a graph on~$n$ vertices in which each pair forms an edge with
probability~$p$ independently of other pairs. By a simple first moment
calculation the threshold for the appearance of a copy of the square of a Hamilton cycle is at least
$n^{-1/2}$.  Resolving a conjecture of K\"uhn and Osthus~\cite{KueOst:Posa} and
improving on results of Nenadov and \v{S}kori\'c~\cite{NenSko:Posa} and of Fischer,
\v{S}kori\'c, Steger and Truji\'c~\cite{FisSkoSteTru}, it was only recently
proved that $n^{-1/2}$ is indeed the threshold by Kahn, Narayanan and
Park~\cite{square_random}, using tools from the pioneering work of Frankston,
Kahn, Narayanan and Park~\cite{FraKahNarPar} on fractional
expectation-thresholds.

In this paper our focus is on a combination of these two themes. We consider the
question when the square of a Hamilton cycle appears in the randomly perturbed
graph model and completely settle this. The \emph{randomly perturbed graph}
$G_\alpha \cup G(n,p)$, a model introduced by Bohman, Frieze and
Martin~\cite{bohman2003}, is a graph obtained by taking a deterministic graph
$G_\alpha$ on~$n$ vertices with minimum degree at least $\alpha n$ and adding
the edges of a random graph $G(n,p)$ on the same vertex set. In this model,
which received a large amount of attention recently (see also
Section~\ref{sec:relatedWork}), one is interested in determining the behaviour
of the threshold for a given graph property in dependence of~$\alpha$. More
precisely, we say that $\hat{p}_\alpha(n)$ is the \emph{perturbed threshold} for a
property~$\cP$ and a fixed $\alpha \in [0,1)$ if there are constants $C>c>0$
such that for any $p \ge C \hat{p}_\alpha$ and for any sequence of $n$-vertex graphs
$(G_{\alpha,n})_{n\in\mathbb{N}}$ with $\delta(G_{\alpha,n}) \ge \alpha n$ we
have $\lim_{n\to\infty}\PP\big(G_{\alpha,n} \cup G(n,p) \in \cP \big)=1$, and
for any $p \le c \hat{p}_\alpha$ there exists a sequence of $n$-vertex graphs
$(G_{\alpha,n})_{n\in\mathbb{N}}$ with $\delta(G_{\alpha,n})\ge\alpha n$ such
that $\lim_{n\to\infty}\PP\big( G_{\alpha,n} \cup G(n,p)\in \cP\big) =0$.

For $\alpha=0$ the perturbed threshold is simply the usual threshold for purely
random graphs and the perturbed threshold is~$0$ for any $\alpha$ such that all
graphs with minimum degree $\alpha n$ are in $\cP$. In this sense, randomly
perturbed graphs interpolate between questions from extremal graph theory and
questions concerning random graphs. For small $\hat p_\alpha$ (and hence
relatively large $\alpha$) there are some analogies to the concept of smoothed
analysis in the theory of algorithms (see, e.g., \cite{smoothedAnalysis}), and
for small~$\alpha$ one asks how much random graph theory results are influenced
by the fact that in a random graph there may be vertices with relatively few
neighbours. But in general one would like to determine the evolution of the
perturbed threshold for the whole range of~$\alpha$, which has so far only been
achieved for very few properties.  

This paper contributes to this line of
research by determining the perturbed threshold for the containment of squares
of Hamilton cycles for $0 < \alpha \le \frac12$, which answers a question of
Antoniuk, Dudek, Reiher, Ruci\'nski and Schacht~\cite{antoniuk2020high} in a
strong form.  In the range $\alpha \in (\frac12,\frac23)$ the perturbed
threshold for squares of Hamilton cycles was determined by Dudek, Reiher,
Ruci\'nski and Schacht~\cite{dudek2020powers}, as we discuss in more detail in
Section~\ref{sec:relatedWork}.  The case $\alpha=0$, on the other hand, is the
purely random graph case addressed in~\cite{square_random} and the range
$\alpha\ge\frac23$ is the purely extremal scenario addressed in~\cite{KSSsquare}.
Therefore, our result completely settles the question of determining the
perturbed threshold for the square of a Hamilton cycle for the whole range of
$\alpha$.

\begin{theorem}[Square of a Hamilton cycle]
	\label{thm:main_small}
	The perturbed threshold $\hat{p}_\alpha(n)$ for the containment of the square of a Hamilton cycle is
	\[
	\hat{p}_\alpha(n) =  
	\begin{cases}
		0 & \text{if } \alpha\ge\frac23\,, \\
		n^{-1} & \text{if } \alpha \in [\frac{1}{2},\frac{2}{3}) \, , \\ 
		{n^{-(k-1)/(2k-3)}} & \text{if } \alpha \in (\frac{1}{k+1},\frac{1}{k}) \text{ for $k\ge 2$}\, , \\ 
		{n^{-(k-1)/(2k-3)}}(\log n)^{1/(2k-3)} & \text{if } \alpha = \tfrac{1}{k+1}  \text{ for $k\ge 2$}\, , \\
		n^{-1/2} & \text{if } \alpha=0\, . 
	\end{cases}
	\]	
\end{theorem}

In other words, as long as $\alpha\in(\frac13,\frac23)$ it suffices to add a linear number
of random edges to the deterministic graph $G_\alpha$ for enforcing the square
of a Hamilton cycle, and for $\alpha\le\frac13$ the perturbed threshold
$\hat{p}_\alpha(n)$ exhibits `jumps' at $\alpha=\frac{1}{k+1}$ for each integer $k \ge 2$,
where an extra $\log$-factor is needed for~$\alpha$
precisely equal to~$\frac{1}{k+1}$. A similar `jumping' phenomenon for the perturbed threshold
has been already observed for other subgraph containment problems: for
$K_r$-factors in~\cite{han2019tilings}
and for $C_\ell$-factors in~\cite{cycles_factors_full}.
However, Theorem~\ref{thm:main_small} is to the best of our knowledge the first
result exhibiting a countably infinite number of `jumps'.
Moreover, for $\alpha$ tending to zero, the threshold $\hat{p}_\alpha(n)$ tends to
$n^{-1/2}$, which is precisely the threshold for the square of a Hamilton cycle
in $G(n,p)$ alone as discussed above.

Since the square of a Hamilton cycle on~$n$ vertices contains each $n$-vertex
graph with maximum degree two as a subgraph, as a corollary to
Theorem~\ref{thm:main_small} we also get the following result, establishing the
perturbed threshold for $2$-universality for all~$\alpha$. Here, we say that a
graph is \emph{$r$-universal} if it contains all graphs of maximum degree at most~$r$.

\begin{theorem}[$2$-universality]
	\label{cor:2_univ}
	The perturbed threshold $\hat{p}_\alpha(n)$ for $2$-universality is
	\[
	\hat{p}_\alpha(n) =  
	\begin{cases}
		0 & \text{if } \alpha\ge\frac23\,, \\
		n^{-1} & \text{if } \alpha \in (\frac{1}{3},\frac{2}{3}) \, , \\ 
		n^{-1}\log n & \text{if } \alpha=\frac13 \, , \\ 
		n^{-2/3} & \text{if } \alpha \in (0,\frac{1}{3}) \, , \\ 
		n^{-2/3}(\log n)^{1/3} & \text{if } \alpha=0\,.
	\end{cases}
	\]	
\end{theorem}

The range $\frac 13 \le \alpha < \frac23$ is a corollary of our Theorem~\ref{thm:main_small}, while the remaining cases follow from known results. For $\alpha=0$, this
is due to Ferber, Kronenberg and Luh~\cite{ferber2019optimal}, for $\alpha\ge
\frac23$ to Aigner and Brandt~\cite{aigner1993embedding}, and for $\alpha \in
(0,\frac13)$ to~\cite{parczyk20202}. The result in the range $\alpha\in[\frac13,\frac23)$ significantly strengthens one of our results from~\cite{triangle_paper},
where the same result was established for the containment of a triangle factor
only.

Observe that the perturbed threshold for the containment of the square of a
Hamilton cycle and the perturbed threshold for $2$-universality differ for
$\alpha<\frac13$. This is due to the fact that in this regime the structure of the
deterministic graph $G_\alpha$ may force us to find many copies of the square of
a short path in $G(n,p)$ if we want to find the square of a Hamilton cycle in $G_\alpha\cup
G(n,p)$ (see Section~\ref{subsec:lower_bounds} for more details).

\subsection{Related work}
\label{sec:relatedWork}

As indicated before, there has recently been a wealth of results on properties
of randomly perturbed graphs. Let us close our introduction by briefly
reviewing those concerning the containment of $r$-th powers of Hamilton
cycles. Analogously to the square, the
\emph{$r$-th power} of a graph~$H$ is obtained from~$H$ by adding edges between all
vertices of distance at most~$r$ in~$H$.

We start with the case $r=1$. Here, the case $\alpha\ge \frac12$ is the classical
Theorem of Dirac~\cite{dirac1952}, which asserts that each $n$-vertex graph with
minimum degree at least $\frac 12 n$ has a Hamilton cycle, hence no random edges
are required.  The case $\alpha=0$ is treated by a famous result of
P\'osa~\cite{posa}, which shows that in $G(n,p)$ the threshold for the
containment of a Hamilton cycle is $n^{-1} \log n$. For $\alpha$ between these
two extremes, Bohman, Frieze and Martin~\cite{bohman2003} determined the
perturbed threshold. They proved that for any $\alpha \in (0,\frac12)$, the randomly
perturbed graph $G_\alpha \cup G(n,p)$ a.a.s. has a Hamilton cycle for $p \ge C/n$
with~$C$ sufficiently large, and that this is optimal because for making the
complete bipartite graph $K_{\alpha n,(1-\alpha)n}$ Hamiltonian we need a linear
number of edges.

Turning to $r\ge 2$, proving a conjecture of
Seymour~\cite{Seymour:conj} for large~$n$, Koml\'os, Sark\"ozy and
Szemer\'edi~\cite{komlos1998proof} showed that any large $n$-vertex graph~$G$ with
minimum degree $\delta(G)\ge\frac{r}{r+1}n$ contains the $r$-th power of a Hamilton
cycle, thus establishing that the perturbed threshold is~$0$ for
$\alpha\ge\frac{r}{r+1}$. In $G(n,p)$ alone, i.e. when $\alpha=0$, the threshold for the
containment of th $r$-th power of a Hamilton cycle is  $n^{-1/r}$.
This follows from a much more general result of Riordan~\cite{oliver} for $r \ge 3$ and from~\cite{square_random} for $r = 2$.  The perturbed threshold
behaves differently for any $\alpha>0$, as was shown
in~\cite{bottcher2017embedding}, where it is proved that in $G_\alpha\cup
G(n,p)$ for any $\alpha \in (0,1)$ there exists $\eta>0$ such that the perturbed
threshold for the containment of the $r$-th power of a Hamilton cycle is at most
$n^{-1/r-\eta}$, and it was asked what the optimal~$\eta$ here is. In this paper
we answer this question in the case $r=2$.

More is known for $\alpha\ge\frac r{r+1}$. In this regime, where~$G_\alpha$
alone contains the $r$-th power of a Hamilton cycle, Dudek, Reiher, Ruci\'nski,
and Schacht~\cite{dudek2020powers} showed that adding a linear number of random
edges suffices to enforce an $(r+1)$-st power of a Hamilton cycle. This was improved by Nenadov and Truji\'c~\cite{nenadov2018sprinkling} who showed
that one can indeed enforce the $(2r+1)$-st power of a Hamilton cycle with these
parameters.
When $\alpha > \frac12$, even higher powers of Hamilton cycles have been studied by
Antoniuk, Dudek, Reiher, Ruci\'{n}ski and Schacht~\cite{antoniuk2020high}, who
proved that in many cases the perturbed threshold is guided by the largest clique required
from $G(n,p)$.

For certain values of~$\alpha$ and $r \ge 3$, the perturbed threshold is not yet precisely
known for $K_{r+1}$-factors (see, e.g., \cite{triangle_paper} for
more details).
This suggests that determining the behaviour of the perturbed threshold in the entire range of $\alpha$
for the $r$-th power of a Hamilton cycle for $r>2$ may be challenging.
We discuss this further for $r=3$ in Section~\ref{sec:concluding}.

\subsection*{Organisation}

The rest of the paper is organised as follows.
In the next section we introduce
some fundamental tools, which we will apply in our proofs.  
Then in Section~\ref{sec:overview} we discuss a more general stability version of our main result and provide an overview of our proofs, together with the lower bound constructions
and some auxiliary lemmas.
In Section~\ref{sec:extremal} and~\ref{sec:non-extremal}, we prove the extremal
case and the non-extremal case of our main result, respectively.
%%%
In Section~\ref{sec:aux_reg} we prove a technical lemma that is used in Section~\ref{sec:auxiliary} to prove the auxiliary lemmas.  
We then finish with some concluding remarks and open problems in
Section~\ref{sec:concluding}.  A few standard proofs are postponed to
Appendix~\ref{sec:appendix}.

\subsection*{Notation}
For numbers $a$, $b$, $c$, we write $a = b \pm c$ for $b-c \le a \le b+c$.
Moreover, for non-negative $a$,$b$ we write $0<a \ll b$, when we require $a \le f(b)$ for some function $f \colon \mathbb{R}_{>0} \mapsto \mathbb{R}_{>0}$.
We will only use this to improve readability and in addition to the precise dependencies of the constants.
Moreover, for an event $A=A(n)$ depending on $n \in \mathbb{N}$, we say that $A$ happens asymptotically almost surely (a.a.s.) if $\PP[A] \rightarrow 1$ as $n \rightarrow \infty$.

We use standard graph theory notation.
For a graph $G$ on vertex set $V$ and two disjoint sets $A$, $B \subset V$, we let $G[A]$ be the subgraph of $G$ induced by $A$, $G[A,B]$ be the bipartite subgraph of $G$ induced by sets $A$ and $B$, $e(A)$ be the number of edges with both endpoints in $A$ and $e(A,B)$ be the number of edges with one endpoint in $A$ and the other one in $B$.
More generally, given pairwise-disjoint sets $U_1,\dots,U_h \subseteq V$, we let $G[U_1,\dots,U_h]$ be the induced $h$-partite subgraph $\bigcup_{1\le i<j\le h}G[U_i,U_j]$ of $G$.
Given an integer $r \ge 1$, we denote by $G^r$ the $r$-power of $G$, i.e. the graph obtained from $G$ by adding edges between all vertices of distance at most $r$ in $G$.
We denote the $k$-vertex path by $P_k$ and the $n$-vertex cycle by $C_n$.
Given a subset $W \subseteq V(G)$ and $v \in V(G) \setminus W$, the notation $N_G(v,W)$ stands for the neighbourhood of $v$ in $W$ in the graph $G$ and we denote its size by $\deg_G(v,W)$, where we may omit the index $G$ when the graph is clear from the context.
The $1$-density of $G$ is defined by $ m_1(G) = \max \left\{ \frac{e(F)}{v(F)-1} : F \subseteq G \text{ with } v(F) \ge 2 \right\}$.

Moreover, given $p \in [0,1]$, an integer $k \ge 1$ and $k$ pairwise-disjoint sets of vertices $V_1, \dots, V_k$, we denote by $G(V_1,\dots,V_k,p)$ the random $k$-partite graph with parts $V_1, \dots, V_k$, where each pair of vertices in two different parts forms an edge with probability $p$, independently of other pairs.
Further, we denote by $G(V_1,p)$  the random graph on $V_1$, where each pair of vertices in $V_1$ forms an edge with probability $p$, independently of other pairs.
Moreover, we denote by $\overrightarrow{G}(n,p)$ denote the binomial random directed graph on vertex set $[n]$, where each tuple $(u,v) \in [n]^2$ with $u \not= v$ is a directed edge with probability $p$ independently of all other choices.

Given a copy $F$ of the square of a path on $k$ vertices $P_k^2$, we let $v_1, v_2,\ldots, v_k$ be an ordering of the vertices of $F$ such that its edges are precisely $v_iv_j$, for each $i,j$ with $1\le |i-j|\le 2$. 
We call $(v_2, v_1)$ and $(v_{k-1}, v_k)$ the \emph{end-tuples} of $F$, we refer to $v_i$ as the $i$-th vertex of $F$ and we refer to $F$ as the square of the path $v_1, v_2,\ldots, v_k$.
The choice of taking $(v_2,v_1)$ rather than $(v_1,v_2)$ as end-tuple is intentional. 
This is to ensure that for both the end-tuples $(v_2, v_1)$ and $(v_{k-1}, v_k)$, it is always the second vertex, i.e. $v_1$ and $v_k$ respectively, that is an endpoint of the path $v_1, v_2,\ldots, v_k$.
For simplicity we will talk about tuples $(u,v)$ from a set $V$, when implicitly meaning from $V^2$.

Finally, given $s \ge 1$ and sets $V_1, \dots, V_s$, when we say that a tuple belongs to $\prod_{i=1}^s V_i^k$, we mean that the tuple belongs to $V_1^k \times \dots \times V_s^k$, i.e. it is of the form $(v_{i,j}:1 \le i \le s, 1 \le j \le k)$ with $v_{i,j} \in V_i$ for $i=1,\dots,s$ and $j=1,\dots,k$.

\section{Tools}

We will repeatedly use the following concentration inequality due to Chernoff (see e.g.~\cite[Corollaries 2.3 and 2.4]{JLR} and~\cite{chernoff_entropy}).
\begin{lemma}[Chernoff's inequality]
	\label{lem:chernoff}
	Let $X$ be the sum of independent Bernoulli random variables, then for any $\delta \in (0,1)$ we have
	\[ \PP \left[ |X-\EE[X]| \ge \delta \, \EE[X] \right] \le 2 \exp \left( -\frac{\delta^2}{3} \EE[X] \right) \]
	and for any $k \ge 7 \cdot \EE[X]$ we have $\PP[X > k] \le \exp(-k)$.
	More precisely, if $p$ is the success probability and there are $n$ summands we get
	\[ \PP \left[ X \le \EE[X] - \delta n \right] \le \exp ( -D( p-\delta || p ) \, n )  \, , \]
	where $D(x||y)=x \log (\tfrac{x}{y}) + (1-x) \log (\tfrac{1-x}{1-y})$ is the relative entropy.
\end{lemma}

\subsection{Subgraphs in random graphs}

The following lemma is well-known and follows from a standard application of Janson's inequality (Lemma~\ref{lem:janson}).

\begin{lemma}
	\label{lem:embedding}
	For any graph $F$ and any $\delta >0$, there exists $C>0$ such that the following holds for $p \ge C n^{-1/m_1(F)}$. In the random graph $G(n,p)$ a.a.s.~any set of $\delta n$ vertices contains a copy of $F$.
\end{lemma}

Note that $m_1(P_{k})=1$ and $m_1(P_{k}^2) = \tfrac{2k-3}{k-1}$ and, therefore, the bounds on $p$ given by Lemma~\ref{lem:embedding} for the containment of a copy of $P_{k}$ and $P_{k}^2$ in any linear sized set are $p \ge C/n$ and $p \ge C n^{-(k-1)/(2k-3)}$, respectively.
In a breakthrough result Johansson, Kahn and Vu~\cite{johansson2008factors} determined the threshold for covering all vertices of $G(n,p)$ with pairwise vertex-disjoint copies of $F$, for any strictly $1$-balanced graph $F$, i.e.~those graphs with $1$-density strictly larger than that of any proper subgraph. 
We state their result below.

\begin{theorem}[Johansson, Kahn and Vu~\cite{johansson2008factors}]
	\label{thm:JKV}
	Let $F$ be a graph such that $m_1(F') < m_1(F)$ for all $F' \subseteq F$ with $F' \not= F$ and $v(F') \ge 2$.
	Then there exists $C >0$ such that a.a.s.~in $G(n,p)$ there are $\lfloor n/v(F) \rfloor$ pairwise vertex-disjoint copies of $F$, provided that $p \ge C (\log n)^{1/e(F)} n^{-1/m_1(F)}$.
\end{theorem}

Often we will need to find combinations of squares of paths in $G(n,p)$ whose vertices must satisfy some additional constraints; for that, we will use the following lemma.

\begin{lemma}
	\label{lem:paths2}
	For all integers $s \ge 1$ and $k \ge 2$, and any $0 < \eta \le 1$, there exists $C>0$ such that the following holds for $p \ge C n^{-(k-1)/(2k-3)}$.
	Let $V$ be a vertex set of size $n$, $V_1,\dots,V_s$ not necessarily disjoint subsets of $V$ and $H$ be a collection of pairwise distinct tuples from $\prod_{i=1}^s V_i^k$.
	Then a.a.s.~revealing $\Gamma=G(n,p)$ on $V$ gives the following.
	For any choice of $W_i \subset V_i$ with $i=1,\dots,s$ such that $H'=H \cap \prod_{i=1}^s W_i^k$ has size at least $\eta n^{sk}$, there is a tuple $(v_{i,j}:1 \le i \le s, 1 \le j \le k)$ in $H'$ with pairwise distinct vertices $v_{i,j} \in W_i$ for $i=1,\dots,s$, $j=1,\dots,k$, such that in $\Gamma$ for $i=1,\dots,s$ we have the square of a path on $v_{i,1},\dots,v_{i,k}$ and for $i=1,\dots,s-1$ we have the edge $v_{i,k}v_{i+1,1}$.
\end{lemma}

Observe that the structure we get from Lemma~\ref{lem:paths2} in $G(n,p)$ is given by $s$ copies of the square of a path on $k$ vertices and $s-1$ additional edges joining two consecutive such copies.
Moreover when $k=2$ the structure is simply a path on $2s$ vertices.
In applications, we will often define several collections of tuples $H_j \subseteq \prod_{i=1}^s V_{j,i}^k$ for $j=1,\dots,m$ and apply Lemma~\ref{lem:paths2} to $H=\bigcup_{j=1}^m H_j$, where it is implicit that we apply it with $V_i= \bigcup_{j=1}^m V_{j,i}$.
Also, we stress that, for a fixed $H$ and a typical revealed $G(n,p)$, the conclusion of the lemma holds for any large enough subset of the form $H \cap \prod_{i=1}^s W_i^k$ with $W_i \subseteq V_i$.
In particular, we will be able to claim the existence of a tuple in each subcollection $H_j$, again provided they have the right size. 
The proof of this lemma is standard, uses Janson's inequality and is given in Appendix~\ref{sec:appendix}.

Recall that $\overrightarrow{G}(n,p)$ denotes the binomial random directed graph on vertex set $[n]$, where each tuple $(u,v) \in [n]^2$ with $u \not= v$ is a directed edge with probability $p$ independently of all other choices.
The next theorem will allow us to find a directed Hamilton cycle in $\overrightarrow{G}(n,p)$.

\begin{theorem}[Angluin and Valiant~\cite{angluin1979fast}]
	\label{lem:directed_ham}
	There exists $C>0$ such that for $p \ge C \log n/n$ a.a.s.~ $\overrightarrow{G}(n,p)$ has a directed Hamilton cycle.
\end{theorem}

\subsection{Regularity}

We will use Szemer\'edi's Regularity Lemma~\cite{szem76} and some of its consequences.
Before stating these, we introduce the relevant terminology.
The \emph{density} of a pair $(A,B)$ of disjoint sets of vertices is defined by 
\begin{align*}
	d(A,B)= \frac{e(A,B)}{|A|\cdot |B|}
\end{align*}
and the pair $(A,B)$ is called \emph{$\varepsilon$-regular}, if for all sets $X \subseteq A$ and $Y \subseteq B$ with $|X| \ge \varepsilon |A|$ and $|Y| \ge \varepsilon |B|$ we have $|d(A,B)-d(X,Y)| \le \varepsilon$.

We will use the following well known result, that follows from definitions.

\begin{lemma}[Minimum Degree Lemma]
	\label{lem:MDL}
	Let $(A,B)$ be an $\eps$-regular pair with $d(A,B)=d$.
	Then, for every $Y \subseteq B$ with $|Y| \ge \eps |B|$, the number of vertices from $A$ with degree into $Y$ less than $(d-\eps)|Y|$ is at most $\eps |A|$.
\end{lemma}

With $d \in [0,1]$, a pair $(A,B)$ is called \emph{$(\varepsilon,d)$-super-regular} if, for all sets $X \subseteq A$ and $Y \subseteq B$ with $|X| \ge \varepsilon |A|$ and $|Y| \ge \varepsilon |B|$, we have $d(X,Y) \ge d$ and $\deg(a) \ge d|B|$ for all $a \in A$ and $\deg(b) \ge d|A|$ for all $b \in B$.

The following result is also well known and follows from the definition of super-regularity and Lemma~\ref{lem:MDL}.

\begin{lemma}[Super-regular Pair Lemma]
	\label{lem:superreg}
	Let $(A,B)$ be an $\eps$-regular pair with $d(A,B)=d$.
	Then there exists $A'\subseteq A$ and $B' \subseteq B$ with $|A'|\ge (1-\eps) |A|$ and $|B'| \ge (1-\eps) |B|$ such that $(A',B')$ is a $(2\eps,d-3\eps)$-super-regular pair. 
\end{lemma}

We will use the following well known degree form of the regularity lemma that can be derived from the original version~\cite{szem76}.

\begin{lemma}[Degree form of Szemer\'edi's Regularity Lemma \cite{simon96}]
	\label{lem:reg}
	For every $\varepsilon>0$ and integer $t_0$ there exists an integer $T>t_0$ such that for any graph $G$ on at least $T$ vertices and $d \in [0,1]$ there is a partition of $V(G)$ into $t_0 < t+1 \le T$ sets $V_0,\dots,V_{t}$ and a subgraph $G'$ of $G$ such that
	\begin{enumerate}[label=\upshape(P\arabic*)]
		\item \label{prop:size} $|V_i| = |V_j|$ for all $1 \le i,j \le t$ and $|V_0|\le \varepsilon |V(G)|$,
		\item \label{prop:degree} $\deg_{G'}(v) \ge \deg_G(v) - (d+\varepsilon) |V(G)|$ for all $v \in V(G)$,
		\item \label{prop:indepen} the set $V_i$ is independent in $G'$ for $1 \le i \le t$,
		\item \label{prop:regular} for $1 \le i < j \le t$ the pair $(V_i,V_j)$ is $\varepsilon$-regular in $G'$ and has density either $0$ or at least $d$.
	\end{enumerate}
\end{lemma}

The sets $V_1, \dots, V_t$ are also called clusters and we refer to $V_0$ as the set of exceptional vertices.
A partition $V_0, \dots, V_t$ which satisfies~\ref{prop:size}--\ref{prop:regular} is called an \emph{$(\eps,d)$-regular partition} of $G$.
Given this partition, we define the \emph{$(\eps,d)$-reduced graph} $R$ for $G$, that is, the graph on vertex set $[t]$, in which $ij$ is an edge if and only if $(V_i,V_j)$ is an $\eps$-regular pair in $G'$ and has density at least $d$.

\subsection{Squares of paths in randomly perturbed graphs}

We also need the following result that allows us to find multiple copies of the square of a short path in randomly perturbed graphs.

\begin{lemma}
	\label{lem:sublinear_square_paths}
	For all integers $k\ge 2$ and $t \ge 1$, there exist $C, \gamma>0$ such that the following holds for any $0 \le m \le \gamma n$ and any $n$-vertex graph $G$ of minimum degree $\delta(G) \ge m$ and maximum degree $\Delta(G) \le \gamma n$.
	For $p \ge C (\log n)^{1/(2k-3)} n^{-(k-1)/(2k-3)}$, a.a.s.~the perturbed graph $G \cup G(n,p)$ contains $t m + t$ pairwise vertex-disjoint copies of the square of a path on $k+1$ vertices.
\end{lemma}

The case $k=2$ and $t=1$ was already covered in~\cite[Theorem~2.4]{triangle_paper}.
The general proof is similar and is given in Appendix~\ref{sec:appendix}.

\section{Proof overview}
\label{sec:overview}

In this section we will sketch the proof of Theorem~\ref{thm:main_small} and discuss a more 
general stability version of it.
As already explained in the introduction, the cases $\alpha=0$ and $\alpha \ge \frac12$ follow from known results.
Moreover, the case $\alpha = \frac 12$ will follow from the monotonicity of the perturbed threshold, once we will have determined the perturbed threshold in the range $\alpha < \frac12$.
Therefore from now on, we can fix an integer $k \ge 2$ and assume $\alpha \in \left[\frac{1}{k+1},\frac1k\right)$. 

We start by discussing the idea of our embedding strategy and explaining how this leads to the threshold probabilities given in Theorem~\ref{thm:main_small}.
We then turn to the arguments for the lower bound on $\hat{p}_\alpha$ and afterwards split the upper bound into two theorems (Theorems~\ref{thm:small_extremal} and~\ref{thm:small_non-extremal}) depending on the structure of the dense graph $G_\alpha$: an extremal case and a non-extremal one.
Here, Theorem~\ref{thm:small_non-extremal} provides a stability version of Theorem~\ref{thm:main_small}.

\subsection{Strategy.}
We recall that, given the square of the path $v_1, v_2,\ldots, v_k$, we define its end-tuples as $(v_2, v_1)$ and $(v_{k-1}, v_k)$.
Let $G$ be any $n$-vertex graph with minimum degree $\alpha n$ and $\alpha \in \left[ \frac{1}{k+1},\frac{1}{k} \right)$.
Our goal is to find the square of a Hamilton cycle $C_n^2$ in the perturbed graph $G \cup G(n,p)$ and therefore we will use a \emph{decomposition} of $E(C_n^2)$ into `deterministic edges' (to be embedded to $G$) and `random edges' (to be embedded to $G(n,p)$).
To get the square of a path we would like vertex disjoint copies $F_1,\dots,F_t$ of $P_k^2$ in the random graph $G(n,p)$ such that the following holds.
For each $i=1,\dots,t-1$, if we denote by $(y_i, x_i)$ and $(u_i, w_i)$ the end-tuples of $F_i$, then $w_ix_{i+1}$ is also an edge in $G(n,p)$.
Moreover, there exist $t-1$ additional vertices $v_1,\dots,v_{t-1}$ such that, for $i=1,\dots,t-1$, all four edges $v_iu_{i}, v_iw_{i}, v_ix_{i+1}, v_iy_{i+1}$ are edges in $G$.
This gives the square of a path on $t(k+1)-1$ vertices with edges from $G\cup G(n,p)$ (c.f.~Figure~\ref{fig_HC2_0}).

\begin{figure}[htpb]
	
	\begin{tikzpicture}[scale=0.65,
		point/.style ={circle,draw=none,fill=#1,
			inner sep=0pt, minimum width=0.2cm,node contents={}}
		]
		%Square of Hamilton path
		\node (w3) at (11,{2-sqrt(3)})   [point=black, label=below:$w_3$];
		\node (u3) at (10,2)   [point=black, label=above:$u_3$];
		\node (x3) at (9,{2-sqrt(3)})   [point=black, label=below:$y_3$];
		\node (y3) at (8,2)   [point=black, label=above:$x_3$];
		\node (v2) at (7,{2-sqrt(3)})   [point=red, label=below:$v_2$];
		\node (w2) at (6,2)   [point=black, label=above:$w_2$];
		\node (u2) at (5,{2-sqrt(3)})   [point=black, label=below:$u_2$];
		\node (x2) at (4,2)   [point=black, label=above:$y_2$];
		\node (y2) at (3,{2-sqrt(3)})   [point=black, label=below:$x_2$];	
		\node (v1) at (2,2)   [point=red, label=above:$v_1$];
		\node (w1) at (1,{2-sqrt(3)})   [point=black, label=below:$w_1$];
		\node (u1) at (0,2)   [point=black, label=above:$u_1$];
		\node (x1) at (-1,{2-sqrt(3)})   [point=black, label=below:$y_1$];
		\node (y1) at (-2,2)   [point=black, label=above:$x_1$];
		
		\draw[blue,dashed,line width=1pt]
		(w3)--(u3)--(x3)--(y3)--(u3)
		(x3)--(w3)
		(y3)--(w2)--(u2)--(x2)--(y2)--(u2)
		(x2)--(w2)
		(y2)--(w1)--(u1)--(x1)--(y1)--(u1)
		(x1)--(w1);
		\draw[black,line width=1pt]
		(x3)--(v2)--(y3)
		(w2)--(v2)--(u2)
		(x2)--(v1)--(y2)
		(w1)--(v1)--(u1);
	\end{tikzpicture}
	\captionsetup{font=footnotesize}
	\captionof{figure}{The square of a path with end-tuples $(y_1,x_1)$ and $(u_3,w_3)$ with our decomposition into random (dashed blue) and deterministic (black) edges for $k=4$ and $t=3$.}
	\label{fig_HC2_0}
\end{figure}
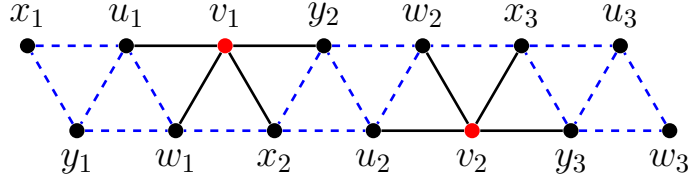

Note that by requiring the edge $w_{t}x_{1}$ from $G(n,p)$ and adding another vertex $v_t$ joined to $u_{t},w_{t},x_{1},y_{1}$ in $G$, we get the square of a cycle on $t(k+1)$ vertices. 
However for divisibility reasons this may not cover all $n$ vertices. 
Hence in order to find the square of a Hamilton cycle and for some additional technical reasons, our proof(s) will allow some of $F_1,\dots,F_t$ to be the squares of paths of different lengths. 

This decomposition of $C_n^2$ already justifies the probabilities that appear in Theorem~\ref{thm:main_small}:
indeed, $n^{-(k-1)/(2k-3)}$ is the threshold in $G(n,p)$ for a linear number of copies of $P_k^2$ (this follows from Lemma~\ref{lem:embedding}), while $n^{-(k-1)/(2k-3)} (\log n)^{1/(2k-3)}$ is the threshold in $G(n,p)$ for the existence of a $P_k^2$-factor (this follows from Theorem~\ref{thm:JKV}).

\subsection{Lower bounds}
\label{subsec:lower_bounds}

For any $\alpha \in (0,1/2)$, let $H_\alpha$ be the complete bipartite $n$-vertex graph with vertex classes $A$ and $B$ of size $\alpha n$ and $(1-\alpha) n$, respectively.
The following two propositions provide lower bounds on $\hat{p}_\alpha$ for $\alpha \in (\tfrac{1}{k+1},\tfrac 1k)$ and $\alpha = \tfrac{1}{k+1}$.
Their proofs are standard and we only give a sketch.

\begin{proposition}
	\label{prop:lower_bound_non_extr}
	Let $\alpha \in (\tfrac{1}{k+1},\tfrac{1}{k})$.
	Then there exists $0 < c < 1$ such that $H_\alpha \cup G(n,p)$ a.a.s.~does not contain a copy of $C_n^2$, provided $p \le c n^{-(k-1)/(2k-3)}$.
\end{proposition}

\begin{proof}[Sketch of the proof of Proposition~\ref{prop:lower_bound_non_extr}]
	Let $1/(k+1) < \alpha < 1/k$, take $0 < c < (1/k - \alpha)/2$ and observe that in $B$ there are a.a.s.~at most $2cn$ copies of $P_k^2$ (by an upper tail bound on the distribution of small subgraphs~\cite{vu_2001}).
	Assume for a contradiction that there is an embedding of $C_n^2$ into $H_\alpha \cup G(n,p)$.
	Then $H_\alpha \cup G(n,p)$ must contain $n/k$ vertex-disjoint copies of $P_k^2$ and only at most $|A|=\alpha n$ of them have a vertex in $A$.
	Therefore there must be at least $n/k - \alpha n > 2cn$ copies of $P_k^2$ in $B$, where the inequality follows from the choice of $c$.
	This gives a contradiction.
\end{proof}

When $\alpha=\tfrac{1}{k+1}$, an additional $\log$-factor is needed, which essentially comes from the fact that we need a $P_{k}^2$-factor in $G(n,p)$.
For example, suppose $k=2$ and consider $H_{1/3} \cup G(n,p)$.
Observe that $H_{1/3}$ contains an independent set $B$ of size $2n/3$ and, in the range of $p$ of our interest, a.a.s.~$G(n,p)[B]$ contains few triangles.
Therefore, in order for a copy of $C_n^2$ to appear in $H_{1/3} \cup G(n,p)$, there has to be a perfect matching in $G(n,p)[B]$.

\begin{proposition}
	\label{prop:lower_bound_extr}
	Let $\alpha = \tfrac{1}{k+1}$.
	Then $H_\alpha \cup G(n,p)$ a.a.s.~does not contain a copy of $C_n^2$, provided $p \le \frac{1}{4k} n^{-(k-1)/(2k-3)} (\log n)^{1/(2k-3)}$.
\end{proposition}

\begin{proof}[Sketch of the proof of Proposition~\ref{prop:lower_bound_extr}]
	Let $c=1/(4k)$. 
	A.a.s.~(by the first moment method) $B$ contains at most $n^{1-2c}$ copies of $P_{k+1}^2$ and a.a.s.~(by the second moment method) at least $n^{1-c}$ vertices from $B$ are not contained in any copy of $P_k^2$ within $B$.
	Assume for a contradiction that there is an embedding of $C_n^2$ into $H_\alpha \cup G(n,p)$.
	Then $H_\alpha \cup G(n,p)$ must contain a $P_{k+1}^2$-factor.
	Since $|B|=k|A|$, the average size of the intersection of a copy of $P_{k+1}^2$ in such a factor with $B$ would be $k$.
	However, given the restrictions above, it is not possible to cover the vertices of $G(n,p)[B]$ with a family of squares of paths whose average size is $k$.
	This gives a contradiction.
\end{proof}

\subsection{Proof of Theorem~\ref{thm:main_small} and stability}
\label{subsec:stability}

As mentioned before, for the proof of Theorem~\ref{thm:main_small} we distinguish between an extremal case, when the deterministic graph $G$ is \emph{close} to $H_{1/(k+1)}$, i.e. the complete bipartite $n$-vertex graph with vertex classes of size $\frac{1}{k+1}n$ and $\frac{k}{k+1} n$, and a non-extremal case.
It turns out that the additional $(\log n)^{1/(2k-3)}$-term in the perturbed threshold at $\alpha=\frac{1}{k+1}$ is only necessary in the extremal case.
The next definition formalises what we mean by \emph{close}.

\begin{definition}[$(\alpha,\beta)$-stable]
	\label{def:stability}
	For $\alpha, \beta>0$ we say that an $n$-vertex graph $G$ is \emph{$(\alpha,\beta)$-stable} if there exists a partition of $V(G)$ into two sets $A$ and $B$ of sizes $|A|=(\alpha \pm \beta) n$ and $|B|=(1-\alpha \pm \beta) n$ such that the minimum degree of the bipartite subgraph $G[A,B]$ is at least $\tfrac 14 \alpha n$, all but at most $\beta n$ vertices from $A$ have degree at least $|B|-\beta n$ into $B$, all but at most $\beta n$ vertices from $B$ have degree at least $|A|-\beta n$ into $A$ and $G[B]$ contains at most $\beta n^2$ edges.
\end{definition}

The following theorem treats the extremal case of Theorem~\ref{thm:main_small}.

\begin{theorem}[Extremal Theorem]
	\label{thm:small_extremal}
	For every $k \ge 2$ there exist $\beta>0$ and $C>0$ such that the following holds.
	Let $G$ be any $n$-vertex graph with minimum degree at least $\tfrac{1}{k+1} n$ that is $(\tfrac{1}{k+1}, \beta)$-stable.
	Then $G \cup G(n,p)$ a.a.s.~contains the square of a Hamilton cycle, provided that $p\ge C n^{-(k-1)/(2k-3)}  (\log n)^{1/(2k-3)}$.  
\end{theorem}	

When the graph $G$ is not stable, the $(\log n)^{1/(2k-3)}$-term is not needed and the next theorem provides a stability version of Theorem~\ref{thm:main_small}.

\begin{theorem}[Stability Theorem]
	\label{thm:small_non-extremal}
	For every $k \ge 2$ and every $0 < \beta < \tfrac{1}{6k}$, there exist $\gamma>0$ and $C>0$ such that the following holds.
	Let $G$ be any $n$-vertex graph with minimum degree at least $(\tfrac{1}{k+1} - \gamma) n$ that is not $(\tfrac{1}{k+1}, \beta)$-stable.
	Then $G \cup G(n,p)$ a.a.s.~contains the square of a Hamilton cycle, provided that $p\ge C n^{-(k-1)/(2k-3)}$.  
\end{theorem}

We sketch the ideas for the proof of these two theorems in the following two subsections.
Together with the lower bounds provided in Propositions~\ref{prop:lower_bound_non_extr} and~\ref{prop:lower_bound_extr}, Theorem~\ref{thm:small_extremal} and~\ref{thm:small_non-extremal} imply Theorem~\ref{thm:main_small} for $\alpha \in [\tfrac{1}{k+1},\tfrac1k)$ with $k \ge 2$.
Indeed, fix $k \ge 2$, let $\beta_0>0$ be given by Theorem~\ref{thm:small_extremal} and take $\beta=\min\{\beta_0, \frac{1}{6k}\}$.
Let $G$ be any $n$-vertex graph with minimum degree at least $\tfrac{1}{k+1} n$.
The upper bound on the threshold for $\alpha = \tfrac{1}{k+1}$ then follows from Theorem~\ref{thm:small_extremal} or~\ref{thm:small_non-extremal}, depending on whether $G$ is $(\tfrac{1}{k+1},\beta)$-stable or not.
When $\alpha \in (\tfrac{1}{k+1},\tfrac1k)$, then, with $\beta=\tfrac15 \left(\tfrac{1}{k}-\alpha \right)>0$,
the upper bound on the threshold follows from  Theorem~\ref{thm:small_non-extremal}, since an $n$-vertex graph with minimum degree $\alpha n$ is not $(\tfrac{1}{k+1},\beta)$-stable.

\subsection{Overview of the proof of Theorem~\ref{thm:small_extremal}}
\label{sec:sketch_extremal}
For the extremal case, suppose that $G$ is an $n$-vertex $(\alpha,\beta)$-stable graph with $\alpha=\frac{1}{k+1}$ and let $p\ge C (\log n)^{1/(2k-3)} n^{-(k-1)/(2k-3)}$.
The definition of stability (Definition~\ref{def:stability}) gives a partition $A \cup B$ of $V(G)$ in which the size of $B$ is roughly $k$ times the size of $A$, the minimum degree of $G[A,B]$ is at least $\alpha n/4$ and all but few vertices of $A$ ($B$, respectively) are adjacent to all but few vertices of $B$ ($A$, respectively).
Our proof will follow three steps.

In the first step, we would like to embed copies $F_i$ of $P_k^2$ into $B$ and vertices $v_i$ into $A$, following the decomposition described above.
However, this is only possible if $|B|=k|A|$ and, therefore, we first embed squares of short paths of different lengths to ensure this divisibility condition in the remainder.
We find a family $\cF_1$ of copies of squares of paths with end-tuples in $B$, such that after removing the vertices $V_1 = \bigcup_{F \in \cF_1}V(F)$ ,
we are left with two sets $A_1=A \setminus V_1$ and $B_1=B \setminus V_1$ with $|B_1| = k(|A_1| - |\cF_1|)$.
Note that we construct the family $\cF_1$ in such a way that we have $|B_1|=k(|A_1|-|\cF_1|)$ rather than $|B_1|=k|A_1|$, because each square path in $\cF_1$ still needs to be connected into the final square of a Hamilton cycle, and for each of these connections we shall use one vertex in $A_1$.
The precise way we find $\cF_1$ depends on the sizes of $A$ and $B$, but in all cases we will ensure that the vertices in the end-tuples of each $F \in \cF_1$ are neighbours of all but few vertices of $A$.
When $|B| > \tfrac {k}{k+1} n$, the family $\cF_1$ consists of copies of $P_{k+1}^2$ inside of $B$. Its existence is guaranteed by Lemma~\ref{lem:sublinear_square_paths} using the minimum degree of $G[B]$.
When $|B| \le \tfrac {k}{k+1} n$, the family $\cF_1$ consists of copies of $P_{2k+3}^2$ with all vertices in $B$, except the $(k+1)$-st and $(k+3)$-rd, that belong to $A$.

Our second step is to cover the vertices in $A_1$ and $B_1$ that do not have a high degree to the other part.
For this we will find another family $\cF_2$ of copies of squares of paths with end-tuples in $B$.
For any vertex $v$ in $A_1$ with small degree into $B_1$, we find a copy of $P_{2k+1}^2$ with $v$ being the $(k+1)$-st vertex and all the remaining vertices belonging to~$B_1$.
Similarly, for any vertex $v$ in $B_1$ with small degree into $A_1$, we find a copy of $P_{3k+2}^2$ consisting of three copies of $P_k^2$ in $B$ connected by edges and two vertices from $A_1$, where $v$ is in the middle copy of $P_k^2$.
We need that $v$ is in the middle copy, because then we can again ensure that the end-tuples of each $F \in \cF_2$ see all but few vertices of $A$.
Moreover, with $V_2 = \bigcup_{F \in \cF_2}V(F)$ and $A_2=A_1 \setminus V_2$ and $B_2=B_1 \setminus V_2$, we have $|B_2| = k (|A_2| -|\cF_1| - |\cF_2|)$.

At this point, each of the vertices in $A_2$ ($B_2$, respectively) is adjacent to all but few vertices of $B_2$ ($A_2$, respectively) and we kept the divisibility condition intact. In the third step, we let $\cF_3$ be pairwise disjoint random copies of $P_k^2$ covering $B_2$, which is possible by Theorem~\ref{thm:JKV} with our $p$ and because $|B_2|$ is divisible by~$k$.

We let $\cF = \cF_1 \cup \cF_2 \cup \cF_3$ and, for each $F \in \cF$, denote its end-tuples by $(y_F,x_F)$ and $(u_F,w_F)$.
We now reveal additional edges of $G(n,p)$ and encode their presence in an auxiliary directed graph $\cT$ on vertex set $\cF$ as follows.
There is a directed edge $(F,F')$ if and only if the edge $w_Fx_{F'}$ appears in $G(n,p)$.
It is easy to see that all directed edges in $\cT$ are revealed with probability $p$ independently of all others and, therefore, we can find a directed Hamilton cycle $\overrightarrow{C}$ in $\cT$ with Theorem~\ref{lem:directed_ham}.
We finally match to each edge $(F,F')$ of $\overrightarrow{C}$ a vertex $v \in A_2$ such that $u_F,w_F,x_{F'},y_{F'}$ are all neighbours of $v$ in the graph $G$.
Owing to the high minimum degree conditions, that this is possible easily follows from Hall's matching theorem.
Thus we get the square of a Hamilton cycle, as wanted.

\subsection{Overview of the proof of Theorem~\ref{thm:small_non-extremal}.}
\label{sec:sketch_nonextremal}

Assume that $G$ is not $(\frac{1}{k+1},\beta)$-stable and let $p\ge C n^{-(k-1)/(2k-3)}$.
Then we apply the regularity lemma to $G$ and we use the following lemma on its reduced graph $R$.

\begin{lemma}[Lemma~4.4 in~\cite{triangle_paper}]
	\label{lem:stable_cluster}
	For any integer $k \ge 2$ and $0<\beta<\tfrac{1}{12}$, there exists $d>0$ such that the following holds for any $0<\eps<d/4$, $4 \beta \le \alpha \le \tfrac 13$ and $t \ge \tfrac{10}{d}$ .
	Let $G$ be an $n$ vertex graph with minimum degree $\delta(G) \ge (\alpha -\tfrac12 d)n$ that is not $(\alpha,\beta)$-stable and let $R$ be the $(\eps,d)$-reduced graph for some $(\eps,d)$-regular partition $V_0,\dots,V_t$ of $G$.
	Then $R$ contains a matching $M$ of size $(\alpha+2kd)t$.
\end{lemma}

In fact,~\cite[Lemma 4.4]{triangle_paper} is weaker as, under the same assumptions, it gives a matching of size $(\alpha+2d)t$, but Lemma~\ref{lem:stable_cluster} follows from a straightforward adaptation of its proof, which in turn builds on ideas from~\cite{BMS_cover}.
With Lemma~\ref{lem:stable_cluster}, it is not hard to show that the reduced graph $R$ can be vertex-partitioned into copies of stars $K_{1,k}$ and matching edges $K_{1,1}$, such that the number of stars is not too large.
For each such star and matching edge, we would like to cover their clusters with the square of a Hamilton path and then connect these square paths, while covering the exceptional vertices, in order to get the square of a Hamilton cycle.
However, since we want to avoid the additional $\log$-term in the edge density of the random graph, for this strategy to work in the randomly perturbed graph, we need that in each star the centre cluster is larger than the other clusters.
Moreover, to ensure that we can connect the squares of Hamilton paths, we need to setup some connections between the stars and matching edges.

Therefore, we first remove some vertices from the leaf clusters of each star to make it unbalanced and ensure that all pairs are super-regular.
Then we label the stars and matching edges arbitrarily as $Q_1,\dots,Q_s$, and for $i=1,\dots,s$ find a copy $F_i$ of $P_6^2$ to connect $Q_i$ and $Q_{i+1}$ (where we identify the index $s+1$ with $1$ and the index $0$ with $s$).
More precisely, for each star $Q_i$, one end-tuple of $F_{i-1}$ and one of $F_i$ belong to the centre cluster of $Q_i$.
Moreover, for each matching edge $Q_i$, each of its two clusters contains exactly one end-tuple, one from $F_{i-1}$ and the other from $F_i$.
We will refer to these squares of paths as the connecting (squares of) paths. 
Let $V_0$ be the set of vertices no longer contained in any of the stars or matching edges.
We cover each $v \in V_0$ by appending $v$ to one of the connecting paths.
Here we use that any vertex $v \in V_0$ has degree at least $(\tfrac{1}{k+1}-\alpha)n$ and, as we do not have too many stars, the vertex $v$ has also many neighbours in clusters which are not centres of stars.
This is crucial, because it allows us to ensure that the relations between the sizes of the sets in any star are suitable for an application of the following lemma, which is the main technical ingredient in the proof.

\begin{lemma}[Star Cover Lemma]
	\label{lem:multipartite2}
	For any $k\ge 2$ and any $0<\delta'\le d<1$ there exist $\delta_0,\delta_1,\eps > 0$ with $\delta'\ge \delta_0>2\delta_1>\eps$ and $C>0$ such that the following holds.
	Let $V,U_1,\dots,U_k$ be pairwise disjoint sets such that $|V|=n+4$, $(1-\delta_0) n \le |U_1|=\dots=|U_k| \le (1-\delta_1)n$ and $n-|U_1| \equiv -1 \pmod{3k-1}$.
	Suppose that $(V,U_i)$ are $(\varepsilon,d)$-super-regular pairs with respect to a graph $G$ for each $i=1,\dots,k$, and $(x,x')$ and $(y,y')$ are two tuples from $V$ such that both tuples have $\tfrac12 d^2n$ common neighbours in $U_i$ for each $i=1,\dots,k$.
	Furthermore, let $G(V,p)$ and $G(U_1,\dots,U_k,p)$ be random graphs with $p\ge C n^{-(k-1)/(2k-3)}$.
	
	Then a.a.s.~there exists the square of a Hamilton path in $G[V,U_1,\dots,U_k] \cup G(V,p) \cup G(U_1,\dots,U_k,p) \cup \{xx',yy'\}$ covering $V \cup U_1 \cup \dots \cup U_k$ with end-tuples $(x,x')$ and $(y,y')$.
\end{lemma}

This implies that for any star $Q_i$ we can connect the end-tuples of $F_{i-1}$ and $F_i$ which belong to the centre cluster of $Q_i$, while covering all vertices in the clusters of $Q_i$.
We emphasise again that, to avoid $\log$-terms, it is crucial that the centre cluster is larger than the leaf clusters.
Similarly for the matching edges we use the following lemma.
\begin{lemma}[Pair Cover Lemma]
	\label{lem:bipartite}
	For any $0<d<1$ there exist $\eps>0$ and $C>0$ such that the following holds for sets $U,V$ with $|V|=n$ and $\tfrac34 n \le |U| \le n$.
	Let $(U,V)$ be an $(\varepsilon,d)$-super-regular pair with respect to a graph $G$ and $(x,x')$ and $(y,y')$ be tuples from $V$ and $U$, respectively, such that the vertices from the tuples have $\tfrac12 d^2n$ common neighbours in $U$ and $V$, respectively.
	If $G(U,p)$, $G(V,p)$ are random graphs with $p\ge C n^{-1}$, then a.a.s.~there exists the square of a Hamilton path in $G[U,V] \cup G(U,p) \cup G(V,p) \cup \{xx',yy'\}$ covering $U \cup V$ with end-tuples $(x,x')$ and $(y,y')$.
\end{lemma}
Together this gives the square of a Hamilton cycle in $G \cup G(n,p)$.
We will give the proof of Lemma~\ref{lem:multipartite2} and~\ref{lem:bipartite} in Section~\ref{sec:auxiliary}.

\section{Proof of Theorem~\ref{thm:small_extremal}}
\label{sec:extremal}

\begin{proof}[Proof of Theorem~\ref{thm:small_extremal}]
	
	Given an integer $k \ge 2$, we let $C_2$ and $0<\gamma \le 1$ be given by Lemma~\ref{lem:sublinear_square_paths} for input $k$ and $t=k+1$.
	We let $C_4$ be given by Theorem~\ref{lem:directed_ham} and set $C=4 C_2+ 8 k C_4 + 8$.
	Next, we let $0<\beta \le \tfrac{1}{100 k^3} \gamma$.
	Given $n$, let $0 \le a \le k$ be such that $n = (k+1) \lfloor \frac{n}{k+1} \rfloor + a$ and $p \ge C (\log n)^{1/(2k-3)} n^{-(k-1)(2k-3)}$.
	We reveal a subgraph of $G(n,p)$ in four rounds $G_i \sim G(n,\tfrac 14 p)$ for $i=1,2,3,4$.
	By Lemma~\ref{lem:embedding} and a union bound over all graphs on at most $4k$ vertices, we can a.a.s.~assume that 
	\begin{align}
		\label{random_graph_G_1}
		G_1 \text{ contains a copy of the graph $F$ in any vertex-set of size at least $\beta n$,} 
	\end{align}
	where $F$ is any graph on at most $4k$ vertices with $m_1(F) \le m_1(P_k^2) = \tfrac{2k-3}{k-1}$.

	Let $G$ be an $n$-vertex graph with minimum degree at least $\frac{1}{k+1} n$ that is $(\frac{1}{k+1},\beta)$-stable.
	Then there exists a partition of $V(G)$ into $A$ and $B$ that satisfies Definition~\ref{def:stability}.
	As outlined in Section~\ref{sec:sketch_extremal} our proof will consist of three steps.
	We will successively build parts of the square of a Hamilton cycle, first covering some vertices to balance the partition, then covering vertices of low degree to the other side and, then, covering the remaining vertices.
	Finally we will connect these parts into the square of a Hamilton cycle.
	
	\textbf{Balancing the partition.}
	Our goal is to find a family $\cF_1$ of pairwise disjoint copies of squares of paths with end-tuples in $B$, such that each end-tuple has at least $|A|-8k^2\beta n$ common neighbours in $A$ and the size of the set $V_1 = \bigcup_{F \in \cF_1}V(F)$ is smaller than $3k^2 \beta n$.
	Moreover, we will guarantee that after removing the vertices of $V_1$ we are left with two sets 
	\begin{align}
		\label{eq:F_1}
		A_1=A \setminus V_1 \quad \text{and} \quad  B_1=B \setminus V_1 \quad \text{such that} \quad  |B_1| = k(|A_1| - |\cF_1|)\, .
	\end{align}
	
	We distinguish between the cases $|A| = \lfloor \frac{n}{k+1} \rfloor + m$ with $1 \le m \le \beta n$ and $|A| = \lfloor \frac{n}{k+1} \rfloor - m$ with $0 \le m \le \beta n$.
	Suppose first that $|A| = \lfloor \frac{n}{k+1} \rfloor + m$ for some $1 \le m \le \beta n$.
	In this case we want $|\cF_1|=m$ and the family $\cF_1$ will consist of $m-1$ copies of $P_{3k+2}^2$ and one copy of $P_{3k+2+a}^2$, such that for each of these $m$ copies, exactly three vertices are in $A$, both end-tuples are in $B$ and each end-tuple has at least $|A|-2\beta n$ common neighbours in $A$.
	
	We can do this greedily in $G \cup G_1$.
	Assume that during this process we have to find a copy of $P_{3k+2}^2$ or $P_{3k+2+a}^2$, i.e. a copy of $P_{3k+b}^2$ for some $2 \le b \le k+2$, such that the above conditions are satisfied.
	There are three vertices $v_1, v_2$ and $v_3$ in $P_{3k+b}^2$, such that none of them is in an end-tuple of $P_{3k+b}^2$, they do not induce a triangle in $P_{3k+b}^2$ and the subgraph $H=P_{3k+b}^2 \setminus \{v_1,v_2,v_3\}$ satisfies $m_1(H) \le m_1(P_k^2)$ (see Figure~\ref{fig_connection_subgraph}).
	We can avoid an induced triangle because there are at least $3k+b - 4 \ge 4$ vertices to choose from that are not in end-tuples, and guarantee the bound on the density because, when distributing the three vertices evenly, the longest square of a path in $H$ has at most $\lceil (3k+b-3)/4 \rceil \le k$ vertices.
	We remark that we can always ask $\{v_1,v_2,v_3\}$ to be an independent set in $P_{3k+b}^2$ when $k>2$.

	\begin{figure}[htpb]
		
		\begin{tikzpicture}[scale=0.65,
			point/.style ={circle,draw=none,fill=#1,
				inner sep=0pt, minimum width=0.2cm,node contents={}}
			]
			%Square of Hamilton path
			\node (v14) at (11,{2-sqrt(3)})   [point=black];
			\node (v13) at (10,2)   [point=black];
			\node (v12) at (9,{2-sqrt(3)})   [point=black];
			\node (v11) at (8,2)   [point=red, label=above:$v_3$];
			\node (v10) at (7,{2-sqrt(3)})   [point=black];
			\node (v9) at (6,2)   [point=black];
			\node (v8) at (5,{2-sqrt(3)})   [point=red, label=below:$v_2$];
			\node (v7) at (4,2)   [point=black];
			\node (v6) at (3,{2-sqrt(3)})   [point=black];	
			\node (v5) at (2,2)   [point=black];
			\node (v4) at (1,{2-sqrt(3)})   [point=red, label=below:$v_1$];
			\node (v3) at (0,2)   [point=black];
			\node (v2) at (-1,{2-sqrt(3)})   [point=black];
			\node (v1) at (-2,2)   [point=black];
			
			\draw[blue,dashed,line width=1pt]
			(v3)--(v2)--(v1)--(v3)--(v5)--(v6)--(v7)--(v5)
			(v7)--(v9)--(v10)--(v12)--(v13)--(v14)--(v12);
			\draw[black,line width=1pt]
			(v2)--(v4)--(v3)
			(v5)--(v4)--(v6)--(v8)--(v7)
			(v9)--(v8)--(v10)--(v11)--(v9)
			(v12)--(v11)--(v13);
		\end{tikzpicture}
		\captionsetup{font=footnotesize}
		\captionof{figure}{Subgraph $H$ (dashed blue) obtained from $P_{3k+b}^2$ after removing three vertices $v_1,v_2,v_3$ (red) for $b=k+2$ and $k=3$. The $1$-density of $H$ is the same as $P_k^2$.}
		\label{fig_connection_subgraph}
	\end{figure}
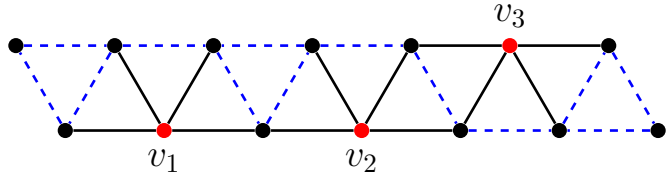
	
	Then we find a copy of $P_{3k+b}^2$ by embedding the vertices $v_1, v_2,v_3$ in $A$ and the other vertices in $B$ in the following way.
	Let $A' \subseteq A$ be the set of vertices of $A$ that have degree at least $|B|-\beta n$ into $B$ and have not yet been covered, and observe that $|A'| \ge |A| - \beta n - 3|\cF_1| \ge \beta n$.
	Then, since $m_1(P_3)=1  \le m_1(P_k^2)$ and given~\eqref{random_graph_G_1}, the random graph $G_1[A']$ contains a path on three vertices $u_1,u_2,u_3$.
	Next, let $B' \subseteq B$ be the set of vertices of $B$ that have degree at least $|A|-\beta n$ into $A$, are common neighbours of $u_1,u_2,u_3$ and have not yet been covered, and observe that $|B'| \ge |B| - \beta n - 3 \beta n - (3k+k+2-3)|\cF_1| \ge \beta n$.
	Since $m_1(H)=m_1(P_k^2)$, using again~\eqref{random_graph_G_1}, the random graph $G_1[B']$ contains a copy of $H$, that together with the vertices $u_1,u_2,u_3$ and some edges from $G$ gives a copy of $P_{3k+b}^2$.
	In particular, since each vertex in the end-tuples of $P_{3k+b}^2$ is embedded in $B'$, then, by definition of $B'$, each end-tuple has at least $|A|-2\beta n$ common neighbours in $A$.
	In view of~\eqref{eq:F_1}, we get that $|A_1| = \lfloor \frac{n}{k+1} \rfloor - 2m$,
	\[ |B_1|=n- |A| - (3k-1) m - a = k \left\lfloor \frac{n}{k+1} \right\rfloor - 3k m = k (|A_1| - |\cF_1|) \, , \]
	and $|V_1| = (m-1)(3k+2)+3k+2+a \le (4k+3)m \le 3k^2 \beta n$.
	
	Now suppose that $|A| = \lfloor \frac{n}{k+1} \rfloor - m$ for some $0 \le m \le \beta n$.
	In this case, the family $\cF_1$ will consist of some copies of $P_{k+1}^2$ and some copies of $P_{2k+1}^2$, such that, for each copy, all vertices are in $B$ and each end-tuple has at least $|A|-8k^2\beta n$ common neighbours in $A$; in particular, we do no touch the set $A$.
	We start from $\cF_1 = \emptyset$ and let $B^\star=\{v \in B: \deg(v,B) \ge 4 k^2 \beta n \}$ be the set of vertices in $B$ with high degree to $B$ in $G$.
	Since by Definition~\ref{def:stability} we have $e(G[B]) \le \beta n^2$, then $|B^\star| \le \frac{n}{2k^2}$.
	Moreover if $v \in B \setminus B^\star$, then $\deg(v,A) \ge \delta(G) - \deg(v,B) \ge \frac{n}{k+1}-4 k^2\beta n \ge |A| - 4 k^2\beta n$.
	With $m_0 = \max \{ m - |B^\star|,0 \}$, we have $\delta(G[B \setminus B^\star]) \ge m_0$, as 
	\[\delta(G[B \setminus B^\star]) \ge \frac{n}{k+1}-|A|-|B^\star| = \frac{n}{k+1}- \left\lfloor \frac{n}{k+1} \right\rfloor + m -|B^\star| \ge m-|B^\star|\, .\]
	Moreover, by definition of $B^\star$, we have $\Delta(G[B \setminus B^\star]) \le 4 k^2 \beta n \le \gamma |B \setminus B^\star|$, where the last inequality follows from the choice of $\beta$ and $|B \setminus B^\star| \ge (\frac{k}{k+1} - \frac{1}{2k^2})n$.
	Therefore we can use Lemma~\ref{lem:sublinear_square_paths} with parameters $k$ and $t=k+1$ and we a.a.s.~find $(k+1)m_0+a$ pairwise disjoint copies of $P_{k+1}^2$ in $(G \cup G_2)[B \setminus B^\star]$, which we add to $\cF_1$.
	All the vertices of such copies belong to $B \setminus B^\star$ and thus in particular each end-tuple has at least $|A|-8 k^2\beta n$ common neighbours in $A$.
	
	Next we want to find $m-m_0 \le \min\{|B^\star|,m\}$ copies of $P_{2k+1}^2$ in $B$ disjoint from any graph already in $\cF_1$.
	First observe that the graph $H$ obtained by taking the disjoint union of two copies of $P_{k}^2$ with the addition of the edge between the last vertex of the first copy and the first vertex of the second copy satisfies $m_1(H) \le m_1(P_k^2)$. 
	The graph $H$ will be embedded in $G_1$ and we will turn that into an embedding of $P_{2k+1}^2$, by adding a vertex and four edges of $G$.
	We can again do this greedily in $G \cup G_1$.
	First we remove vertices from $B^\star$ such that $|B^\star| = m-m_0 \le m \le \beta n$.
	Then we let $B'$ be the set of vertices in $B$ with less than $|A|-\beta n$ neighbours in $A$ and note that $|B'| \le \beta n$.
	We then pick a vertex $w$ from $B^\star$ not yet covered and denote by $N_w$ the set of neighbours of $w$ in $B \setminus (B^\star \cup B')$ in the graph $G$, that have not yet been covered.
	Then
	\[ |N_w| \ge 4 k^2 \beta n - |B'| - |B^\star| - ((k+1)m_0+a)(k+1) - (m-m_0)(2k+1-1) \ge \beta n\, , \]
	where we use that any $P_{2k+1}^2$ from $\cF_1$ has exactly one vertex in $B^\star$.
	Therefore, since $m_1(H) \le m_1(P_k^2)$ and using~\eqref{random_graph_G_1}, the random graph $G_1[N_w]$ contains a copy of this $H$, that together with $w$ and four edges from $G$, gives a copy of $P_{2k+1}^2$ as desired. 
	The end-tuples of this copy belong to $B \setminus B'$ and thus each of them has at least $|A|-2\beta n$ common neighbours in $A$.
	Once this is done, in view of~\eqref{eq:F_1}, we indeed get $|A_1| = |A| = \lfloor \frac{1}{k+1}n \rfloor - m$,
	\begin{align*}
		|B_1|=n- |A| - (k+1) ((k+1)m_0+a) - (2k+1)(m-m_0) \\
		= k \left( \left\lfloor \frac{n}{k+1} \right\rfloor - a - 2m - k m_0 \right) = k (|A_1| - |\cF_1|) \, ,
	\end{align*}
	and $|V_1| = ((k+1)m_0+a)(k+1) + (m-m_0)(2k+1) \le (k+1)^2m+k(k+1) \le 3k^2 \beta n$.
	This finishes the first step of our proof.
	
	\textbf{Covering low degree vertices.} 
	In this step the goal is to find a family $\cF_2$ of pairwise disjoint copies of $P_{2k+1}^2$ and $P_{3k+2}^2$ that cover all vertices from $A$ (respectively $B$) that do not have high degree to $B$ (respectively $A$) and such that, for each copy, each end-tuple is in $B$ and has at least $|A|-16 k^2 \beta n$ common neighbours in $A$.
	We will do this such that after removing $V_2 = \bigcup_{F \in \cF_2}V(F)$ we are left with two sets 
	\begin{align}
		\label{eq:F_2}
		A_2=A_1 \setminus V_2 \quad \text{and} \quad B_2=B_1 \setminus V_2 \quad \text{such that} \quad |B_2| = k(|A_2| - |\cF_1| - |\cF_2|)\, .
	\end{align}
	We let $A'=\{v \in A_1: \deg(v,B) \le |B| - \beta n\}$, $B'=\{v \in B_1: \deg(v,A) \le |A| - 8k^2\beta n\}$ and note that $|A'| \le \beta n$ and $|B'| \le \beta n$.
	Let $\cF_2 = \emptyset$.
	We start by taking care of the vertices in $A'$ and we cover each of them with a copy of $P_{2k+1}^2$ with all other vertices in $B_1 \setminus B'$.
	For $u \in A'$, let $N_u$ be the set of neighbours of $u$ in $B_1 \setminus B'$ in $G$ that are not yet covered by any graph in $\cF_2$ and observe that
	\[|N_u| \ge \deg(u,B) - |B \setminus B_1| - |B'| - 2k |\cF_2| \ge \tfrac{n}{4(k+1)} - 3k^2 \beta n - \beta n - 2k \beta n \ge \beta n \, ,\]
	where we used that $\deg(u,B) \ge \tfrac{n}{4(k+1)}$ as $G$ is $(\tfrac{1}{k+1},\beta)$-stable (see Definition~\ref{def:stability}).
	Note that the graph $H$ obtained by taking the union of two copies of $P_{k}^2$ with the addition of an edge between two end-vertices satisfies $m_1(H) \le m_1(P_k^2)$. 
	Therefore, using~\eqref{random_graph_G_1}, the random graph $G_1[N_u]$ contains a copy of $H$ and together with $u$ and four edges of $G$, this gives the desired copy of $P_{2k+1}^2$.
	Both end-tuples of this copy of $P_{2k+1}^2$ belong to $B_1 \setminus B'$ and, thus, have at least $|A|-16 k^2 \beta n$ common neighbours in $A$.
	We add this copy of $P_{2k+1}^2$ to $\cF_2$ and we continue until we cover all vertices of $A'$.
	
	Now we cover each vertex from $B'$ with a copy of $P_{3k+2}^2$, where each copy uses one vertex from $B'$, two vertices from $A_1 \setminus A'$ and the other $3k-1$ vertices from $B_1 \setminus B'$.
	Let $w \in B'$ and $u_1,u_2 \in A_1 \setminus A'$ be vertices not yet covered.
	We denote by $N_w$ the subset of $B_1 \setminus B'$ which contains the common neighbours of $w,u_1,u_2$ in $G$ that are not yet covered.
	Observe that the definitions of $A'$ and $B'$ give
	\begin{align*}
		|N_w| 	&\ge (\delta(G) - \deg(w,A)) - (|B|-\deg(u_1,B)) - (|B|-\deg(u_2,B)) + \\
		& \hspace{7cm} - |B \setminus B_1| - 3k |B'| - 2k |A'| \\
		&\ge \tfrac{n}{k+1} - (|A| - 8k^2 \beta n) - \beta n - \beta n - 3k^2 \beta n - 3k \beta n - 2k \beta n \ge \beta n \, .
	\end{align*}
	Similarly as above, using~\eqref{random_graph_G_1}, the random graph $G_1[N_w]$ contains a copy of a graph $H$ on $3k-1$ vertices, that together with $w,u_1,u_2$ and some edges from $G$, gives a copy of $P_{3k+2}^2$.
	The end-tuples of the copy of $P_{3k+2}^2$ belong to $B_1 \setminus B'$ and, thus, have at least $|A|-16 k^2 \beta n$ common neighbours in $A$.
	We add the copy of $P_{3k+2}^2$ to $\cF_2$ and repeat until all of $B'$ is covered. 
	We then get~\eqref{eq:F_2}, because of~\eqref{eq:F_1} and since for each graph added to $\cF_2$ the ratio of vertices removed from $A_1$ and $B_1$ is one to $2k$ or two to $3k$.
	Moreover, we have $\deg(v,B_2) \ge |B_2|-\beta n$ for $v \in A_2$, $\deg(v,A_2) \ge |A_2| - 8k^2 \beta n$ for $v \in B_2$ and $|V_2| \le 6k \beta n$, which implies $|A_2| \ge |A| - |V_1| - |V_2| \ge \tfrac{n}{2(k+1)}$.
	
	\textbf{Covering everything and connecting.}
	In this step we first cover $B_2$ with copies of $P_k^2$.
	Then, using the uncovered vertices in $A_2$, we connect all the copies of squares of paths found so far, to get the square of a Hamilton cycle.
	Observe that after the cleaning steps, $k$ divides $|B_2|$ by~\eqref{eq:F_2} and thus Theorem~\ref{thm:JKV} implies that a.a.s~the random graph $G_3[B_2]$ has a $P_k^2$-factor.
	We denote the family of such copies of $P_k^2$ by $\cF_3$ and observe that~\eqref{eq:F_2} implies $|\cF_3| = |A_2|-|\cF_1|-|\cF_2|$.
	
	We let $\cF= \cF_1 \cup \cF_2 \cup \cF_3$ be the family of all the squares of paths that we have constructed and, for each $F \in \cF$, denote the end-tuples of $F$ by $(y_F,x_F)$ and $(u_F,w_F)$.
	Recall that the end-tuples are defined such that $x_F$ and $w_F$ is the first resp.~last vertex of the path.
	Note that by construction, each pair $x_F,y_F$ and $u_F,w_F$ has at least $|A_2|- 16 k^2 \beta n$ common neighbours in $A_2$.
	We now reveal the edges of $G_4$ and construct an auxiliary directed graph $\cT$ on vertex set $\cF$ as follows.
	Given any two $F, F' \in \cF$, there is a directed edge $(F,F')$ if and only if the edge $w_Fx_{F'}$ appears in $G_4$.
	Since all directed edges in $\cT$ are revealed with probability $\tfrac 14 p$ independently of all others, $\cT$ is distributed as $\overrightarrow{G}(|\cF|,\tfrac 14 p)$.
	Then, as $|\cF| \ge \tfrac{1}{2k} n$ and $\tfrac 14 p \ge C_4 \tfrac{\log |\cF|}{|\cF|}$, there a.a.s.~is a directed Hamilton cycle $\overrightarrow{C}$ in $\cT$ by Theorem~\ref{lem:directed_ham}.
	
	In order to get the desired square of a Hamilton path, it remains to match the edges $(F,F')$ of $\overrightarrow{C}$ to the vertices $v \in A_2$ such that $u_F,w_F,x_{F'},y_{F'}$ are all neighbours of $v$ in the graph $G$.
	Observe indeed that $|E(\overrightarrow{C})|=|\cF_1|+|\cF_2|+|\cF_3|=|A_2|$, the cycle $\overrightarrow{C}$ gives a cyclic ordering of the squares of paths in $\cF$, and between consecutive copies $(F,F')$ we have the edge $w_Fx_{F'}$ (by definition of the graph $\cT$) and the matching will give a unique vertex $v$ incident to $u_F,w_F,x_{F'}$ and $y_{F'}$.
	For that we define the following auxiliary bipartite graph $\cB$ with classes $E(\overrightarrow{C})$ and $A_2$.
	There is an edge between $(F,F') \in E(\overrightarrow{C})$ and $v \in A_2$ if and only if $u_F,w_F,x_{F'},y_{F'}$ are all neighbours of $v$ in the graph $G$.
	The existence of a perfect matching in $\cB$ easily follows from Hall's condition, as the minimum degree of $\cB$ is large: by the choice of $\beta$ and the lower bound on $|A_2|$, the degree of $(F,F') \in E(\overrightarrow{C})$ in $\cB$ is at least $|A_2| - 32k^2 \beta n \ge \tfrac 12 |A_2|$ and the degree of $v \in A_2$ in $\cB$ is at least $|A_2| - 4 \beta n \ge \tfrac 12 |A_2|$.
	This finishes the proof.
\end{proof}

\section{Proof of Theorem~\ref{thm:small_non-extremal}}
\label{sec:non-extremal}

\begin{proof}[Proof of Theorem~\ref{thm:small_non-extremal}]
	Let $k \ge 2$ and $0 < \beta < \tfrac{1}{6k}$.
	We obtain $d>0$ from Lemma~\ref{lem:stable_cluster} with input $k$ and $\beta$, and let $\gamma=\tfrac{1}{k-1} d$ and $0<\delta' < 2^{-12} (k-1)^{-2} d^{2k^2}$.
	Next we obtain $\delta_0,\delta,\eps'>0$ and $C_1$ from the Star Cover Lemma (Lemma~\ref{lem:multipartite2}) with input $k$, $\delta'$ and $\tfrac12 d$ such that $\delta' \ge \delta_0 > 2\delta > \eps'$ ($\delta$ plays the role of $\delta_1$ in the lemma).
	We additionally assume that $\eps'$ is small enough for the Pair Cover Lemma (Lemma~\ref{lem:bipartite}) with input $\tfrac{1}{2} d$ and also obtain $C_2$ from this.
	Then we let $0<\eps<\tfrac{1}{8}\eps'$.
	The constant dependencies can be summarised as follows:
	\[ \eps \ll \eps' \ll \delta \ll \delta_0 \ll \delta' \ll d \ll \beta < \frac{1}{6k} \, . \]
	We now apply Lemma~\ref{lem:reg} with input $\eps$ and $t_0=\tfrac{1}{10}d$, and get $T$.
	Further, let the following parameters be given by Lemma~\ref{lem:paths2}:
	$C_3$ for input $3$ (in place of $s$), $2$ (in place of $k$) and $\delta (kT)^{-6}$ (in place of $\eta$);
	$C_4$ for input $2k$, $k$ and $\delta (kT)^{-2k^2}$;
	$C_5$ for input $4$, $2$ and $\delta (kT)^{-8}$; and
	$C_6$ for input $1$, $2$ and $\delta (kT)^{-2}$.
	Then let $C$ be large enough such that with $p \ge Cn^{-(k-1)/(2k-3)}$ the random graph $G(n,p)$ contains the union $\bigcup_{i=1}^6 G_i$, where $G_1 \sim G(n,C_1 \cdot 2(k-1)T n^{-(k-1)/(2k-3)})$, $G_2 \sim G(n,C_2 \cdot 2(k-1)T n^{-1})$, $G_3 \sim G(n,C_3 n^{-1})$, $G_4 \sim G(n,C_4 n^{-(k-1)/(2k-3)})$, $G_5 \sim G(n,C_5 n^{-1})$ and $G_6 \sim G(n,C_6 n^{-1})$.
	
	Let $G$ be an $n$-vertex graph with vertex set $V$ and minimum degree $\delta(G) \ge (\tfrac{1}{k+1} - \gamma)n$ that is not $(\tfrac{1}{k+1},\beta)$-stable.
	We apply the Regularity Lemma (Lemma~\ref{lem:reg}) to $G$ and get a subgraph $G'$ of $G$, a constant $t$ with $3 < t +1 \le T$ and an $(\eps,d)$-regular partition $V'_0,\dots,V'_t$ of $V$, satisfying~\ref{prop:size}--\ref{prop:regular}.
	Consider the $(\eps,d)$-reduced graph $R$ for $G$ and observe $\delta(R) \ge (\tfrac{1}{k+1} - 2d)t$, as, otherwise, in $G'$ there would be a vertex of degree at most $(\tfrac{1}{k+1} - 2d) t \tfrac{n}{t} + \eps n < (\tfrac{1}{k+1} - \gamma) n - (d + \eps)n$ in contradiction to~\ref{prop:degree}.
	As outlined in Section~\ref{sec:sketch_nonextremal}, the proof will consist of four steps.
	We will cover the reduced graph with copies of stars isomorphic to $K_{1,k}$ and $K_{1,1}$, connect those stars with the squares of short paths, cover the exceptional vertices and, finally, cover the whole graph with the square of a Hamilton path.
	
	\textbf{Covering $R$ with stars.}
	We start by covering the vertices of $R$ with vertex-disjoint stars, each with at most $k$ leaves; then we will turn this into a cover with copies of stars isomorphic to $K_{1,k}$ and $K_{1,1}$.
	We let $M_1$ be a largest matching in $R$ and, with Lemma~\ref{lem:stable_cluster}, we get that $|M_1| \ge \left(\tfrac{1}{k+1} + 2kd\right) t$ since $G$ is not $(\tfrac{1}{k+1},\beta)$-stable.
	By the maximality of $M_1$, the remaining vertices $X_1 =V(R) \setminus V(M_1)$ form an independent set in $R$; moreover only one endpoint of each edge in $M_1$ can be adjacent to more than one vertex from $X_1$ and the endpoints of each edge in $M_1$ cannot have different neighbours in $X_1$.
	For each $i=2,\dots,k$, we greedily pick a maximal matching $M_i$ between $X_{i-1}$ and $V(M_1)$ and we set $X_i = X_{i-1} \setminus V(M_i)$.
	Observe that, using the properties coming from the maximality of $M_1$ outlined above, the matching $M_i$ covers at least $\min \{ |X_{i-1}| , \delta(R)  \}$ vertices of $X_{i-1}$.
	Since $2|M_1| + (k-1) \delta(R) \ge t $, we have $X_k= \emptyset$ and the union $\bigcup_{i=1}^k M_i$ covers all vertices of $R$ with stars, each isomorphic to one of $K_{1,1}, \dots, K_{1,k}$.
	Moreover note that the number of $K_{1,k}$ is at most $|M_k| \le t- 2|M_1| - (k-2) \delta(R) \le \left(\tfrac{1}{k+1} - 2kd\right)t$.
	
	For simplicity we only want to work with stars isomorphic to $K_{1,k}$ and $K_{1,1}$.
	To obtain this, we split each cluster $V$ arbitrarily into $k-1$ parts $V^1, \dots, V^{k-1}$ of the same size, where we move at most $t (k-1)$ vertices to $V'_0$ for divisibility reasons.
	Note that from any $(\eps,d)$-regular pair we get $(k-1)^2$ pairs that are $(k \eps,d-\eps)$-regular.
	We denote this new partition by $V_0=V_0',V_1,\dots,V_{t'}$ with $t'=t(k-1)$ and denote the reduced graph for this partition by $R'$.
	We now show that we can cover $R'$ with copies of stars isomorphic to $K_{1,k}$ and $K_{1,1}$.
	Any copy of $K_{1,1}$ or $K_{1,k}$ in $R$ immediately gives $k-1$ copies of $K_{1,1}$ or $K_{1,k}$ in $R'$, respectively.
	Moreover given any copy $V,U_1,\dots,U_i$ of $K_{1,i}$ in $R$, with $V$ being the center cluster and $2 \le i \le k-1$, we find $i-1$ copies of $K_{1,k}$ and $k-i$ copies of $K_{1,1}$ in $R'$ in the following way.
	For each $j=1,\dots,i-1$, the clusters $V^j,U_j^1,\dots,U_j^{k-1},U_i^j$ give a copy of $K_{1,k}$ in $R'$ and, for each $j=i,\dots,k-1$, the clusters $V^j,U_i^j$ give a copy of $K_{1,1}$ in $R'$.
	Therefore, we have covered the vertices of $R'$ with a collection $\cK$ of copies of $K_{1,k}$ and $K_{1,1}$.
	We remark that we can still upper bound the number of copies of $K_{1,k}$ in $\cK$ as follows.
	Since each copy of $K_{1,i}$ in the original cover gives $i-1$ copies of $K_{1,k}$ in $\cK$, we get the largest number of copies of $K_{1,k}$ in $\cK$ when the total number of stars in the original cover is minimal. 
	The original cover of $R$ had at most $(\tfrac{1}{k+1} -2kd)t$ copies of $K_{1,k}$ and the remaining $t-(k+1)(\tfrac{1}{k+1} -2kd)t=2k(k+1)dt$ vertices can give at most $\frac{2k(k+1)dt}{k}=2(k+1)dt$ copies of $K_{1,k-1}$.
	Therefore the collection $\cK$ has at most $(k-1)(\tfrac{1}{k+1} -2kd)t+(k-2)2(k+1)dt = (\tfrac{1}{k+1}-\tfrac{4d}{k-1})t'$ copies of $K_{1,k}$.
	We let $n_0=|V_1|=\lfloor |V_1'|/(k-1) \rfloor$ be the size of the clusters in $R'$ and observe that $(1-2\eps) n/t' \le n_0 \le n/t'$.
	
	For convenience we relabel the clusters as follows.
	We let $\cI \subseteq [t']$ be the set of indices of those clusters of $R'$ that are the centre cluster in a copy of $K_{1,k}$ in $\cK$ and, for $i \in \cI$, we denote by $U_{i,1},\dots,U_{i,k}$ the clusters of $R'$ that, together with $V_i$, create a copy of $K_{1,k}$ in $\cK$.
	Then we let $\cJ \subseteq [t']$ be any set of indices with the following property: each index in $\cJ$ corresponds to a cluster of a copy of $K_{1,1}$ in $\cK$ and, for each copy of $K_{1,1}$ in $\cK$, exactly one of its clusters has its index in $\cJ$.
	Moreover, for each $i \in \cJ$, we let $U_{i,1}$ be the cluster of $R'$ that, together with $V_i$, creates a copy of $K_{1,1}$ in $\cK$.
	
	We would like to apply the Star Cover Lemma (Lemma~\ref{lem:multipartite2}) and the Pair Cover Lemma (Lemma~\ref{lem:bipartite}) to the clusters corresponding to the copies of $K_{1,k}$ and $K_{1,1}$, respectively.
	However we first need to make each regular pair super-regular and unbalance some of the clusters to allow an application of Lemma~\ref{lem:multipartite2}.
	For that we arbitrarily move $\delta n_0+4(1-\delta)$ vertices from each cluster $U_{i,j}$ with $i \in \cI$ and $j=1,\dots,k$ to $V_0$.
	Observe this ensures that $|U_{i,k}| \le (1-\delta)(|V_i|-4)$.
	Next we repeatedly use the Super-regular Pair Lemma (Lemma~\ref{lem:superreg}) and move at most $k^2\eps n_0$ vertices from each cluster to $V_0$, to ensure that all edges of $R'$ within a copy of $K_{1,k}$ or $K_{1,1}$ from $\cK$ are $(2k \eps,d-4k\eps)$-super-regular.
	When doing this we can ensure that for each $i \in \cI$ all clusters $V_i$ have the same size and that for $j=1,\dots,k$ all clusters $U_{i,j}$ have the same size, except the cluster $U_{i,1}$ that has two more vertices than other $U_{i,j}$.
	Additionally, we can ensure that for each $i \in \cJ$, the clusters $V_i$ and $U_{i,1}$ have the same size.
	Moreover, by moving only at most $3k-2$ additional vertices from each $U_{i,j}$ to $V_0$ (which do not harm the bounds above), we can ensure that for $i \in \cI$ we have $|V_i| - 4 - |U_{i,k}| \equiv -1 \pmod{3k-1}$, again in view of a later application of the Star Cover Lemma (Lemma~\ref{lem:multipartite2}).
	Note that at this point we have $|V_0| \le \eps n + t(k-1) + t' \delta n_0 + t' k^2 \eps n_0 + t' 3k \le 2 \delta n$.

	\textbf{Connecting the stars.}
	In this step, we fix an arbitrary cyclic ordering of the copies of $K_{1,k}$ and $K_{1,1}$, and we connect each consecutive pair using the square of a short path.
	For the rest of the proof, we will refer to these (squares of) short paths as the connecting paths.
	We first explain the connection between two copies of $K_{1,k}$ and assume without loss of generality that $1,2 \in \cI$; when at least one of the copies is $K_{1,1}$, the connection is similar and will be explained later.
	We use the square of a path on six vertices with end-tuples within $V_1$ and $V_2$, such that, again in view of the Star Cover Lemma (Lemma~\ref{lem:multipartite2}), the end-tuples have many common neighbours into the other clusters $U_{1,j}$ and $U_{2,j}$ for $j=1,\dots,k$, respectively.
	Recall that both $U_{1,1}$ and $U_{2,1}$ contain two more vertices than other leaf clusters, one of which will be used for this connection.
	
	\begin{figure}[htpb]
		
		\includegraphics{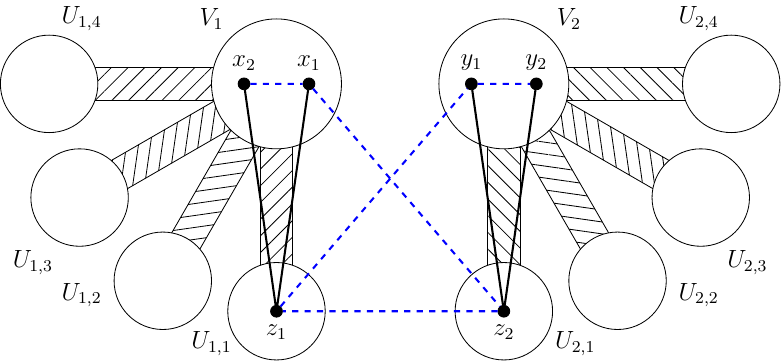}		
		
		\captionsetup{font=footnotesize}
		\captionof{figure}{Construction of the square of a path with end-tuples $(x_1,x_2)$ and $(y_1,y_2)$ that connects two copies of $K_{1,k}$ in the cluster graph $R'$. The dashed blue edges come from the random graph and the black edges from the deterministic graph.}
		\label{fig_connection}
	\end{figure}
	
	We want to find $x_1, x_2 \in V_1$, $z_2 \in U_{2,1}$, $z_1 \in U_{1,1}$ and $y_1, y_2 \in V_2$ such that the following holds (see Figure~\ref{fig_connection}):
	\begin{enumerate}[label=(A\arabic*)]
		\item \label{a} the tuples $(x_1,x_2)$ and $(y_1,y_2)$ have at least $\tfrac 34 d^2 n_0$ common neighbours in each of $U_{1,1},\dots,U_{1,k}$ and $U_{2,1},\dots,U_{2,k}$ in the graph $G$, respectively;
		\item \label{b} $x_1z_1$, $x_2z_1$, $y_1z_2$ and $y_2z_2$ are edges of $G$;
		\item \label{c} $x_2,x_1,z_2,z_1,y_1,y_2$ is a path in $G_3$.
	\end{enumerate}
	For that we use Lemma~\ref{lem:paths2} on the following collection $H$ of tuples.
	We pick subsets $X$ in $V_1$, $Y$ in $V_2$, $Z_1$ in $U_{1,1}$ and $Z_{2}$ in $U_{1,2}$, all of size $n_0/3$, and we let $H$ be the set of those tuples $( x_2,x_1,z_2,z_1,y_1,y_2) \in X \times X \times Z_2 \times Z_1 \times Y \times Y$ which satisfy the properties~\ref{a} and~\ref{b}.
	Note that $H$ contains enough tuples for an application of Lemma~\ref{lem:paths2}.
	Indeed, with Lemma~\ref{lem:MDL}, we get that all but at most $2 k \eps n_0$ vertices $x_1 \in X$ have degree at least $(d-6k\eps )|Z_1|$ and then, fixing any such $x_1$, all but at most $4k^2 \eps n_0$ vertices $x_2 \in X$ have at least $\tfrac 34 d^2 n_0$ common neighbours in $U_{1,1},\dots,U_{1,k}$ with $x_1$ and at least $\tfrac 14 d^2 n_0$ common neighbours in $Z_{1}$ with $x_1$ with respect to $G$ and similarly for $Y$ with $U_{2,1},\dots,U_{2,k}$ and $Z_{2}$.
	\label{key}Therefore, $H$ has size at least $d^4 2^{-12} n_0^6 \ge \delta (kT)^{-6} n^6$.
	Now we reveal the edges of $G_3[X \cup Z_{2} \cup Z_{1} \cup Y]$ and with Lemma~\ref{lem:paths2} and given the choice of $C_3$, we a.a.s.~find a tuple $(x_2,x_1,z_2,z_1,y_1,y_2) \in H$ satisfying the property~\ref{c} as well.
	This will give the desired connecting square of a path. 
	We then remove the vertices $z_1$ and $z_2$ from $U_{1,1}$ and $U_{2,1}$, respectively.
	
	Given an arbitrary cycling ordering of the indices from $\cI$ and $\cJ$, we use this construction to connect all neighbouring pairs $(i,j)$.
	If $i,j \in \cI$ we proceed as described above for the pair $(1,2)$; if $i \in \cJ$, we let $U_{i,1}$ take the role of all $U_{1,1},\dots,U_{1,k}$; and if $j \in \cJ$, we let $U_{j,1}$ take the role of $V_2$ and $V_j$ the role of all $U_{2,1},\dots,U_{2,k}$.
	As we pick sets $X$, $Y$, $Z_1$ and $Z_2$ of size $n_0/3$ and each cluster is involved in at most two distinct connections, we can choose disjoint sets for each connection and avoid clashes.
	Therefore each edge of $G_3$ is revealed at most once and, as there are at most $t'$ connections, we a.a.s.~get all the desired edges of $G_3$.
	Observe that the choices done so far are needed for a later application of the Star Cover Lemma (Lemma~\ref{lem:multipartite2}) and the Pair Cover Lemma (Lemma~\ref{lem:bipartite}).
	Indeed, for $i \in \cI$, there are two end-tuples of connecting paths in $V_i$ and, for $i \in \cJ$, there is one end-tuple of a connecting path in each of $V_i$ and $U_{i,1}$; moreover, for $i \in \cI$, since $U_{i,1}$ is involved in exactly two connections, during this construction we removed exactly two vertices from $U_{i,1}$ and now all $U_{i,j}$ have the same size for $j=1,\dots,k$.

	\textbf{Covering $V_0$.}
	In the next step we cover all vertices of $V_0$ by extending the connecting paths that we have already constructed and we recall that $|V_0| \le 2 \delta n$.
	It is crucial for the rest of the argument (in particular for the applications of the Star Cover Lemma and the Pair Cover Lemma) that the conditions on the relation between the sizes of the clusters are still satisfied and that the end-tuples remain in the same clusters, i.e. if we extend the square of a path with one end-tuple in a cluster $V_i$, then the extended path needs to have the new end-tuple in the same cluster $V_i$.
	We will again make use of Lemma~\ref{lem:paths2}, with a suitable collection of tuples $H$.
	Before giving a precise description, we refer to Figure~\ref{fig_absorb1b}
	and illustrate the extension when $k=3$ and we want to cover a vertex $v \in V_0$ by extending the square of a connecting path with end-tuple $(x_1, x_2)$ in the centre cluster $V_i$ of a copy of $K_{1,3}$. 
	We will find $24$ additional vertices as drawn in Figure~\ref{fig_absorb1b}, where we stress the following conditions.
	The vertices $v_{3,1}, \dots ,v_{3,6}$ are all neighbours of $v$ in $G$ (which will be guaranteed by the minimum degree condition), the blue edges are random edges (which will be guaranteed by Lemma~\ref{lem:paths2}), while the black edges are from the graph $G$ (which will be guaranteed by regularity). 
	Thus we extend the connecting path with one end-tuple $(x_1,x_2)$ in $V_i$ to a longer connecting path with end-tuple $(x_1',x_2')$ still in $V_i$, by appending the square of a path containing $v$ and with end-tuples $(x_2,x_1)$ and $(x_1',x_2')$ (recall this is consistent with the way we defined end-tuples).
	Moreover, since six new vertices have been covered from each cluster, the relation between their sizes still holds.

	\begin{figure}[htpb]
		\begin{minipage}{.85\textwidth}
			%\centering
			\includegraphics{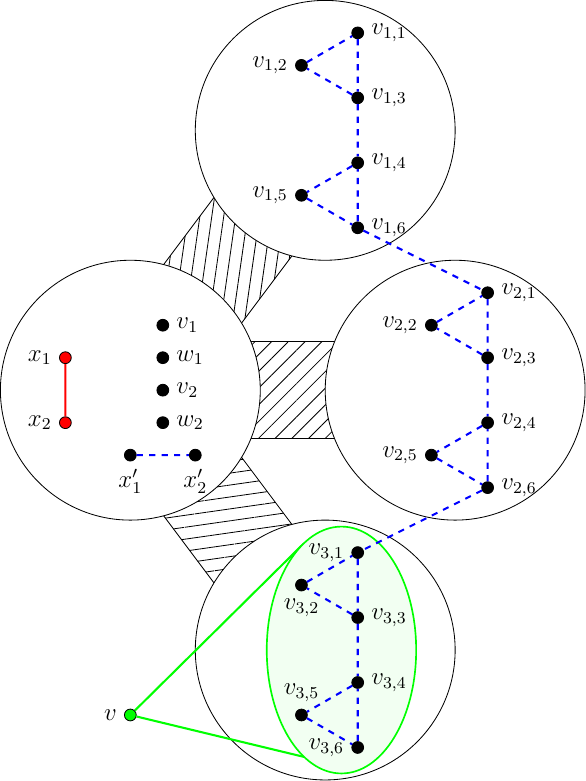}
		\end{minipage}
		\begin{minipage}{.1\textwidth}
			%\centering
			\hspace{-2cm}
			\includegraphics{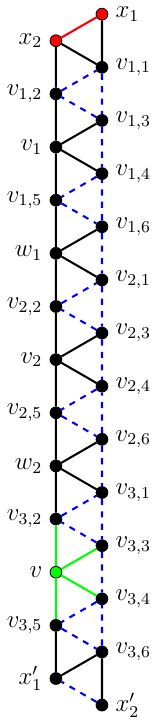}
		\end{minipage}
		\captionsetup{font=footnotesize}
		\captionof{figure}{Covering one vertex $v \in V_0$ by the square of a path for a copy of $K_{1,k}$ in $\cK$ %in the case 
			for $k=3$.
			We keep the position of the end-tuples in $V_i$ and the sizes of the clusters balanced.
			The edges within the clusters (dashed blue) come from the random graph and the edges between the clusters (black) come from regularity.}
		\label{fig_absorb1b}
	\end{figure}
	
	We will now give the details of these constructions and we start by defining the collections of tuples we will use for the applications of Lemma~\ref{lem:paths2}.
	For $i \in \cI$, we let $H_{1,i}$ be the set of those tuples in $\prod_{j=1}^k (U_{i,j}^k \times U_{i,j}^k)$ such that the $2k^2$ vertices in each tuple have at least $\tfrac{1}{2} d^{2k^2} n_0$ common neighbours in $V_i$ in $G$.
	Then we let $H_1=\bigcup_{i \in \cI} H_{1,i}$.
	Similarly, for $j \in \cJ$, we let $H_{2,j}$ be the set of those tuples in $V_j^8 \cup U_{j,1}^8$ such that the $8$ vertices in each tuple have at least $\tfrac 12 d^8 n_0$ common neighbours in the other set in $G$.
	Then we let $H_2=\bigcup_{j \in \cJ} H_{2,j}$.
	Moreover, for $i \in \cI$, we let $H_{3,i}$ be the set of those tuples in $V_i^2$ such that the $2$ vertices in each tuple have at least $\tfrac 34 d^2 n_0$ common neighbours in each of the sets $U_{i,j}$ for $j=1,\dots,k$.
	For $j \in \cJ$, we let $H_{3,j}$ be the set of those tuples in $V_j^2 \cup U_{j,1}^2$ such that the $2$ vertices in each tuple have at least $\tfrac 34 d^2 n_0$ common neighbours in the other set in $G$.
	Then we let $H_3 = \bigcup_{i \in \cI} H_{3,i} \cup \bigcup_{j \in \cJ} H_{3,j}$.
	With the constants specified in the beginning of the proof for obtaining $C_4$, $C_5$ and $C_6$, we apply Lemma~\ref{lem:paths2} to $H_1$ with the random graph $G_4$, to $H_2$ with $G_5$ and to $H_3$ with $G_6$.
	Given the choice of the constant $C$ done at the beginning, we can a.a.s.~assume that 
	\begin{align}
		\label{random_graph_Lemma2.5} 
		\text{$G_4$, $G_5$ and $G_6$ are all in the good event of Lemma~\ref{lem:paths2} for the application above.}
	\end{align}
	
	Now we explain how we cover the vertices of $V_0$ and we show that suitable subsets of $H_1$, $H_2$ and $H_3$ are large enough for Lemma~\ref{lem:paths2}, i.e. larger than $\delta (kT)^{-sk} n^{sk}$ where $sk =2k^2,8,2$, respectively.
	
	Given a vertex $v \in V_0$, we insist that the neighbours of $v$ that we use to cover $v$ do not come from any of the centre clusters $V_i$ with $i \in \cI$, because in this case we could not ensure to use the same number of vertices from each cluster in the copy of $K_{1,k}$ and we would unbalance the star, creating issues for a later application of the Star Cover Lemma (Lemma~\ref{lem:multipartite2}).
	Observe that, as we have $|V_0 \cup (\bigcup_{i \in \cI} V_i)| \le 2 \delta n + (\tfrac{1}{k+1}-4kd)t' n_0 \le (\tfrac{1}{k+1}-\gamma) n - \tfrac{2d}{k-1}n$, every vertex in $G$ has at least $\tfrac{2d}{k-1}n$ neighbours outside of $V_0 \cup (\bigcup_{i \in \cI} V_i)$ in the graph $G$.
	Moreover, for each vertex $v \in V_0$, we will use at most $\max \{2k(k+1),22 \} \le 6k^2$ vertices outside of $V_0$.
	During the process of covering $V_0$, we let $V^\ast \subseteq V_0$ be the set of vertices of $V_0$ already covered and $W$ be the set of vertices outside of $V_0$ that we already used to cover $V^\ast$; at the beginning $V^\ast=\emptyset$ and $W=\emptyset$. 
	Note that we have $|W| \le 6 k^2 |V^\ast| \le 12 k^2 \delta n$.
	Then we let $\cT \subseteq \cI \cup \cJ$ be the set of those indices $i \in \cI \cup \cJ$ such that $U_{i,1}$ intersects $W$ in at least $\sqrt{\delta} n_0$ vertices and note that the bound on $|W|$ implies that $|\cT| \le 12 k^2\sqrt{\delta} t'$.
	We recall that at the beginning of the process for each $i \in \cI$ all the clusters $U_{i,j}$ had the same size for $j=1,\dots,k$.
	Similarly, for each $i \in \cJ$ the clusters $V_i$ and $U_{i,1}$ had the same size as well.
	Since throughout the process we always cover the same number of vertices in each cluster of a copy of $K_{1,1}$ or $K_{1,k}$, if $i \not\in \cT$, then $W$ intersects each cluster of the copy corresponding to the index $i$ in less then $\sqrt{\delta} n_0$ vertices.
	Moreover, notice that as each $v \in V_0$ has at least $\tfrac{2d}{k-1}n$ neighbours outside of $V_0 \cup (\bigcup_{i \in \cI} V_i)$, there are at least $\frac{d}{k-1}t'$ clusters that are not the centre cluster of a copy of $K_{1,k}$ and such that $v$ has at least $\tfrac{d}{k-1}n_0$ neighbours in it. 
	Therefore, for every $v \in V_0 \setminus V^\ast$ there exists $i(v) \in (\cI \cup \cJ) \setminus \cT$ and some $j(v)$ such that $v$ has at least $\tfrac{d}{k-1}n_0$ neighbours in $U_{i(v),j(v)} \setminus W$ with respect to $G$.

	Fix any $v \in V_0 \setminus V^\ast$ and let $i=i(v)$.
	We start discussing the case $i \in \cI$; the case $i \in \cJ$ is conceptually simpler and will be treated afterwards.
	We recall that, since we removed the vertices $z$ used for the connecting paths, all the leaf clusters $U_{i,j}$ for $j=1,\dots,k$ have the same size and thus, without loss of generality, we can assume $j(v)=k$.
	Let $(x_1,x_2)$ and $(y_1,y_2)$ be the end-tuples in $V_i$ of the connecting paths found in the previous step and recall that $x_1$ and $x_2$ have at least $\tfrac 34 d^2 n_0$ common neighbours in $U_{i,j}$ for each $j=1,\dots,k$.
	We will extend the square of the connecting path with end-tuple $(x_1,x_2)$, without using neither $y_1$ nor $y_2$.
	Moreover, we will make sure that the new end-tuple $(x_1',x_2')$ belongs to $V_i$ and that $x_1'$ and $x_2'$ have at least $\tfrac 34 d^2 n_0$ common neighbours in $U_{i,j}$ for each $j=1,\dots,k$.
	
	Let $Z_1$ be the set $(N_G(x_1,U_{i,1}) \cap N_G(x_2,U_{i,1})) \setminus W$, let $Z_j$ be the set $U_{i,j} \setminus W$ for $j=2,\dots,k-1$ and let $Z_k$ be the set $N_G(v,U_{i,k}) \setminus W$.
	Observe that $|Z_1| \ge \tfrac 12 d^2n_0$ since $x_1$ and $x_2$ have at least $\tfrac 34 d^2 n_0$ common neighbours in $U_{i,1}$, $|Z_j| \ge (1-\sqrt{\delta}) n_0 \ge d n_0$ for $j=2,\dots,k-1$ and $|Z_k| \ge \tfrac{d}{k-1} n_0$ by the choice of $i=i(v)$.
	Using the regularity properties, analogously as above, it is easy to see that there are at least $\delta (kT)^{-2k^2} n^{2k^2}$ tuples
	in $\prod_{j=1}^k (Z_j^k \times Z_j^k)$ such that the $2k^2$ vertices in each tuple have at least $\tfrac{1}{2} d^{2k^2} n_0$ common neighbours in $V_i$ in $G$.
	These tuples are in $H_{1,i} \subseteq H_1$ as well and thus~\eqref{random_graph_Lemma2.5} guarantees that we find one of such tuples $(v_{j,j'}: 1 \le j \le k, 1 \le j' \le 2k)$ in $\prod_{j=1}^k (Z_j^k \times Z_j^k)$ where $v_{j,j'} \in Z_j$, such that in $G_4$ we have the square of a path on $v_{j,1},\dots,v_{j,k}$, the square of a path on $v_{j,k+1},\dots,v_{j,2k}$ and the edges $v_{j,k}v_{j,k+1}$ and $v_{j,2k}v_{j+1,1}$.
	Note that we applied Lemma~\ref{lem:paths2} with $s=2k$ and $U_{i,j}$ as $W_{2j-1}$ and $W_{2j}$ for $j=1,\dots,k$, but then relabelled the vertices to achieve the outcome above.
	
	Let $Z \subseteq V_i \setminus (W \cup \{y_1,y_2\})$ be the common neighbourhood of the vertices $v_{j,j'}$ with $1 \le j \le k$ and $1 \le j' \le 2k$ in $V_i$, and observe $|Z| \ge \tfrac{1}{4} d^{2k^2} n_0$.
	Again using regularity, there are at least $\delta (k T)^{-2} n^2$ tuples in $Z^2 \cap H_3$ and, therefore,~\eqref{random_graph_Lemma2.5} guarantees that there is a tuple $(x_1',x_2') \in Z^2 \cap H_3$ such that $x_1'x_2'$ is an edge of $G_6$ and $x_1'$ and $x_2'$ have at least $\tfrac 34 d^2 n_0$ common neighbours in each $U_{i,j}$ for $j=1,\dots,k$.
	We then greedily pick additional $2(k-1)$ vertices $v_1,w_1,\dots,v_{k-1},w_{k-1}$ in $Z$ and we claim that
	\begin{multline*}
		x_1,x_2,v_{1,1},\dots,v_{1,k},v_1,v_{1,k+1},\dots,v_{1,2k},w_1,\\
		v_{2,1},\dots,v_{2,k},v_2,\dots,w_{k-1},v_{k,1},\dots,v_{k,k},v, v_{k,k+1},\dots,v_{k,2k},x_1',x_2'
	\end{multline*}
	is the square of a path with end-tuples $(x_2,x_1)$ and $(x_1',x_2')$ that contains $v$.
	Indeed, $x_1$ and $x_2$ are common neighbours of $v_{1,1}$ and $v_{1,2}$, while $x_1'$ and $x_2'$ are common neighbours of  $v_{k,2k-1}$ and $v_{k,2k}$; moreover the vertex $v$ is a common neighbour of $v_{k,k-1},v_{k,k},v_{k,k+1}$ and $v_{k,k+2}$, the vertex $v_j$ is a common neighbour of $v_{j,k-1},v_{j,k},v_{j,k+1}$ and $v_{j,k+2}$, and the vertex $w_j$ is a common neighbour of $v_{j,2k-1},v_{j,2k},v_{j+1,1}$ and $v_{j+1,2}$.
	This fills the gaps left after the initial construction above where we used Lemma~\ref{lem:paths2} with the graph $G_4$.
	We now add $v$ to $V^\ast$ and all other used vertices $v_{j,j'}$ for $1 \le j \le k$ and $1 \le j' \le 2k$, $v_i,w_i$ for $1 \le i \le k-1$, and $x_1$ and $x_2$ to $W$.
	Note we do not add $x_1'$ and $x_2'$ to $W$ as $(x_1',x_2')$ is the end-tuple of the extended square of a path.
	Observe that we used $2k(k+1)$ vertices to cover $v$ and that we covered exactly $2k$ vertices from each cluster of the copy of $K_{1,k}$.
	
	\begin{figure}[htpb]
		\centering
		\includegraphics{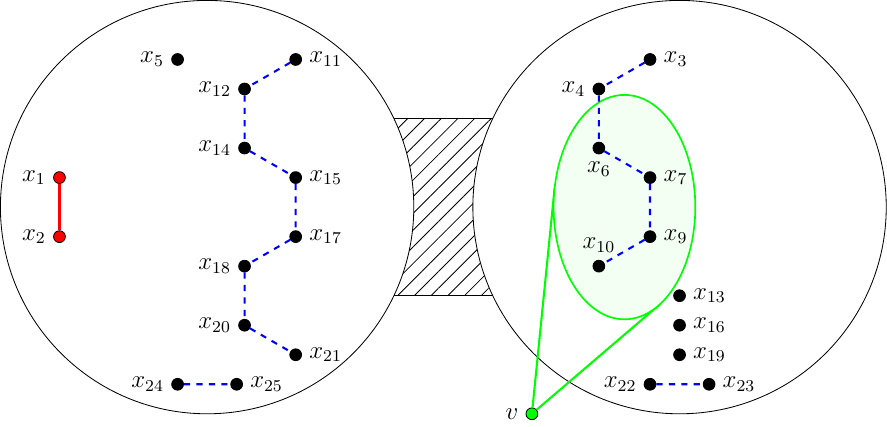}
		\includegraphics{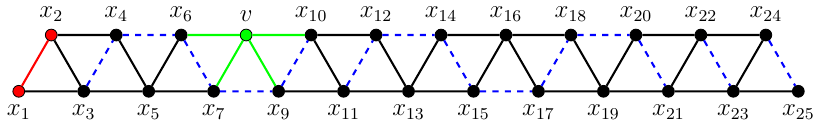}
		\captionsetup{font=footnotesize}
		\captionof{figure}{Covering one vertex $v \in V_0$ by the square of a path for a copy of $K_{1,1}$ in ${\mathcal K}$.
			We keep the position of the end-tuples in $V_i$ and the sizes of the clusters balanced.
			The edges within the clusters (dashed blue) come from the random graph and the edges between the clusters (green and black) come from regularity.}
		\label{fig_absorb2}
	\end{figure}
	
	Now we move to the construction for the case $i=i(v) \in \cJ$.
	We will use always the same construction regardless of the value of $k$ as illustrated in Figure~\ref{fig_absorb2} in a similar way as Figure~\ref{fig_absorb1b} earlier.
	By assumption we have $j(v)=1$ and thus $|N_G(v,U_{i,1} \setminus W)| \ge \tfrac{d}{k-1}n_0$.
	Let $(x_1,x_2) \in V_i^2$ and $(y_1,y_2) \in U_{i,1}^2$ be the end-tuples of the connecting paths found in the previous step and recall that $x_1$ and $x_2$ have at least $\tfrac 34 d^2 n_0$ common neighbours in $U_{i,1}$.
	We will extend the square of the connecting path ending in $(x_1,x_2)$, without using neither $y_1$ nor $y_2$, by constructing the square of a path on pairwise distinct vertices $x_1,x_2,\dots,x_{24},x_{25}$, where $x_8=v$.
	Moreover, we will make sure that the new end-tuple $(x_{24},x_{25})$ belongs to $V_i$, and that $x_{24}$ and $x_{25}$ have at least $\tfrac 34 d^2 n_0$ common neighbours in $U_{i,1}$.
	Let $Z_3=Z_4$ be the common neighbourhood of $x_1$ and $x_2$ in $U_{i,1} \setminus (W \cup \{y_1,y_2\})$ and $Z_6=Z_7=Z_9=Z_{10}$ be the neighbourhood of $v$ in $U_{i,1} \setminus (W \cup \{y_1,y_2\})$.
	Note that $|Z_3| \ge \tfrac12 d^2 n_0$ and $|Z_6| \ge \tfrac{d}{k-1}n_0-2$.
	Using regularity, it is easy to see that for at least $\delta (kT)^{-8} n^8$ tuples in $Z_3 \times Z_4 \times Z_6 \times Z_7 \times Z_9 \times Z_{10} \times U_{i,1} \times U_{i,1}$ the common neighbourhood of their vertices in $V_i$ has size at least $\frac 12 d^8 n_0$ (where we added the set $U_{i,1}$ twice only for the following application of Lemma~\ref{lem:paths2}).
	Thus~\eqref{random_graph_Lemma2.5} guarantees we find a path in $G_5$ on six vertices $x_3,x_4,x_6,x_7,x_9,x_{10}$ (we ignore the two vertices in $U_{i,1}$) with $x_j \in Z_j$ and such that the set $Z$ of the common neighbours of $x_3,x_4,x_6,x_7,x_9,x_{10}$ in $V_i \setminus (W \cup \{x_1,x_2\})$ has size at least $|Z| \ge \frac 14 d^8 n_0$.
	We let $x_5$ be any vertex of $Z$.
	Again using regularity, there are at least $\delta (kT)^{-8} n^8$ tuples in $(Z \setminus \{x_5\})^8$ such that the eight vertices in the tuple have at least $\tfrac 12 d^8 n_0$ common neighbours in $U_{i,1}$.
	These tuples are in $H_2$ as well and thus~\eqref{random_graph_Lemma2.5} guarantees we find in $G_5$ a path on eight vertices $x_{11}, x_{12}, x_{14}, x_{15}, x_{17}, x_{18}, x_{20}, x_{21}$ such that these eight vertices belong to $Z \setminus \{x_5\}$ and the set $Z'$ of their common neighbours in $U_{i,1} \setminus (W \cup \{y_1,y_2,x_3,x_4,x_5,x_6,x_7,x_9,x_{10}\})$ has size $|Z'| \ge \tfrac 14 d^8n_0$.
	We let $x_{13}$, $x_{16}$ and $x_{19}$ be any three vertices of $Z'$.
	Again with~\eqref{random_graph_Lemma2.5}, we find an edge $x_{22}x_{23}$ of $G_6$ such that $x_{22},x_{23} \in Z' \setminus \{x_{13},x_{16},x_{19}\}$ and their common neighbourhood $Z''$ in $V_i \setminus (W \cup \{y_1,y_2,x_1,x_2,x_5,x_{11}, x_{12}, x_{14}, x_{15}, x_{17}, x_{18}, x_{20}, x_{21}\})$ has size at least $\tfrac 12 d^2n_0$.
	With another application of~\eqref{random_graph_Lemma2.5} we find an edge $x_{24}x_{25}$ of $G_6$ such that $x_{24}, x_{25} \in Z''$ and they have at least $\tfrac 34 d^2n_0$ common neighbours in $U_{i,1}$.
	This gives the square of a path on the vertices $x_1, \dots,x_7,v,x_9,\dots,x_{25}$. 
	We add $v$ to $V^\ast$ and $x_1, \dots,x_7,x_9,\dots,x_{23}$ to $W$; note we do not remove of $x_{24}$ and $x_{25}$ as $(x_{24},x_{25})$ is the end-tuple of the extended square of a path.
	Note that we used $22$ vertices to cover $v$ and that we covered exactly $11$ vertices from each cluster $V_i$ and $U_{i,1}$.
	
	We keep covering the vertices of $V_0$ in this way until $V^\ast=V_0$, then we remove the vertices in $W$ from the clusters.
	
	\textbf{Completing the square of the Hamilton cycle.}
	Before finishing the proof, we summarise what we have done so far and, abusing notation, we still denotes the clusters by $V_i,U_{i,1},\dots,U_{i,k}$ for $i \in \cI$ and $V_i, U_{i,1}$ for $i \in \cJ$, even if we removed several vertices from them in the previous steps of the proof while connecting stars and covering $V_0$.
	We have covered all vertices of $G$, except those that are still in the clusters of a copy of $K_{1,1}$ or $K_{1,k}$, with the squares of short paths with the following properties.
	Their end-tuples belong to some $V_i$ with $i \in \cI$, or some $V_j$ or some $U_{j,1}$ with $j \in \cJ$.
	Moreover, for each $i \in \cI$, the cluster $V_i$ contains exactly two of such tuples, say $(x_1,x_2)$ and $(y_1,y_2)$, such that $x_1$ and $x_2$ have at least $\tfrac 34 d^2 n_0 - \sqrt{\delta}n_0 \ge \tfrac 12 d^2 n_0$ common neighbours in $U_{i,j}$ for $j=1,\dots,k$, and the same holds for $y_1$ and $y_2$.
	Additionally, for each $j \in \cJ$, the clusters $V_j$ and $U_{j,1}$ contain exactly one such tuple each, say $(x_1,x_2)$ and $(y_1,y_2)$ respectively, such that $x_1$ and $x_2$ have at least $\tfrac 12 d^2 n_0$ common neighbours in $U_{j,1}$, and $y_1$ and $y_2$ have at least $\tfrac 12 d^2 n_0$ common neighbours in $V_j$.
	Therefore, it remains to cover the clusters by extending the squares of paths we already have.
	
	Let $i \in \cI$.
	Note that we still have that $|U_{i,1}|=\dots=|U_{i,k}|$ and $|V_i| -4 - |U_{i,k}| \equiv -1 \pmod{3k-1}$ for $i \in \cI$.
	Recall that we removed $\delta n_0+4(1-\delta)$ vertices from each $U_{i,j}$ at the beginning.
	Moreover, while making the regular pairs super-regular and during the previous step, we removed the same number of vertices from each cluster of a copy of $K_{1,k}$ and at most $2 \sqrt{\delta}n_0$ vertices from each. 
	Therefore,  using that $2 \delta<\delta_0$, we get $|U_{i,j}| \ge (1-\delta_0)(|V_i|-4)$ for $j=1, \dots, k$.
	Note that also $|U_{i,k}| \le (1-\delta)(|V_i|-4)$ still holds.
	In addition, observe that $(V_i,U_{i,j})$ is $(\eps',\tfrac{d}{2})$-super-regular.
	Now let $(x_1,x_2)$ and $(y_1,y_2)$ be the two end-tuples in $V_i$ of connecting paths and recall they both have at least $\tfrac 12 d^2 n_0$ common neighbours in $U_{i,j}$ for each $j=1,\dots,k$.
	We apply the Star Cover Lemma (Lemma~\ref{lem:multipartite2}) with the random graph $G_1$ to $V_i,U_{i,1},\dots,U_{i,k}$ and get the square of a path with end-tuples $(x_2,x_1)$ and $(y_2,y_1)$, covering all vertices in $V_i \cup U_{i,1} \cup \dots \cup U_{i,k}$.
	
	For $i \in \cJ$ we proceed similarly.
	We have $|V_i|=|U_{i,1}|$ and two end-tuples $(x_1,x_2)$ in $V_i$ and $(y_1,y_2)$ in $U_{i,1}$ of connecting paths.
	Since $x_1$ and $x_2$ have at least $\tfrac 12 d^2 n_0$ common neighbours in $U_{i,1}$, and $y_1$ and $y_2$ have at least $\tfrac 12 d^2 n_0$ common neighbours in $V_i$, we can apply the Pair Cover Lemma (Lemma~\ref{lem:bipartite}) with the random graph $G_2$ to $V_i,U_{i,1}$, and get the square of a path with end-tuples $(x_2,x_1)$ and $(y_2,y_1)$, covering all vertices in $V_i \cup U_{i,1}$.
	
	This completes the square of a Hamilton cycle covering all vertices of $G$ and finishes the proof.
\end{proof}

\section{Regularity in auxiliary graphs}
\label{sec:aux_reg}

The aim of this section is to prove the main technical lemma behind the Star Cover Lemma (Lemma~\ref{lem:multipartite2}), whose proof is provided in Section~\ref{sec:auxiliary}.
In Lemma~\ref{lem:multipartite2}, we have to find the square of a Hamilton path (subject to additional conditions), where we can use both  deterministic and random edges.
Here we look at the edges coming from the random graph and show that we can find many disjoint copies of the square of a short path in the random graph, with some additional properties with respect to the deterministic graph, which guarantee we can then nicely connect them and build the desired structure.
We state a more general version of the result we need for Lemma~\ref{lem:multipartite2}, as we believe it might be of independent interest and helpful in other problems of similar flavour.
Before stating the lemma we introduce some definitions. 

\begin{definition}[$H$-transversal family]
	\label{def:transversal}
	Let $H$ be a graph and set $h=|V(H)|$.
	Let $V, U_1, \dots, U_h$ be disjoint sets of vertices and $G$ be a graph on vertex set $V \cup U_1 \cup \dots \cup U_h$. 
	A family $\cH$ of pairwise vertex-disjoint copies of $H$ in $G[U_1,\dots,U_h]$ is said to be a $H$-transversal family if, in addition, there exists a labelling $v_1,\dots,v_h$ of the vertices of $H$ such that, for each copy $H \in \cH$, the vertex $v_i$ is embedded into $U_i$ for each $i=1,\dots,h$.
\end{definition}

With $\cH$ being a $H$-transversal family in $G$, we define the following auxiliary bipartite graph $\cT_G(\cH,V)$.

\begin{definition}[Auxiliary graph $\cT_G(\cH,V)$]
	\label{def:aux_T_G(H,V)}
	We define $\cT_G(\cH,V)$ to be the bipartite graph with partition classes $\cH$ and $V$, where the edge between $H \in \cH$ and $v \in V$ appears if and only if all vertices of $H$ are incident to $v$ in $G$.
\end{definition}

We prove the following general lemma, which was already proved in the case when $H=K_2$ in~\cite[Lemma~8.1]{triangle_paper}.

\begin{lemma}
	\label{lem:aux_reg}
	Let $H$ be a graph on $h$ vertices. 
	For any $d,\delta,\eps'>0$ with $2\delta \le d$ there exist $\eps,C>0$ such that the following holds.
	Let $V,U_1,\dots,U_h$ be pairwise disjoint sets with $|V|=n$ and $|U_i|=m=(1 \pm \tfrac 12) n$ for $i=1,\dots,h$ such that $(V,U_i)$ is $(\eps,d)$-super-regular with respect to a graph $G$ for $i=1,\dots,h$.
	Furthermore, suppose that $p \ge C n^{-1/m_1(H)}$.
	
	Then a.a.s.~there exists an $H$-transversal family $\cH \subseteq G(U_1,\dots,U_h,p)$ of size $|\cH|\ge (1-\delta)m$ such that the pair $(\cH,V)$ is $(\eps',d^{h+1} 2^{-h-3})$-super-regular with respect to the auxiliary graph $\cT_G(\cH,V)$.
\end{lemma}

The lemma shows that not only there is a large $H$-transversal family $\cH$, but that we can additionally require that $(\cH,V)$ is a super-regular pair in $\cT_G(\cH,V)$.
The proof of Lemma~\ref{lem:multipartite2} will use the special case of $H$ being the square of a path on $k$ vertices.

Before giving a proof of Lemma~\ref{lem:aux_reg}, we introduce an auxiliary $h$-partite $h$-uniform hypergraph $F=F_{G,V}(U_1,\dots,U_h)$ to encode the potential tuples in $U_1 \times \dots \times U_h$ that we would like to use for building the copies of $H$ for the family $\cH$.

\begin{definition}[Auxiliary hypergraph $F$]
	\label{def:aux_F_{G,V}}
	Let $h \ge 1$ be an integer, $V,U_1,\dots,U_h$ be pairwise disjoint sets with $|V|=n$ and $|U_i|=m=(1 \pm \tfrac 12) n$ for $i=1,\dots,h$.
	We define $F=F_{G,V}(U_1,\dots,U_h)$\footnote{We remark that the definition of $F$ depends on $d$ as well. However, as this will always be clear from the context, we omit writing $d$ explicitly in $F_{G,V}(U_1,\dots,U_h)$.} to be the $h$-partite $h$-uniform hypergraph on $U_1 \times \dots \times U_h$ where a tuple $(u_1,\dots,u_h) \in U_1 \times \dots \times U_h$ is an edge of $F$ if and only if the vertices $u_1,\dots,u_h $ have at least $\tfrac12 d^h n$ common neighbours in the set $V$ in the graph $G$, i.e.~$| \bigcap_{i=1}^hN_G(u_i,V)| \ge \tfrac12 d^h n$.
	
	Similarly, given a set $X \subseteq V$, we call an edge $(u_1,\dots,u_h) \in E(F)$ \emph{good for $X$} if and only if there are at least $\tfrac12 d^h|X|$ vertices in $X$ that are incident to all of $u_1,\dots,u_h$ in $G$.
	We denote by $F_X$ the spanning subgraph of $F$ with edges that are good for $X$.
\end{definition}

However, we can only use those copies of $H$ which actually do appear in the random graph.
In order to encode that, we define the random spanning subgraph $\tilde{F}$ of $F$ as follows.

\begin{definition}[Auxiliary hypergraph $\tilde{F}$]
	\label{def:aux_tilde_F_{G,V}}
	Let $H$ be a graph on $h$ vertices. 
	Let $V,U_1,\dots,U_h$ be pairwise disjoint sets and $F$ be the hypergraph defined in Definition~\ref{def:aux_F_{G,V}}.
	After revealing the edges of the random $h$-partite graph $G(U_1,\dots,U_h,p)$, we denote by $\tilde F$ the (random) spanning subhypergraph of $F$ formed by those edges $(u_1,\dots,u_h)$ of $F$ for which the vertices $u_1,\dots,u_h$ give a copy of $H$ in the revealed random graph $G(U_1,\dots,U_h,p)$.
	We will say that $\tilde{F}$ is supported by $G(U_1,\dots,U_h,p)$.
\end{definition}

We remark that if $(u_1,\dots,u_h) \in U_1 \times \dots \times U_h$ is an edge of $\tilde F$, then the vertices $u_1, \dots, u_h$ give a copy of $H$ in $G(U_1,\dots,U_h,p)$ and have at least $\tfrac12 d^h n$ common neighbours in the set $V$ in the graph $G$.
We state some additional properties of $F$ below.

\begin{lemma}
	\label{lem:auxF}
	Let $H$ be a graph on $h$ vertices. 
	Let $0<d<1$ and $\eps \le \min \{ \tfrac12 d^{h-1}, d(1-2^{-1/h}) \}$.
	Let $G$ be a graph on vertex set $V \cup U_1 \cup \dots \cup U_h$ with $|V|=n$ and $|U_1|=\dots=|U_h|=m=(1 \pm \tfrac12)n$, and assume $(V,U_i)$ is a $(\eps,d)$-super-regular pair with respect to $G$ for each $i=1,\dots,h$.
	Let $F=F_{G,V}(U_1,\dots,U_h)$ be the hypergraph defined in Definition~\ref{def:aux_F_{G,V}}.
	Then the following holds:
	\begin{enumerate}[label=(\roman*)]
		\item \label{claim:min_degree} The minimum degree of $F$ is at least $(1-h\eps)m^{h-1}$.
		\item \label{claim:degree_god} If $|X| \ge 2\eps n d^{1-h}$, all but at most $\eps m$ vertices from each $U_i$ have degree at least $(1-h\eps)m^{h-1}$ in $F_X$.		
	\end{enumerate}
\end{lemma}

Moreover the subgraph $\tilde F$ keeps roughly the expected number of edges of $F$.	

\begin{lemma}	
	\label{lem:H_copies}
	For any graph $H$ on $h \ge 2$ vertices and any $\delta>0$ there exist $\eps>0$ and $C>0$ such that the following holds for $p \ge C n^{-1/m_1(H)}$.
	Let $V,U_1,\dots,U_h$ be pairwise disjoint sets with $|V|=n$ and $|U_i|=m=(1 \pm \tfrac 12) n$ for $i=1,\dots,h$, and $F=F_{G,V}(U_1,\dots,U_h)$ be the hypergraph defined in Definition~\ref{def:aux_F_{G,V}}.
	Then a.a.s.~for any sets $U_i' \subseteq U_i$ of size at least $\delta m$ for $i=1,\dots,h$, we have that $\tilde{F'} =\tilde{F}[U_1',\dots,U_h']$ satisfies
	\begin{equation}
		\label{eq:copies_F}
		e(\tilde F')  = (1 \pm \sqrt{\eps}) \prod_{i=1}^h |U_i'| p^{e(H)}.
	\end{equation}
	Moreover, if $|X| \ge 2\eps n d^{1-h}$, then with probability at least $1-e^{-n}$, for any choice of $U_1',\dots,U_h'$ as above and with $\tilde{F_X'}=\tilde{F_X}[U'_1,\dots,U'_h]$, we have
	\begin{equation}
		\label{eq:copies_F_good}
		e(\tilde F_X')  \ge (1-\sqrt{\eps}) \prod_{i=1}^h |U_i'| p^{e(H)}.
	\end{equation}
\end{lemma}

We remark that we will show a more general version of Lemma~\ref{lem:H_copies}.
Indeed our proof will only use that $F$ satisfies~\ref{claim:min_degree} and that $F_X$ satisfies~\ref{claim:degree_god} for all $X$ with $|X| \ge 2\eps n d^{1-h}$. 
Thus~\eqref{eq:copies_F} holds for any $h$-partite $h$-uniform hypergraph $F$ on partition classes $U_1,\dots,U_h$ of size $m=(1 \pm \tfrac 12 )n$ with minimum degree $(1-h\eps)m^{h-1}$.
Similarly~\eqref{eq:copies_F_good} holds for all subgraphs $F_X$ of such $F$, such that all but $\eps m$ vertices in each class have degree at least $(1-h\eps)m^{h-1}$ in $F_X$.

The proof of Lemma~\ref{lem:auxF} relies on a standard application of the regularity method and that of Lemma~\ref{lem:H_copies} follows from an application of Chebyshev's and Janson's inequalities. 
Therefore we postpone them to Appendix~\ref{sec:appendix} and we now turn to the main proof of this section.

\begin{proof}[Proof of Lemma~\ref{lem:aux_reg}]
	Given a graph $H$ on $h \ge 2$ vertices and $d,\delta,\eps' > 0$ with $2\delta \le d$, suppose that 
	\[\eps \ll \eta \ll \gamma \ll d,\delta,\eps'\]
	are positive real numbers such that
	\[
	\gamma \le \tfrac{\eps'(1-\delta)}{2}, \quad
	\eta \log(\tfrac{1}{\eta}) \le \tfrac14 \gamma, \quad
	\eps \le \min \{ \tfrac12 \eta d^{h-1}, d(1-2^{-1/h}) , \tfrac{1}{14^2} \gamma^2 \}\, ,
	\]
	where additionally we require that $\eps$ is small enough for  Lemma~\ref{lem:H_copies} with input $H$ and $\delta$.
	Furthermore, let $C>0$ be large enough for Lemma~\ref{lem:H_copies} with the same input.
	
	Given $n$, let $V,U_1,\dots,U_h$ be sets of size $|V|=n$ and $|U_j|=m=(1 \pm \frac 12 ) n$ for $j=1,\dots,h$ such that $(V,U_j)$ is $(\eps,d)$-super-regular with respect to a graph $G$ for $j=1,\dots,h$.
	To find the family $\cH$ we will now reveal edges of $G(U_1,\dots,U_h,p)$ with probability $p \ge C n^{-1/m_1(H)}$ and consider the spanning subgraph $\tilde{F}$ of $F=F_{G,V}(U_1,\dots,U_h)$ as defined in Definitions~\ref{def:aux_F_{G,V}} and~\ref{def:aux_tilde_F_{G,V}}.
	Let $X \subseteq V$ be of size $\eta n$.
	With Lemma~\ref{lem:H_copies}, we can assume that for all $U_j' \subseteq U_j$ of size $\delta m$ for $j=1,\dots,h$ we have that \eqref{eq:copies_F} and~\eqref{eq:copies_F_good} hold, where the latter holds for all $X$ as above by a union bound.
	
	Using a random greedy process we now choose a family of transversal copies $\cH$ of $H$ of size $(1-\delta)m$ in $\tilde{F}$ as follows.
	Having chosen copies $H_1,\dots,H_t\in \tilde{F}$ with $t<(1-\delta)m$, we pick $H_{t+1}$ uniformly at random from all edges of $\tilde{F}$ that do not share an endpoint with any of $H_1,\dots,H_t$.
	This is possible since by~\eqref{eq:copies_F} there is always an edge in $\tilde F[U_1',\dots,U_h']$ for any subsets $U_j' \subseteq U$ of size at least $\delta m$ for $j=1,\dots,h$, and thus a transversal copy of $H$ in $G(U_1,\dots,U_h,p)$.
	For $i=1, \dots, t$, we denote the $i$-th chosen copy of $H$ for $\cH$ by $H_i$, by $\cH_i$ the history $H_1,\dots,H_i$, and let $\cH_0$ be empty.
	It remains to show that a.a.s. $(\cH,V)$ is $(\eps',2^{-h-3} d^{h+1})$-super-regular with respect to the auxiliary graph $\cT=\cT_G(\cH,V)$.
	
	Observe that any $H \in \cH$ has $|N_\cT(H)| \ge \tfrac12 d^h n \ge d^{h+1} 2^{-h-3} |V|$ by construction.
	Moreover for any $v \in V$ we have $|N_\cT(v)| \ge 2^{-h-3} d^{h+1} m$; this can be shown as follows.
	Consider the first $\tfrac12 dm$ chosen copies, then for $i=1,\dots,\tfrac12 dm$, by~\eqref{eq:copies_F}, there are at most $(1+\sqrt{\eps}) \prod_{j=1}^h |U_j| p^{e(H)}$ available copies to chose $H_i$ from.
	On the other hand, as long as $i< \tfrac12 dm$, the vertex $v$ has at least $\tfrac12 d m \ge \delta m$ neighbours $U_j' \subseteq U_j$ for $j=1,\dots,h$ that are not covered by the edges in  $\cH_{i-1}$. 
	Therefore, by~\eqref{eq:copies_F} there are at least $(1-\sqrt{\eps}) \prod_{j=1}^h |U_j'| p^{e(H)}$ choices for $H_i$ such that $H_i \in N_{\cT}(v)$.
	
	Hence, for $i=1,\dots,\tfrac 12 dn$, we get
	\[\PP[H_i \in N_{\cT}(v) | \cH_{i-1}] \ge \frac{(1-\sqrt{\eps}) \prod_{j=1}^h |U_j'| p^{e(H)}}{(1+\sqrt{\eps}) \prod_{j=1}^h |U_j| p^{e(H)}} \ge \frac{(1-\sqrt{\eps}) (\tfrac12 d)^h }{(1+\sqrt{\eps})} \ge 2^{-h-1} d^h.\]
	As this holds independently of the history of the process, this process dominates a binomial distribution with parameters $\tfrac12 dm$ and $2^{-h-1} d^h$.
	Therefore, even though the events are not mutually independent, we can use Chernoff's inequality (Lemma~\ref{lem:chernoff}) to infer that $|N_\cT(v)| \ge 2^{-h-3} d^{h+1} m$ with probability at least $1-n^{-2}$.
	Then, by applying the union bound over all $v \in V$, we obtain that a.a.s.~$|N_\cT(v)| \ge 2^{-h-3} d^{h+1} m \ge 2^{-h-3} d^{h+1} |\cH|$ for all $v \in V$.
	
	Next let $X \subseteq V$ be any subset with $|X| = \eta n$ and let $t=(1-\delta)m$.
	For $i=0$, $1,\dots,t-1$, we obtain from~\eqref{eq:copies_F} that there are at most $(1+\sqrt{\eps}) \prod_{j=1}^h |U_j'| p^{e(H)}$ edges in $\tilde{F} \setminus \cH_i$ available for choosing $H_{i+1}$, of which, by~\eqref{eq:copies_F_good}, at least $(1-\sqrt{\eps}) \prod_{j=1}^h |U_j'| p^{e(H)}$ are in $\tilde F_X$.
	Then
	\[ \PP \left[ H_i \text{ good for }X \Big| \cH_{i-1} \right] \ge \frac{(1-\sqrt{\eps}) \prod_{j=1}^h |U_j'| p^{e(H)}}{(1+\sqrt{\eps}) \prod_{j=1}^h |U_j'| p^{e(H)}} \ge (1-2\sqrt{\eps}).\]
	Again, as the lower bound on the probability holds independently of the history of the process, this process dominates a binomial distribution with parameters $(1-\delta)m$ and $(1-2\sqrt{\eps})$.
	We let $B_X \subseteq \cH$ be the copies in $\cH$ that are not good for $X$ and deduce
	\[\EE[|B_X|] \le (1-\delta) m \,  2 \sqrt{\eps} \le 2 \sqrt{\eps} m\, .\]
	Then we get from Chernoff's inequality (Lemma~\ref{lem:chernoff}) that, since $\gamma m \ge 14 \sqrt{\eps} m \ge 7 \EE[|B_X|]$, we have
	\[ \mathbb{P}\left[ |B_X|  > \gamma m \right] \le  \exp \left(-\gamma m \right)\,.\]	
	There are at most $\binom{n}{\eta n} \le \left(\tfrac{e}{\eta} \right)^{\eta n} \le  \exp (\eta \log(\tfrac{1}{\eta}) n) \le \exp(\tfrac 12 \gamma m)$ choices for $X$ and, thus, with the union bound over all these choices, we obtain that a.a.s.~there are at most $\gamma m$ bad copies in $\cH$ for any $X \subseteq V$ with $|X|=\eta n$.
	
	Fix a choice of $\cH$ such that there are at most $\gamma m$ bad copies for any $X \subseteq V$ with $|X|=\eta n$.
	Then for any set $X' \subseteq V$ and $\cH'\subseteq \cH$ with $|X'| \ge \eps' n$ and $|\cH'|\ge \eps' |\cH|$  we find
	\[ e_\cT(\cH',X') \ge (|\cH'| - \gamma m) \frac{d^{h} \eta n}{2} \frac{|X'|}{2\eta n} \ge \frac{d^h}{8} |\cH'| |X'|\]
	by partitioning $X'$ into pairwise disjoint sets of size $\eta n$.
	This implies that a.a.s.~the pair $(\cH,V)$ is $(\eps',d^{h+1}2^{-h-3})$-super-regular with respect to $\cT_G(\cH,V)$.
\end{proof}

\section{Proof of auxiliary lemmas}
\label{sec:auxiliary}

In this section we prove the Star Cover Lemma (Lemma~\ref{lem:multipartite2}) and then derive the Pair Cover Lemma (Lemma~\ref{lem:bipartite}).
Throughout the section we denote the square of a path on $k$ vertices by $H^{(k)}$ and we list their vertices as $u_1,\dots,u_k$, meaning that the edges of the square are $u_iu_j$ for each $1 \le |i-j| \le 2$.
We start with a short overview of our argument for the proof of Lemma~\ref{lem:multipartite2}, where we want to find the square of a Hamilton path covering $V, U_1, \dots, U_k$ and with end-tuples $(x,x')$ and $(y,y')$. 
Our proof will follow four steps, with the decomposition of this square of a Hamilton path in random and deterministic edges being outlined in Figure~\ref{fig_HC2_k3}.

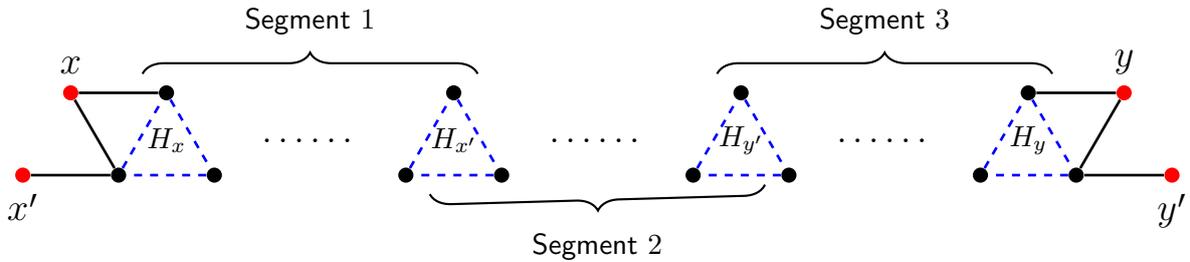
\begin{figure}[htpb]
	
	\begin{tikzpicture}[scale=0.63,
		point/.style ={circle,draw=none,fill=#1,
			inner sep=0pt, minimum width=0.2cm,node contents={}}
		]
		%Square of Hamilton path
		\node (y) at (19,{2-sqrt(3)})   [point=red, label=below:$y'$];
		\node (yp) at (18,2)   [point=red, label=above:$y$];
		\node (xp) at (-4,2)   [point=red, label=above:$x$];
		\node (x) at (-5,{2-sqrt(3)})   [point=red, label=below:$x'$];
		\foreach \s in {-2,4,10,16} {
			\node (a\s) at (\s-1,{2-sqrt(3)})   [point=black];
			\node (b\s) at (\s,2)   [point=black];
			\node (c\s) at (\s+1,{2-sqrt(3)})   [point=black];
		}
		
		\node[font=\small\sffamily] (Hx) at (-2,1) {$H_x$};
		\node[font=\small\sffamily] (Hx') at (4,1) {$H_{x'}$};
		\node[font=\small\sffamily] (Hy') at (10,1) {$H_{y'}$};
		\node[font=\small\sffamily] (Hy) at (16,1) {$H_y$};
		\node () at (1,1) {$\dots \dots$};
		\node () at (7,1) {$\dots \dots$};	
		\node () at (13,1) {$\dots \dots$};	
		
		\foreach \s in {-2,4,10,16} {
			\draw[blue,dashed,line width=1pt] (a\s)--(b\s)--(c\s)--(a\s);
		}
		\draw[black,line width=1pt]
		(b-2)--(xp)--(a-2)--(x)
		(b16)--(yp)--(c16)--(y);
		
		\draw [decorate, decoration = {brace,amplitude=8pt}, thick] (-2.5,2.4) --  (4.5,2.4);
		\draw [decorate, decoration = {brace,amplitude=8pt}, thick] (9.5,2.4) --  (16.5,2.4);
		\draw [decorate, decoration = {brace,mirror,amplitude=8pt}, thick] (3.5,{2-sqrt(3)-0.4}) --  (10.5,{2-sqrt(3)-0.2});
		\node[font=\small\sffamily] () at (1,3.5) {Segment $1$};
		\node[font=\small\sffamily] () at (13,3.5) {Segment $3$};	
		\node[font=\small\sffamily] () at (7,{2-sqrt(3)-1.5}) {Segment $2$};	
	\end{tikzpicture}
	
	\captionsetup{font=footnotesize}
	\captionof{figure}{The square of a Hamilton path with end-tuples $(x,x')$ and $(y,y')$ in Lemma ~\ref{lem:multipartite2} and its decomposition into edges from $G$ (black) and from $G(n,p)$ (dashed blue). Each dotted $H_x, H_{x'}, H_{y'}, H_y$ stands for a copy of $P_k^2$ with edges all from $G(n,p)$. Segment $1$ and $3$ (resp.~segment $2$) are realised through several copies of the structure in Figure~\ref{fig_HC2_k3_2} (resp.~Figure~\ref{fig_connectHH'}).}
	\label{fig_HC2_k3}
\end{figure}

To ensure that $(x,x')$ and $(y,y')$ are the end-tuples of the square of the path, we will first find copies $H_x$ and $H_y$ of $H^{(k)}$ that are connected to the tuples $(x,x')$ and $(y,y')$ (c.f.~Figure~\ref{fig_HC2_k3}).
Moreover, with Lemma~\ref{lem:aux_reg}, we will find a large family $\cH$ of transversal copies of $H^{(k)}$ in $U_1,\dots,U_k$ (c.f.~Definition~\ref{def:transversal}) such that $(\cH,V)$ is super-regular with respect to the auxiliary graph $\cT_G(\cH,V)$ (c.f.~Definition~\ref{def:aux_T_G(H,V)}).
In particular, this will guarantee that most pairs $(H,H') \in \cH^2$ have many common neighbours in $V$ in the graph $G$.

The next step is to find random edges between the copies in $\cH$.
For that we will consider a directed auxiliary graph $\overline{F}$ with vertex set $\cH$ where, given $H,H' \in \cH$, with $H$ on $u_1,\dots,u_k$ and $H'$ on $u_1',\dots,u_k'$, the pair $(H,H')$ is an edge of $\overline{F}$ if and only if $u_ku_1'$ is a random edge and $V(H) \cup V(H')$ have many common neighbours in $V$ in the graph $G$.
This will allow us to connect $H$ to $H'$ with a random edge, while also giving many choices for vertices from $V$ to turn this into the square of a path on $2k+1$ vertices (c.f.~Figure~\ref{fig_connectHH'}).
We will use a random greedy procedure to find a long directed path $D$ in $\overline{F}$ that covers most of $\cH$, which is possible by the properties of $\cH$ and  our choice of $p \ge C n^{-1}$.
Additionally, we can guarantee that we will later be able to extend this into the square of a path using any subset of vertices $V' \subset V$ of the right size.
We denote the first and last copy of the path $D$ by $H_x'$ and $H_y'$, respectively.

\begin{figure}[!htb]
	\begin{minipage}[b]{.49\textwidth}
		\centering
		\begin{tikzpicture}[scale=0.65,
			point/.style ={circle,draw=none,fill=#1,
				inner sep=0pt, minimum width=0.2cm,node contents={}}
			]
			%Square of Hamilton path
			\node (w1) at (5,{2-sqrt(3)})   [point=black];
			\node (w2) at (4,2)   [point=black];
			\node (w3) at (3,{2-sqrt(3)})   [point=black];	
			\node (v) at (2,2)   [point=red, label=above:$v$];
			\node (x1) at (1,{2-sqrt(3)})   [point=black];
			\node (x2) at (0,2)   [point=black];
			\node (x3) at (-1,{2-sqrt(3)})   [point=black];
			
			\node[font=\small\sffamily] (H) at (0,1) {$H$};
			\node[font=\small\sffamily] (Hp) at (4,1) {$H'$};
			
			\draw[blue,dashed,line width=1pt]
			(x1)--(x2)--(x3)--(x1)
			(x3)--(w1)
			(w1)--(w2)--(w3)--(w1);
			\draw[black,line width=1pt]
			(x2)--(v)--(x1)
			(w2)--(v)--(w3);
		\end{tikzpicture}
		\captionsetup{font=footnotesize}
		\captionof{figure}{Connecting two copies $H$ and $H'$ of $P_k^2$, using one vertex $v \in V$ (red), edges from $G(n,p)$ (dashed blue) and edges from $G$ (black).}
		\label{fig_connectHH'}
	\end{minipage}
	\begin{minipage}[b]{.5\textwidth}
		\centering
		\begin{tikzpicture}[scale=0.65,
			point/.style ={circle,draw=none,fill=#1,
				inner sep=0pt, minimum width=0.2cm,node contents={}}
			]
			%Square of Hamilton path
			\node (y1) at (9,{2-sqrt(3)})   [point=black];
			\node (y2) at (8,2)   [point=black];
			\node (y3) at (7,{2-sqrt(3)})   [point=black];
			\node (v) at (6,2)   [point=red];
			\node (w1) at (5,{2-sqrt(3)})   [point=red];
			\node (w2) at (4,2)   [point=black, label=above:$z$];
			\node (w3) at (3,{2-sqrt(3)})   [point=red];	
			\node (vp) at (2,2)   [point=red];
			\node (x1) at (1,{2-sqrt(3)})   [point=black];
			\node (x2) at (0,2)   [point=black];
			\node (x3) at (-1,{2-sqrt(3)})   [point=black];
			
			\node[font=\small\sffamily] (H) at (0,1) {$H$};
			\node[font=\small\sffamily] (Hp) at (8,1) {$H'$};
			
			\draw[blue,dashed,line width=1pt]
			(x1)--(x2)--(x3)--(x1)
			(v)--(w1)--(w3)--(vp)
			(y1)--(y2)--(y3)--(y1);
			\draw[black,line width=1pt]
			(y2)--(v)--(y3)--(w1)--(w2)--(v)
			(x2)--(vp)--(x1)--(w3)--(w2)--(vp);
		\end{tikzpicture}
		\captionsetup{font=footnotesize}
		\captionof{figure}{Absorbing a vertex $z \in Z \subseteq U_1 \cup \dots \cup U_k$, using two copies $H$ and $H'$ of $P_k^2$, four vertices from $V$ (red), edges from $G(n,p)$ (dashed blue) and edges from $G$ (black).}
		\label{fig_HC2_k3_2}
	\end{minipage}
\end{figure}

In the next step, we take care of the set $Z$ of those vertices in $U_1 \cup \dots \cup U_k$ that are not covered by any copy of $H^{(k)}$ from $\cH$.
We will absorb each vertex $z \in Z$ into the square of a short path, using four vertices from $V$, two copies of $H^{(k)}$ in $\cH \setminus V(D)$ and random edges within $V$ (c.f.~Figure~\ref{fig_HC2_k3_2}).
In fact, we will be able to do that simultaneously for each vertex in $Z$, by constructing two squares of paths, one from $H_x$ to $H_x'$ and one from $H_y'$ to $H_y$, which contain all vertices of $Z$ and all copies of $H^{(k)}$ in $\cH \setminus V(D)$.

In the final step, we will find a perfect matching between the edges $(H,H')$ of $D$ and the remaining vertices of $V$, while making sure that the sizes of the two sets we want to match are the same.
A vertex $v \in V$ can be matched to $(H,H')$ if and only if $u_k,u_{k-1},u_1',u_2'$ are neighbours of $v$ in $G$ (with the labelling of the vertices of $H,H'$ as above).
This matching will close the gap between the two copies $H$ and $H'$ for each edge $(H,H')$ of $D$ with a vertex $v$ from $V$.
This will give the square of a path from $H_x'$ to $H_y'$ and, together with the other pieces from $H_x$ to $H_x'$ and from $H_y$ to $H_y'$, we will get the square of a Hamilton path with the correct end-tuples.
Ultimately, the shape of this square of the path is as illustrated in Figure~\ref{fig_HC2_k3}, where the segments between $H_x$ and $H_{x'}$ and between $H_{y'}$ and $H_y$ (resp.~between $H_{x'}$ and $H_{y'}$) are obtained by repeatedly inserting Figure~\ref{fig_HC2_k3_2} (resp.~Figure~\ref{fig_connectHH'}) several times.
We will now turn to the details of the argument.

\begin{proof}[Proof of Lemma~\ref{lem:multipartite2}]
	Given an integer $k \ge 2$, let $H^{(k)}$ be the square of a path on $k$ vertices and observe that $m_1(H^{(k)})=\tfrac{2k-3}{k-1}$.
	Given $0 < \delta' \le d \le 1$, let $\delta_1,\delta_0,\eps'>0$ with $2\delta_1 < \delta_0 < \min\{ \delta', d^{3k+3} 2^{-5k-22} \}$ and $\eps'< \delta_1^8$.
	Let $C_{\ref{lem:embedding}}$ be given by Lemma~\ref{lem:embedding} for input $2 \delta_1$ and $F$, where $F$ is the path on four vertices.
	Then let $\eps_{\ref{lem:aux_reg}}$ and $C_{\ref{lem:aux_reg}}$ be given by Lemma~\ref{lem:aux_reg} for input $H^{(k)}$, $d/2$, $\eps'$ and $\eps'/2$, where $\eps'$ plays the role of $\delta$ in the statement of Lemma~\ref{lem:aux_reg}.
	Let $\eps_{\ref{lem:H_copies}}$ and $C_{\ref{lem:H_copies}}$ be given by Lemma~\ref{lem:H_copies} with input $H^{(k)}$ and $\delta_0$.
	Finally let $\eps < \min\{\eps_{\ref{lem:aux_reg}}/2, \eps_{\ref{lem:H_copies}}, \eps'/4  \}$ and $C = \max \{ C_{\ref{lem:embedding}}, 2C_{\ref{lem:aux_reg}},2C_{\ref{lem:H_copies}}, 48 \eps'^{-1} \delta^{-1} \}$.
	Observe that $\eps < 2\delta_1$, as required. 
	The constant dependencies can be summarised as follows:
	\[
	\eps \ll \eps' \ll \delta_1 \ll \delta_0 \ll \delta' \le d\, .
	\]
	
	Let $G$ be a graph on $V \cup U_1 \cup \dots \cup U_k$, where $V,U_1,\dots,U_k$ are pairwise disjoint sets of size $|V|=n+4$ and $(1-\delta_0)n \le |U_i|=m \le (1-\delta_1)n$ for $i=1,\dots,k$ such that $n-m \equiv -1 \pmod{3k-1}$.
	Suppose that $(V,U_i)$ is a $(\eps,d)$-super-regular pair for $i=1,\dots,k$.
	Further let $(x,x')$ and $(y,y')$ be two tuples from $V$ such that both tuples have $d^2m/2$ common neighbours in $U_i$ for $i=1,\dots,k$ in the graph $G$.
	Let $\delta$ be such that $m=(1-\delta)n+1$ and observe that $\delta_1 \le \delta \le \delta_0+1/n$ and that the divisibility condition on $n-m$ implies that $\delta n \equiv 0 \pmod{3k-1}$.
	For later convenience, we define $m_0=\left(1-\tfrac{3k}{3k-1} \delta\right)n-1$, $t=\left(1-\tfrac{4k}{3k-1} \delta\right)n+1$ and $s=\frac{k}{3k-1}\delta n$.
	Since $\delta n \equiv 0 \pmod{3k-1}$, we observe that $m_0$, $t$ and $s$ are positive integers.
	Moreover, we have that
	\begin{gather}
		\label{eq:m_0} m_0 - t = s -2 \, ,\\
		\label{eq:s/k} (1-\delta)n-1-m_0= \frac{s}{k} \, ,\\
		\label{eq:t-1} n-4s = t-1 \, . 
	\end{gather}
	We reveal random edges in three rounds $G_1 \sim G(U_1,\dots,U_k, \tfrac 12 p)$, $G_2 \sim G(U_1,\dots,U_k,\tfrac 12 p)$ and $G_3 \sim G(V,p)$, where $p \ge C n^{-(k-1)/(2k-3)}$.
	Moreover we assume that a.a.s.~the events of Lemma~\ref{lem:embedding} and~\ref{lem:H_copies} hold in $G_3$ and $G_1$, respectively.
	
	\textbf{Finding transversal copies of $H^{(k)}$.}
	We start by ensuring that $(x,x')$ and $(y,y')$ can be the end-tuples of the square of a path.
	Fix $i=1,\dots,k$.
	Recall that each of the tuples $(x,x')$ and $(y,y')$ has $d^2 m /2$ common neighbours in $U_i$. 
	Thus we can pick disjoint sets $U_{i,x},U_{i,y} \subset U_i$ of size $d^2m/4$ such that $U_{i,x}$ and $U_{i,y}$ are in the common neighbourhoods of $(x,x')$ and $(y,y')$ in $U_i$, respectively.
	Let $F$ and $\tilde{F}$ be the hypergraphs defined in Definition~\ref{def:aux_F_{G,V}} and~\ref{def:aux_tilde_F_{G,V}}, and with $H^{(k)}$ and $\tilde{F}$ supported by $G_1$.
	Then, as we assumed that Lemma~\ref{lem:H_copies} holds in $G_1$, we find an edge in $\tilde F[U_{1,x},\dots,U_{k,x}]$ and an edge in $\tilde F[U_{1,y},\dots,U_{k,y}]$.
	Given the definition of $\tilde{F}$, these edges correspond to copies of $H^{(k)}$ in $G_1$ and we denote them by $H_x=(x_1,\dots,x_k)$ and $H_y=(y_1,\dots,y_k)$.
	
	For $i=1,\dots,k$, let $U_i'=U_i \setminus \{ x_i,y_i \}$ and observe that $|U_i'|=(1-\delta)n-1$ and $(U_i',V)$ is $(2\eps,d/2)$-super-regular.
	Then we apply Lemma~\ref{lem:aux_reg} with $G_2[U_1',\dots,U_k']$ to a.a.s.~obtain a family $\cH$ of transversal copies of $H^{(k)}$, of size $|\cH| \ge (1-\eps') (m-2) \ge (1-\tfrac{3k}{3k-1} \delta)n-1=m_0$ and such that the pair $(\cH,V)$ is $(\eps'/2,d^{k+1} 2^{-2k-4})$-super-regular with respect to $\cT_G(\cH,V)$, where $\cT_G(\cH,V)$ is the graph defined in Definition~\ref{def:aux_T_G(H,V)}.
	By removing some copies of $H^{(k)}$, we can assume that $|\cH| = m_0$ and still have that the pair $(\cH,V)$ is $(\eps',d^{k+1} 2^{-2k-5})$-super-regular with respect to $\cT_G(\cH,V)$.
	
	\textbf{Building the directed path $D$.}	
	Ultimately we want to find a directed path $D$, that has some of the copies of $H^{(k)}$ in $\cH$ as vertices.
	As we later would like to connect two copies of $H^{(k)}$ by one random edge and a vertex from $V$ to get the square of a path on $2k+1$ vertices (c.f.~Figure~\ref{fig_connectHH'}), we want them to appear consecutively in $D$ only if all their vertices have enough common neighbours in $V$ in the graph $G$.
	We encode this condition in the auxiliary graph $F^\ast$ with vertex set $\cH$ and where, given $H, H' \in \cH$, the edge $HH'$ is in $F^\ast$ if and only if the vertices in $V(H) \cup V(H')$ have at least $d^{2k+2} 2^{-4k-8} n$ common neighbours in $V$ in the graph $G$.
	
	\begin{claim}
		\label{claim:min_degree3}
		The minimum degree of $F^\ast$ is at least $(1-2\eps')m_0$.
	\end{claim}
	
	\begin{claimproof}[Proof of Claim~\ref{claim:min_degree3}]
		Any copy of $H \in \cH$ has degree at least $d^{k+1} 2^{-2k-4} n$ into $V$ in the graph $\cT_G(\cH,V)$.
		Then, by super-regularity in $\cT_G(\cH,V)$, all but $2 \eps' m_0$ copies of $H' \in \cH$ have at least $d^{2k+2} 2^{-4k-8} n$ common neighbours with $H$.
		This implies that $H$ has degree $(1-2\eps')m_0$ in $F^\ast$.
	\end{claimproof}
	
	For a set $X \subseteq V$, we call an edge $HH' \in E(F^\ast)$ \emph{good for $X$} if there is at least one vertex $v \in X$ that is incident to $H$ and $H'$ in $\cT_G(\cH,V)$.
	We denote the subgraph of $F^\ast$ with edges that are good for $X$ by $F^\ast_X$.
	
	\begin{claim}
		\label{claim:degree_good3}
		If $|X| \ge d^{-k-1} 2^{2k+5} \eps' n$, then all but at most $\eps' n$ vertices of $\cH$ have degree at least $(1-4\eps')m_0$ in $F^\ast_X$.
	\end{claim}
	
	\begin{claimproof}[Proof of Claim~\ref{claim:degree_good3}]
		By super-regularity, all but at most $\eps' n$ copies $H \in \cH$ have degree at least $d^{k+1} 2^{-2k-5} |X| \ge \eps' n$ into $X$ in the graph $\cT_G(\cH,V)$.
		Fixing any $H \in \cH$ with this property, all but at most $\eps' m_0$ copies $H' \in \cH \setminus\{H\}$ have at least one common neighbour with $H$ in $X$.
		From Claim~\ref{claim:min_degree3} we know that $\delta(F^\ast) \ge (1- 2\eps')m_0$ and, therefore, all but at most $\eps' n$ vertices from $\cH$ have degree at least $(1-4 \eps')m_0$ in $F^\ast_X$.
	\end{claimproof}
	
	We now define an auxiliary directed graph $\bar F$ on vertex set $\cH$ as follows.
	Given any $H$ and $H' \in \cH$ with $H=u_1,\dots,u_k$ and $H'=u_1',\dots,u_k'$, the pair $(H,H')$ is a directed edge of $\bar F$ if and only if $HH'$ is an edge of $F^\ast$ (which means that $H$ and $H'$ have many common neighbours in $V$ in $G$) and $u_ku_1'$ is an edge of $G_2$.
	As $G_2 \sim G(U_1,\dots,U_k, \tfrac12 p)$, observe that $\bar F$ is a random directed graph, which, for each edge $HH'$ of $F^\star$, contains each of the directed edges $(H,H')$ and $(H',H)$ with probability $p/2$, independently of each other and of all other edges.
	Therefore, given an edge $HH'$ of $F^\star$, we can talk about \emph{revealing} $(H,H')$ where, actually, we reveal the edge $u_ku_1'$ in $G_2$.
	If indeed $u_ku_1'$ is an edge of $G_2$, then we say that the edge $(H,H')$ is \emph{successful}, as it belongs to $\bar{F}$.
	
	We recall that we want to find a long directed path $D$ in $\bar F$ (see Claim \ref{claim:directed_path} below), whose edges satisfy additional properties (see Claim \ref{claim:H_matching2} below).
	For this we will use a random greedy process that explores $\bar F$ using a depth-first search algorithm.
	We do not reveal all the edges of $\bar F$ at the beginning, but, at each step of the algorithm, we only reveal those edges that are relevant for that step.
	In each step, the algorithm maintains a directed path $H_1,\dots,H_r$ in $\bar{F}$ and a set $B$ consisting of the vertices in $\cH \setminus \{ H_1,\dots,H_r \}$ whose neighbours have already been all revealed (we call them ``dead-ends").
	Additionally, we keep track of the vertices which have already been visited at least twice (due to backtracking of the algorithm) and denote their set by $A$, for which we note that $|A| \le |B|$.
	We initialise $r=0$, $A=\emptyset$ and $B=\emptyset$.
	
	The algorithm proceeds as follows.
	If $r=0$, then we choose an arbitrary vertex $H_1 \in \cH \setminus B$ and increase $r$ by one.
	If $r>0$, we let $\cH' = \cH \setminus (\{H_1,\dots,H_r\} \cup B)$ be the vertices that have not been used and that are not ``dead-ends".
	If $H_r \not\in A$, then, for all edges $H_rH'$ in $F^*$ with $H' \in \cH'$, we reveal the directed edge $(H_r, H')$ in $\bar F$ with probability $p/2$ independently of all other such
	edges.
	If possible, we pick one neighbour uniformly at random among all those that are successful, denote it by $H_{r+1}$ and increase $r$ by one.
	If none of them is successful, we add $H_{r}$ to $B$, $H_{r-1}$ to $A$ and decrease $r$ by one.
	If $H_r \in A$, then all the directed edges of the form $(H_r,H')$ with $H_rH' \in E(F^*)$ and $H' \in \cH'$ have already been revealed earlier in $\bar F$.
	If there is such an edge $(H_r,H')$ in $\bar F$, we let $H_{r+1}=H'$ and increase $r$ by one.
	Otherwise, we add $H_{r}$ to $B$, $H_{r-1}$ to $A$ and decrease $r$ by one.
	The algorithm stops if $r=(1-\tfrac{4k}{3k-1} \delta)n+1$ or when $|B| \ge \eps' n$, whichever happens first.
	We claim that the algorithm terminates and the latter does not happen.
	We remark that while the algorithm picks one neighbour uniformly at random among all those that are successful, the next claim would hold even if the choice of the neighbour was arbitrary.
	
	\begin{claim}
		\label{claim:directed_path}
		A.a.s.~the graph $\bar F$ contains a directed path $D$ on $t$ vertices (with $t$ being $(1-\tfrac{4k}{3k-1} \delta)n+1$, as defined above).
	\end{claim}
	
	\begin{claimproof}[Proof of Claim~\ref{claim:directed_path}]
		First, we observe that the algorithm terminates.
		Indeed, if $|B|< \eps'n$ then $|\cH'| \ge |\cH|-t-|B| > 2 \eps' m_0$ and with Claim~\ref{claim:min_degree3} there is at least one edge of $F^*$ from $H_r$ to $\cH'$.
		Secondly, we claim that a.a.s.~$|B| < \eps' n$.
		Assume that at some point we have $|B|=\eps' n$ and $r < t= (1-\tfrac{4k}{3k-1}\delta)n+1$.
		Since at least $|\cH|-r - |B| \ge \delta n/4$ vertices of $\cH$ are not covered by the path or a vertex from $B$, we can pick a set $\cH'$ of exactly $\delta n/4$ of them.	
		This implies that all edges from $B$ to $\cH'$ that are in $F^*$ have been revealed but none was successful to be present in $\bar F$.
		However, with Claim~\ref{claim:min_degree3}, the expected number of edges in $\bar F$ from $B$ to $\cH'$ with $|B|=\eps' n$ and $|\cH'|=\tfrac14 \delta n$ is $\tfrac18 p \eps' \delta n^2$ and by Chernoff's inequality (Lemma~\ref{lem:chernoff}) the probability that none of the edges in $F^*$ from $B$ to $\cH'$ appears in $\bar F$ is at most $2\exp(-\tfrac{1}{24} \eps' p \delta n^2) \le 2\exp(-2n)$.
		A union bound over the at most $2^{2n}$ choices for $B$ and $\cH'$ implies that the probability that there exist $B$ and $\cH'$ as above is $o(1)$.
		Therefore the process stops when $r=t$ and we obtain a directed path $D$ on $t$ vertices in $\bar F$.		
	\end{claimproof}
	
	\textbf{Preparing the final matching.}	
	The previous claim already guarantees a long directed path $D$ on $t$ vertices, but  to finish the proof later we need an additional property of $D$.
	An edge of $D$ corresponds to an edge $HH'$ of $F^*$ and, for each of them, there are many choices for a vertex $v \in V$ that turns this into the square of a path on $2k+1$ vertices.
	We will need to do this simultaneously for all edges of $D$ in the last step of the proof.
	However, before the last step, we have to cover the leftover of $U_1 \cup \dots \cup U_k$ and $\cH$, which will be possible by using some vertices of $V$.
	This will leave a subset $V' \subseteq V$ of size $t-1$ to match to the edges of $D$ in the last step, where we remark that the path $D$ has exactly $t-1$ edges.
	We now show that this is possible for any subset $V' \subseteq V$ of size $t-1$.
	We encode this task as follows.
	Given a vertex $v \in V$, we define $\cF_v$ to be the set of all pairs $(H,H') \in \cH^2$ such that both $H$ and $H'$ are adjacent to $v$ in $\cT_G(\cH,V)$; note that this means that $v$ is adjacent to all vertices in $V(H) \cup V(H')$ in the graph $G$.
	Then we define an auxiliary bipartite graph $\cF_D$ with partition $E(D)$ and $V$, where for $e \in E(D)$ and $v \in V$, the pair $ev$ is an edge of $\cF_D$ if and only if $e \in \cF_v$.
	
	\begin{claim}
		\label{claim:H_matching2}
		Assume that $D$ has $t-1$ edges.
		Then a.a.s.~for any $V' \subseteq V$ of size $t-1$ the graph $\cF_D[V',E(D)]$ contains a perfect matching.
	\end{claim}
	
	\begin{claimproof}[Proof of Claim~\ref{claim:H_matching2}]
		We denote by $\cD$ the event that we have a directed path $D$ with $t-1$ edges and assume that $\cD$ holds.
		It suffices to show that a.a.s.~for each $X \subseteq V$ of size at most $t-1$ we have $|\bigcup_{v \in X} N_{\cF_D}(v)| \ge |X|$.
		The claim follows then by Hall's condition.
		Let $X \subseteq V$ of size at most $t-1$ be given.
		
		First suppose that $|X| > (1-d^{k+2} 2^{-4k-9})n$ and let $e=(H,H')$ be any edge of $D$.
		Since $HH'$ is in particular an edge of $F^\ast$, the vertices $V(H) \cup V(H')$ have at least $d^{2k+2} 2^{-4k-8}n$ common neighbours in $V$.
		As $|V \setminus V'| \le 2 \delta n \le d^{2k+2} 2^{-4k-9} n$, $e$ has a neighbour in $X$ with respect to $\cF_D$.
		Since this is true for any edge $e$ of $D$, we conclude that $|\bigcup_{v \in X} N_{\cF_D}(v)|=t-1 \ge |X|$.
		
		Secondly, suppose that $|X|<\delta n$.
		Here it suffices to show that for any $v \in V$ we have $|N_{\cF_D}(v)| \ge \delta n$.
		Fix any $v \in V$ and let $\ell=d^{k+1} 2^{-2k-7} n$.
		As the pair $(\cH,V)$ is $(\eps',d^{k+1}2^{-2k-5})$-super-regular with respect to $\cT_G(\cH,V)$, we have that $v$ has degree at least $d^{k+1}2^{-2k-5} m_0$ into $\cH$ with respect to $\cT_G(\cH,V)$.
		Consider any point during the first $\ell$ steps of the algorithm, where $H_1,\dots,H_r$ is the current path and $H_r$ is not (yet) in $A$.
		This last assumption is crucial for the rest of the proof as it implies that the next vertex $H_{r+1}$ is chosen uniformly at random between the neighbours of $H_r$.
		We denote by $\cM$ the history of the algorithm until this point.
		Now we will look into the next two steps of the algorithm and estimate the probability that two vertices $H_{r+1}$ and $H_{r+2}$ are added to the path and the edge $(H_{r+1},H_{r+2})$ is in $\cF_v$.
		
		We write $H_r,H_{r+1} \not\in B$ for the event that both $H_r$ and $H_{r+1}$ do not get added to $B$ after two steps of the algorithm.
		Since $\PP[\cD]=1$ (this follows from Claim~\ref{claim:directed_path}) and
		\[\PP[(H_{r+1},H_{r+2}) \in \cF_v \land H_r, H_{r+1} \not \in B | \cM \land \cD] \ge \PP[(H_{r+1},H_{r+2}) \in \cF_v \land H_r, H_{r+1} \not \in B | \cM ] - o(1) \, , \]
		we do not need to condition on $\cD$ in the following calculation.
		
		We first bound the probability that $H_r$ gets added to $B$ after one step.
		For that we note that, with Claim~\ref{claim:min_degree3}, in each of the first $\ell$ steps we have at least $(1-2\eps')m_0 - \ell \ge \delta n/3$ neighbours of $H_r$ in $F^\star$ that are still available.
		Then $\PP[H_r \in B | \cM] \le (1-\tfrac p2)^{\delta n/3} \le \exp(- \tfrac 16 p \delta n) \le \eps'$.
		Therefore $\PP [H_r,H_{r+1} \not\in B | \cM ] \ge (1-\eps')^2$.
		Next, we want to bound the number of valid choices for $H_{r+1}$ in $\cH'$ that are neighbours of $v$ in $\cT_G(\cH,V)$.
		From the neighbours of $v$ in $\cT_G(\cH,V)$, we have to exclude those $H'$ such that $H_r H'$ is not an edge of $F^*$ and those $H'$ that are currently ``dead-ends": in the first case their number is at most $2\eps' m_0$ by Claim~\ref{claim:min_degree3}, in the second case their number is at most $|B| \le \eps'n$.
		Therefore there are at least $d^{k+1} 2^{-2k-5}m_0 - 2 \eps' m_0 - \eps' n - \ell \ge d^{k+1} 2^{-2k-6}m_0$ valid choices for $H_{r+1}$ in $\cH'$ that are neighbours of $v$ in $\cT_G(\cH,V)$.
		Repeating the same argument in the next step of the algorithm, there are at least $d^{k+1} 2^{-2k-6}m_0$ valid choices for $H_{r+2}$ in $\cH'$ that are neighbours of $v$ in $\cT_G(\cH,V)$.
		In particular for such choices of $H_{r+1}$ and $H_{r+2}$, the edge $(H_{r+1},H_{r+2})$ is in $\cF_v$.
		
		If $H_r$ (resp.~$H_{r+1}$) is not in $B$, then the vertex $H_{r+1}$ (resp.~$H_{r+2}$) is chosen uniformly at random from the at most $|\cH|=m_0$ available possibilities as we revealed the edges of $\bar F$ in each step separately.
		Therefore, together with the bound on $\PP[H_r,H_{r+1} \not\in B | \cM]$, we get
		\[\PP[(H_{r+1},H_{r+2}) \in \cF_v \land H_r, H_{r+1} \not \in B | \cM] \ge (1-\eps')^2 \frac{(d^{k+1} 2^{-2k-6} m_0)^2} {m_0^2} \ge d^{2k+2} 2^{-4k-13} \, .\]
		Crucially, this lower bound holds independently of the history $\cM$. 
		As among the first $\ell$ steps we can have at most $\eps' n$ many steps in which $H_r \in A$ and as the same lower bound holds when we additionally condition on $\cD$, this process dominates a binomial distribution with parameters $\ell-\eps'n$ and $d^{2k+2} 2^{-4k-13} n$.
		Therefore, even though the events are not mutually independent, we can use Chernoff's inequality (Lemma~\ref{lem:chernoff}) to infer that with probability at least $1-n^{-2}$ at least $d^{3k+3} 2^{-5k-21} n$ of these edges are in $\cF_v$.
		Some of these edges might not appear in the final path $D$, because of the ``dead-ends" and the backtracking of the algorithm, but their number is at most $\eps' n$.
		Thus we get that $|N_{\cF_D}(v)| \ge d^{3k+3} 2^{-5k-21} n - \eps'n \ge \delta n$ with probability at least $1-n^{-2}$.
		By applying the union bound over all $v \in V$, we obtain that a.a.s.~$|N_{\cF_D}(v)| \ge \delta n$ for all $v \in V$, as desired.

		Finally, assume that $\delta n \le |X| \le (1- d^{k+2}2^{-4k-9})n$.
		Here it suffices to show that, for any $X \subseteq V$ with $|X|=\delta n$, we have $|\bigcup_{v \in X} N_{\cF_D}(v)| \ge (1- d^{k+2}2^{-4k-9})n$.
		We use a similar argument as above, but this time we need to give more precise estimates.
		Consider any step of the algorithm where the current path is $H_1, \dots, H_r$ for some $r < t-1$, again with $H_r \not\in A$, and denote by $\cM$ the history of the algorithm until this point.
		As before we do not need to worry about conditioning on $\cD$.
		We want to bound the number of valid choices for $H_{r+1}$ in $\cH'$ that are neighbours of some $v \in X$ in $\cT_G(\cH,V)$.
		With Claim~\ref{claim:min_degree3}, there are at least $m_0-2\eps' m_0- \eps'n-r$ choices for $H_{r+1} \in \cH'$ such that $H_r H_{r+1}$ is an edge of $F^*$ and $H_{r+1} \not\in B$ (i.e. $H_{r+1}$ is not currently a ``dead-end").
		Using Claim~\ref{claim:degree_good3}, for at least $m_0-r-4 \eps'n$ of these choices, $H_{r+1}$ has degree at least $(1-4\eps')m_0$ in $F^*_X$.
		Then there are at least $m_0-r-4\eps' m_0-\eps'n \ge m_0 - r - 5 \eps' n$ choices for $H_{r+2}$, such that $(H_{r+1}, H_{r+2})$ is in $\bigcup_{v \in X} \cF_v$.
		
		On the other hand, there are at most $(m_0-r)$ choices for each of $H_{r+1}$ and $H_{r+2}$ and, as above, we have at least $(1-2\eps') m_0 - r - \eps'n \ge \delta n/3$ available neighbours of $H_r$ and $H_{r+1}$ and $\PP [H_r,H_{r+1} \not\in B | \cM ] \ge (1-\eps')^2$.
		Using that $m_0-r \ge \delta n/3$ we get
		\begin{align*} 
			\PP\left[ (H_{r+1},H_{r+2}) \in \bigcup_{v \in X} \cF_v \land H_r, H_{r+1} \not\in B \Big| \cM \right]
			\ge \frac{(1 - \eps')^2 (m_0-r-5 \eps'n)^2}{(m_0-r)^2} \ge 1- 5 \frac{\eps'}{\delta}\, .
		\end{align*}
		Again, as the lower bound holds independently of the history $\cM$ and as there are at most $\eps'n$ steps with $H_r \in A$ when we additionally condition on $\cD$, this process dominates a binomial distribution with parameters $t-1-\eps'n$ and $1-5 \tfrac{\eps'}{\delta}$.
		Therefore, the number $Y$ of these edges that are in $\bigcup_{v \in X} \cF_v$ is in expectation at least $(t-1-\eps'n) (1-5\tfrac{\eps'}{\delta}) \ge (1-\delta) t$
		and, using the more precise version of Chernoff's inequality (Lemma~\ref{lem:chernoff}) where $\delta-5\tfrac{\eps'}{\delta}$ plays the role of $\delta$, we get that
		\begin{align*}
			\PP\left[ Y  < ( 1- 2\delta) t \right] \le \PP\left[Y \le \EE[Y] - \left(\delta-5\frac{\eps'}{\delta}\right) (t-1 - \eps'n) \right] \\
			\le \exp \left( -D \left( (1-\delta) \big| \big| 1-5\frac{\eps'}{\delta} \right) (t-1 - \eps'n) \right) \\
			\le \exp \left( -\delta \left(\log \left( \frac{\delta^2}{5\eps'} \right) -2 \right)  (t-1-\eps'n)\right) \\
			\le \exp\left( -\delta \log\left(\frac{1}{\eps'}\right) \frac 12 t \right)\,.
		\end{align*}
		There are at most $\binom{n}{\delta n} \le (\tfrac{e}{\delta} )^{\delta n} \le  \exp (\delta \log(\tfrac{1}{\delta}) n) \le \exp(\delta \log(\tfrac{1}{\eps'}) \tfrac 14 t)$
		choices for $X$ and, thus, with the union bound over all these choices, we obtain that a.a.s.~for every $X \subseteq V$ with $|X|=\delta n$ at least $(1-2\delta) t$ of the edges are in $\bigcup_{v \in X} \cF_v$.
		At most $\eps' n$ of these edges ultimately do not belong to $D$ and putting this together we a.a.s.~have
		\begin{align*}
			\big|\bigcup_{x \in X} N_{\cF_D}(x)\big|  \ge (1- 2 \delta ) t -\eps'n \ge (1-4\delta) n \ge (1- d^{k+2}2^{-4k-9})n
		\end{align*}
		for any $X \subseteq V$ with $|X|=\delta n$, as wanted.
	\end{claimproof}

	Let $D$ be the directed path in $F^\ast$ given by Claim~\ref{claim:directed_path} and assume that the assertion of Claim~\ref{claim:H_matching2} also holds.
	We denote the first vertex of $D$ by $H_x'$ and the last by $H_y'$.
	Before dealing with the next step, we summarise what we have so far. 
	We have several copies of $H^{(k)}$: $H_x$, $H_y$ and those in $\cH$.
	The vertices $x$ and $x'$ (resp.~$y$ and $y'$) are adjacent in $G$ to all vertices of $H_x$ (resp.~$H_y$), and thus $(x,x')$ and $(y,y')$ can be end-tuples of the square of a Hamilton path we want to construct.
	Moreover, we have an ordering (given by the directed path $D$) of $t$ copies of $H^{(k)}$ in $\cH$, such that if $H=u_1,\dots,u_k$ and $H=u_1',\dots,u_k'$ appear consecutively, then $u_k u_1'$ is an edge of the random graph and all their vertices $u_1,\dots,u_k,u_1',\dots,u_k'$ have many common neighbours in $V$ in the graph $G$.
	
	\textbf{Covering the left-over vertices from $U_1 \cup \dots \cup U_k$.}
	Let $\cH'$ be the copies of $H^{(k)}$ in $\cH$ not used for the path $D$ and observe that $|\cH'|=|\cH \setminus V(D)|=m_0-t = s-2$, where the last equality follows from \eqref{eq:m_0}.
	Further observe that the number of vertices in $U_i$ not in any copy of $H^{(k)}$ in $\cH$ is $|U_i|-2-|\cH|= (1-\delta)n-m_0= \tfrac{s}{k}$, where the last equality follows from \eqref{eq:s/k}.
	Therefore we have exactly $s$ vertices in total in $U_1 \cup \dots \cup U_k$ to absorb; let $Z$ be the set of these vertices.
	We want to cover the $s$ vertices in $Z$ with the square of two paths connecting $H_x$ to $H_x'$ and $H_y'$ to $H_y$ respectively, while using all copies of $H^{(k)}$ in $\cH'$ and exactly $4s$ vertices from $V$(c.f.~Figure~\ref{fig_HC2_k3_2}).
	
	We start from connecting $H_x$ to $H_x'$, while absorbing two vertices of $Z$.
	We pick $H' \in \cH'$ and $z_x,z_x' \in Z$ such that the vertices in $H_x \cup \{z_x\} \cup H'$ and $H' \cup \{z_x' \} \cup H_x'$ each have at least $2\delta n$ common neighbours in $V$.
	This is possible by using Claim~\ref{claim:min_degree3} and the regularity property of $G$.
	Then, as we assumed that Lemma~\ref{lem:embedding} holds in $G_3$, there is a path on four vertices within each of these two sets of $2\delta n$ vertices, that gives the desired connection (c.f.~Figure~\ref{fig_HC2_k3_2}).
	
	Now we connect $H_y'$ to $H_y$, while absorbing the other $s-2$ vertices of $Z \setminus \{z_x,z_x'\}$.
	Let $H_1'=H_y$, $H_{s-1}'=H_y'$ and $H_2',\dots,H_{s-2}'$ be a labelling of the remaining $s-3$ copies of $H^{(k)}$ in $\cH' \setminus \{ H' \}$ such that for $j=1,\dots,s-2$ we have that all vertices in $V(H_j) \cup V(H_{j+1})$ have at least $d^{2k+2} 2^{-4k-8} n$ common neighbours in $V$ in $G$.
	This is possible by Dirac's Theorem and because, by Claim~\ref{claim:min_degree3}, for each $H \in \cH' \cup \{ H_y,H_y' \}$ all but $6 \tfrac{\eps'}{\delta}s$ choices $H' \in \cH'$ are such that the vertices $V(H) \cup V(H')$ have $d^{2k+2} 2^{-4k-8} n$ common neighbours.
	
	Next, we want to find a labelling $z_1,\dots,z_{s-2}$ of the vertices from $Z'=Z \setminus \{z_x,z_x' \}$ such that for $j=1,\dots,s-2$ the vertices $V(H_j) \cup \{z_j\} \cup V(H_{j+1})$ have at least $2\delta n$ common neighbours in $V$.
	This again follows easily from Hall's condition for perfect matchings and because, by Claim~\ref{claim:min_degree3}, for each $j=1,\dots,s-2$ all but $6 \eps's/\delta$ choices $z \in Z$ are such that $V(H_j) \cup V(H_{j+1}) \cup \{ z \}$ have $2\delta n$ common neighbours and, similarly, vice versa.
	Then by Lemma~\ref{lem:embedding}, a.a.s.~we can greedily choose a path on four vertices in the common neighbourhood of the vertices from $V(H_j) \cup V(H_{j+1}) \cup \{ z_j \}$ in $V$ for $j=1,\dots,s-2$, with all the edges coming from the random graph $G_3$.
	This again gives the desired connection (c.f.~Figure~\ref{fig_HC2_k3_2}).
	
	This completes the square of two paths from $H_x$ to $H_x'$ and from $H_y'$ to $H_y$.
	These two cover exactly $4s$ vertices of $V$.
	Therefore, there are precisely $|V|-4-4s=n-4s=t-1$ vertices of $V \setminus \{x,x',y,y'\}$ not yet covered by the square of a path, where we used \eqref{eq:t-1}; we let $V'$ be the set of such vertices.
	Observe that $|V'|=t-1=|E(D)|$.	
	
	\textbf{Finishing the square of a path.}
	We finish the proof by constructing the square of a path with $H_x'$ and $H_y'$ at the ends using precisely the vertices of $V'$ and the copies of $H^{(k)}$ that are vertices of $V(D)$.
	For this we use that by Claim~\ref{claim:H_matching2} there is a perfect matching in $\cF_D[V',E(D)]$.
	For $i=1,\dots,t-1$, let $v_i$ be the vertex of $V$ matched to the edge $(H_i,H_{i+1}) \in E(D)$ in $\cF_D$.
	With $H_i=u_1,\dots,u_k$ and $H_{i+1}=u_1',\dots,u_k'$, we then have that $v_i$ is incident to $u_{k-1}u_{k},u_1',u_2'$ by definition of $\cF_D$.
	This completes the construction of the square of the path with $H_x'$ to $H_y'$ at the ends.
	By adding the two connections found above from $H_x$ to $H_x'$ and from $H_y'$ to $H_y$ and the initial tuples $(x,x')$ and $(y,y')$, we get the square of a Hamilton path with end-tuples $(x,x')$ and $(y,y')$ as desired (c.f.~Figure~\ref{fig_HC2_k3}).
	This finishes the proof of the lemma.
\end{proof}

We end this section by giving the proof of Lemma~\ref{lem:bipartite}, that follows from Lemma~\ref{lem:multipartite2}, once we split appropriately the super-regular regular pair $(U,V)$ into two copies of super-regular $K_{1,2}$, both suitable for an application of Lemma~\ref{lem:multipartite2} with $k=2$.

\begin{proof}[Proof of Lemma~\ref{lem:bipartite}]
	Let $0<d<1$, choose $\delta'$ with $0 < \delta' \le \tfrac18 d$ and apply Lemma~\ref{lem:multipartite2} with $k=2$, $\delta'$ and $\tfrac 18 d$ to obtain $\delta_0,\delta,\eps'$ with $\delta' \ge \delta_0 > 2\delta>\eps'>0$ and $C'>0$.
	Then let $0 < \eps \le \tfrac18 \eps'$, $C \ge 4C'$ and $p \ge C n^{-1}$.
	Next let $U$ and $V$ be vertex-sets of size $|V|=n$ and $3n/4 \le |U|=m \le n$ and assume that $(U,V)$ is an $(\eps,d)$-super-regular pair.
	Let $(x,x')$ and $(y,y')$ be tuples from $V$ and $U$, respectively, such that they have $\tfrac12 d^2n$ common neighbours into the other set.
	We will reveal $G(V,p)$ and $G(U,p)$ both in two rounds as $G_1, G_3 \sim G(V,\tfrac 12 p)$,  and $G_2, G_4 \sim G(U,\tfrac 12 p)$.
	
	We partition $V$ into $V_1$, $U_2$, $W_2$ and $U$ into $V_2$, $U_1$, $W_1$ such that for $i=1,2$ the pairs $(U_i,V_i)$ and $(W_i,V_i)$ are $(\eps',\tfrac 18 d)$-super-regular pairs and $(1-\delta_0)|V_i| \le |U_i|=|W_i| \le (1-\delta) |V_i|$.
	Additionally, we require that $(x,x')$ is in $V_1$ and that $(y,y')$ is in $V_2$ and that they have at least $\tfrac12 (\tfrac d8)^2 n$ common neighbours in $U_1$, $W_1$ and in $U_2$, $W_2$, respectively.
	To obtain this we split the sets according to the following random distribution.
	We put any vertex of $V$ into each of $U_2$ and $W_2$ with probability $q_1$ and into $V_1$ with probability $1-2q_1$.
	Similarly, we put any vertex of $U$ into each of $U_1$ and $W_1$ with probability $q_2$ and into $V_2$ with probability $1-2q_2$.
	We choose $q_1$ and $q_2$ such that the expected sizes satisfy for $i=1,2$
	\[ \EE[|U_i|]=\EE[|W_i|] = \left( 1-\frac{\delta_0+\delta}{2} \right) \EE[|V_i|]. \]
	This is possible since such conditions give a linear system of two equations in two unknowns $q_1$ and $q_2$ and, as $3n/4 \le m \le n$, the solution satisfies $1/7 \le q_1,q_2 \le 3/7$.
	Then by Chernoff's inequality (Lemma~\ref{lem:chernoff}) and with $n$ large enough there exists a partition such that for $i=1,2$ we have that $|W_i|$, $|U_i|$ and $|V_i|$ are all within $\pm n^{2/3}$ of their expectation 
	and the minimum degree within both pairs $(U_i,V_i)$ and $(W_i,V_i)$ is at least a $d/4$-fraction of the other set.
	For $i=1,2$ we redistribute $o(n)$ vertices between $U_i$ and $W_i$ and move at most one vertex from or to $V_i$ to obtain
	\[ (1-\delta_0)|V_i| \le |U_i| = |W_i| \le (1-\delta)|V_i| \]
	with minimum degree within both pairs $(U_i,V_i)$ and $(W_i,V_i)$ at least a $d/8$-fraction of the other set.
	Moreover, for $i=1,2$, we can ensure that with $n_i=|V_i|-4$ we have $n_i-|U_i| \equiv -1 \pmod{5}$.
	
	From this we get that for $i=1,2$ the pairs $(U_i,V_i)$ and $(W_i,V_i)$ are $(\eps',\tfrac18 d)$-super-regular.
	With $G_1$ and $G_2$ we reveal random edges within $V_1$ and $V_2$ with probability $p/2$ to find tuples $(z,z')$ in $V_1$ and $(w,w')$ in $V_2$ such that together they give a copy of $K_4$ and $(z,z')$ and  $(w,w')$ have at least $\tfrac12 (\tfrac d8)^2 n$ common neighbours in $U_1$, $W_1$ and in $U_2$, $W_2$, respectively.
	Then we use Lemma~\ref{lem:multipartite2} and $G_3$, $G_4$ with $C n^{-1} \ge C' \min\{ |V_1|, |V_2|\}^{-1}$ to a.a.s.~find the square of a Hamilton path on $V_i$, $U_i$, $W_i$ for $i=1,2$ with end-tuples $(x,x')$, $(z,z')$ and $(y,y')$, $(w,w')$, respectively.
	Together with the edges between $(z,z')$ and $(w,w')$ this gives the square of a Hamilton path covering $U$ and $V$ with end-tuples $(x,x')$ and $(y,y')$.
\end{proof}

\section{Concluding remarks}
\label{sec:concluding}

We determined the exact perturbed threshold for the containment of the square of a Hamilton cycle in randomly perturbed graphs. 
As already pointed out in Section~\ref{sec:relatedWork}, much less is known for larger powers of Hamilton cycles and it would be interesting to investigate them further.
In this section we discuss the perturbed threshold $\hp$ for the containment of the third power of a Hamilton cycle.

Recall from Section~\ref{sec:relatedWork} that $\hp$ is already known for $\alpha = 0$ and $1/2 < \alpha \le 1$: we have $\hp=n^{-1/3}$ for $\alpha=0$, $\hp=n^{-1}$ for $1/2 < \alpha < 3/4$, and $\hp=0$ for $\alpha \ge 3/4$.
For $0<\alpha<1/2$, it is only known that there exists $\eta>0$ such that $\hp \le n^{-1/3-\eta}$.

We observe that we can obtain natural lower bounds by determining the sparsest possible structure that remains for $G(n,p)$ after mapping the third power of a Hamilton cycle into the complete bipartite graph $H_\alpha$ with parts of size $\alpha n$ and $(1-\alpha)n$.
When $\alpha= 1/4$, this structure is essentially the square of a Hamilton cycle on $3n/4$ and is obtained by mapping every fourth vertex of the third power of a Hamilton cycle into the smaller part of $H_{1/4}$. 
Therefore, in order for $H_{1/4} \cup G(n,p)$ to contain the third power of a Hamilton cycle, we need $G(n,p)$ to contain the square of a Hamilton cycle on $3n/4$ vertices.
This gives $\hat{p}_{1/4}(n) \ge n^{-1/2}$ and we believe this is actually tight.
\begin{conjecture}
	\label{conj:third}
	The perturbed threshold $\hat{p}_{\alpha}(n)$ for the containment of the third power of a Hamilton cycle satisfies $\hat{p}_{1/4}(n) = n^{-1/2}$.
\end{conjecture}
However, as discussed in the introduction, finding the square of a Hamilton cycle at this probability is a particularly challenging problem and, additionally, it is not possible to first embed small parts arbitrarily and then connect them, as we do in the proof of our main result.

For each of the ranges $0 < \alpha < 1/4$ and $1/4 < \alpha < 1/2$, it is not clear whether to expect a similar `jumping' behaviour as the one proved for the square of a Hamilton cycle in the range $0<\alpha< 1/2$.
We can obtain natural lower bounds similarly as we did for $\alpha=1/4$.
Again, the sparsest structure that remains for $G(n,p)$ is obtained essentially by mapping every $1/\alpha$-th vertex of $C_n^3$ into the smaller part of $H_\alpha$.
However, in contrast to Section~\ref{subsec:lower_bounds}, the threshold for the appearance of this structure in $G(n,p)$ is not determined by the second or third power of a short path.
For example, when $\alpha = 1/5$, by doing as described above and mapping every fifth vertex of $C_n^3$ into the smaller part of $H_\alpha$, we are left with copies of $P_4^3$, which are connected by three edges, cyclically.
By a first moment argument, the threshold for this structure in $G(n,p)$ is at least $n^{-4/9}$ and thus $\hat{p}_{1/5}(n) \ge n^{-4/9}$, which is larger than the threshold for a $P_4^3$-factor in $G(n,p)$.
More generally, we get for $0 \le \alpha \le 1/4$ that $\hp \ge n^{-(1-\alpha)/(3-6\alpha)}$, and for $1/4 \le \alpha \le 1/2$ that $\hp \ge n^{-(1-\alpha)/(5/2-4\alpha)}$.
Note that these lower bounds match the already known value of $\hat{p}_\alpha(n)$ for $\alpha=0$ and coincide for $\alpha=1/4$, in support of Conjecture~\ref{conj:third}. 
Moreover, if it is tight, this `continuous' perturbed threshold would be an exciting new behaviour.

\begin{question}
	Does the perturbed threshold $\hat{p}_{\alpha}(n)$ for the containment of the third power of a Hamilton cycle satisfy
	\[
	\hat{p}_\alpha(n) =  
	\begin{cases}
		n^{-(1-\alpha)/(3-6\alpha)} & \text{if } \alpha\in [0,\tfrac 14)\,, \\
		n^{-(1-\alpha)/(5/2-4\alpha)} & \text{if } \alpha \in [\tfrac 14, \tfrac 12]\,?
	\end{cases}
	\]	
\end{question}

For neither of the two ranges of $\alpha$ this lower bound seems to be attainable with our approach, because at this probability there is no small structure that we can find and then connect into the third power of a Hamilton cycle.
Taking again the example of $\alpha= 1/5$, our lower bound for the perturbed threshold is at $n^{-4/9}$, but at this probability it is not possible to first find the copies of $P_4^3$ arbitrarily and then connect them.
However we believe our methods can give the following.
We map every ninth and tenth vertex of $C_n^3$ into the smaller part of $H_{1/5}$, which leaves copies of $P_{8}^3$ connected by single edges, cyclically.
We expect that it is possible to extend our argument to this set-up, but this would only imply $\hat{p}_{1/5}(n) \le n^{-7/18}$.
More generally, we believe we can show that $\hat{p}_\alpha(n) \le n^{-(2k-1)/(6k-6)}$ for $\alpha \in \left( \tfrac{1}{k+1}, \tfrac{1}{k} \right)$ and $k \ge 4$, which is the threshold for the containment of linearly many copies of $P^3_{2k}$ in $G(n,p)$.
Similarly, we believe we can show that $\hat{p}_\alpha(n) \le n^{-(3k-1)/(6k-3)}$ for $\alpha \in \left( \tfrac{k+1}{4k+1}, \tfrac{k}{4k-3} \right)$ and $k \ge 1$, which is the threshold for the containment of linearly many copies of $P^2_{3k}$ in $G(n,p)$.
This would improve on the bounds obtained in~\cite{bottcher2017embedding}, but it is still far from the lower bounds discussed above.
Note that the latter bound tends to $n^{-1/2}$, as $k$ tends to infinity (and thus $\alpha$ tends to $\tfrac 14$), supporting again Conjecture~\ref{conj:third}.

\section*{Acknowledgements}
We are grateful to an anonymous referee for their insightful and valuable comments.

\bibliographystyle{amsplain}
\bibliography{references}

\appendix

\section{Supplementary proofs}
\label{sec:appendix}

In this section we will prove Lemmas~\ref{lem:paths2},~\ref{lem:sublinear_square_paths},~\ref{lem:auxF} and~\ref{lem:H_copies}.
The proof of Lemma~\ref{lem:sublinear_square_paths} generalises the argument in~\cite[Theorem~2.4]{triangle_paper}, and the proof of Lemma~\ref{lem:auxF}  only requires basic regularity type arguments.
The remaining two lemmas concern random graphs and their proofs are based on Janson's inequality (see e.g.~\cite[Theorem 2.18]{JLR}).
\begin{lemma}[Janson's inequality]
	\label{lem:janson}
	Let $p \in (0,1)$ and consider a family $\{ H_i \}_{i \in \cI}$ of subgraphs of the complete graph on the vertex set $[n]=\{1,\ldots,n\}$. For each $i \in \cI$, let $X_i$ denote the indicator random variable for the event that $H_i \subseteq \Gnp$ and, write $H_i \sim H_j$ for each ordered pair $(i,j) \in \cI \times \cI$ with $i \neq j$ if $E(H_i) \cap E(H_j) \not= \emptyset$.
	Then, for $X = \sum_{i \in \cI} X_i$, $\mathbb{E}[X] = \sum_{i \in \cI} p^{e(H_i)}$,
	\begin{align*}
		\Delta [X] = \sum_{H_i \sim H_j} \mathbb{E}[X_i X_j] = \sum_{H_i \sim H_j} p^{e(H_i) + e(H_j) - e(H_i \cap H_j)}
	\end{align*}
	and any $0 < \gamma < 1$ we have
	\begin{align*}
		\mathbb{P} [X \le (1-\gamma) \mathbb{E}[X]] \le \exp \left(-\frac{\gamma^2 \mathbb{E}[X]^2}{2(\mathbb{E}[X] + \Delta[X])} \right).
	\end{align*}
\end{lemma}

\begin{proof}[Proof of Lemma~\ref{lem:paths2}]
	Let $s \ge 1$ and $k \ge 2$ be integers and $0 < \eta \le 1$.
	Moreover let $C\ge 2^{6} (sk)^{2sk} \eta^{-2} $ and $p \ge C n^{-(k-1)/(2k-3)}$.
	
	Let $V$ be a vertex set of size $n$, $V_i$ be a subset of $V$ for $i=1,\dots,s$ and $H$ be a collection of pairwise distinct tuples from $\prod_{i=1}^s V_i^k$. 
	Let $W_i \subseteq V_i$ for each $i=1,\dots,s$ and assume $H'=H \cap \prod_{i=1}^s W_i^k$ has size at least $\eta n^{sk}$.
	Since the number of tuples from $\prod_{i=1}^s V_i^k$ which contain a vertex more than once is $O \left( n^{sk-1} \right)$, there are at least $\frac{\eta}{2} n^{sk}$ tuples of $H'$ such that their vertices are pairwise distinct.
	We restrict our analysis to the set of those tuples, which, abusing notation, we still denote by $H'$.
	
	For each tuple $(v_{i,j}:1 \le i \le s, 1 \le j \le k)$ in $H'$, we consider the graph with vertex set $V$ and the following edges.
	For $i=1,\dots,s$ we have the square of the path on $v_{i,1},\dots,v_{i,k}$ and for $i=1,\dots,s-1$ we have the edge $v_{i,k}v_{i+1,1}$.
	This gives a family $\{ H_i \}_{i \in \left[ |H'| \right]}$ of graphs with vertex set $V$ and, using the same notation as in Lemma~\ref{lem:janson}, a collection of random variables $\{ X_i \}_{i \in \left[ |H'| \right]}$.
	Note that for each $i=1,\dots,s$, we have $e(H_i)=s(2k-3)+(s-1)=2s(k-1)-1$ and thus, for $X = \sum_{ i \in \left[ |H'| \right]} X_i$, we have $\EE[X]=|H'|p^{2s(k-1)-1} \ge \sqrt{C} n$.
	To compute the quantity $\Delta[X]= \sum_{H_i \sim H_j} p^{e(H_i) + e(H_j) - e(H_i \cap H_j)}$, we split the sum according to the number of vertices in the intersection $E(H_i \cap H_j)$.
	Suppose $H_i$ and $H_j$ intersect in $m$ vertices.
	Then $2 \le m \le sk-1$ and the largest size $\tilde{e}(m)$ of the intersection $E(H_i \cap H_j)$ is
	\[
	\tilde{e}(m) = 
	\begin{cases}
		\frac{m}{k} (2k-3) + \frac{m}{k} -1, & \text{if } m \equiv 0 \pmod{k} \\
		\lfloor \frac{m}{k} \rfloor (2k-3) + \lfloor \frac{m}{k} \rfloor, & \text{if } m \equiv 1 \pmod{k} \\
		\lfloor \frac{m}{k} \rfloor (2k-3) + \lfloor \frac{m}{k} \rfloor + 2 \left (m-k\lfloor \frac{m}{k} \rfloor \right) -3, & \text{otherwise.}
	\end{cases}
	\]
	In particular, observing that $\tilde{e}(m)=2m-3$ if $m < k$ (as $m \ge 2$ we are in the third case) and $\tilde{e}(m) \le 2m-2 \frac{m}{k} -1$ if $m \ge k$ (the inequality follows from $\lfloor \frac{m}{k} \rfloor \ge \frac{m}{k} -1$), we can conclude that $p^{-\tilde{e}(m)} n^{-m}  \le C^{-1} n^{-1}$ for each $2 \le m \le sk-1$.
	Therefore,
	\begin{align*}
		\Delta[X] & \le \sum_{m=2}^{sk-1} m! \binom{sk}{m}^2 n^{2sk-m}  p^{[2s(k-1)-1]+[2s(k-1)-1]-\tilde{e}(m)} \le  \sum_{m=2}^{sk-1} (sk)^{2m}\frac{ n^{2sk-m} }{|H'|^2} \EE^2[X] p^{-\tilde{e}(m)} \\
		&\le 4 (sk)^{2sk-2} \eta^{-2} \sum_{m=2}^{sk-1} \EE^2[X] p^{-\tilde{e}(m)} n^{-m} \le 4(sk)^{2sk-1} \eta^{-2} C^{-1} \EE^2[X] n^{-1} \le \tfrac 18 s^{-1} \EE^2[X] n^{-1} ,
	\end{align*}
	where in the first inequality we used that there are at most $m! \binom{sk}{m}^2 n^{2sk-m}$ choices for $H_i$ and $H_j$ intersecting in $m$ vertices, in the second we used $\EE[X]=|H'|p^{2s(k-1)-1} \ge \sqrt{C} n$, and in the third we used $n^{sk}/|H'| \le 2/\eta$.
	Then with Lemma~\ref{lem:janson} applied with $\gamma=2^{-1/2}$, we get that the probability that none of the graphs of the family $\{ H_i \}_{i \in \left[ |H'| \right]}$ appears in $G(n,p)$ is bounded from above by 
	\begin{align*}
		\exp \left(- \frac{\EE^2[X]}{4(\EE[X]+\Delta[X])} \right) & \le \exp \left(- \tfrac 18 \min\{ \EE[X] , \EE^2[X]/\Delta[X] \} \right) \\
		& \le \exp \left( - \tfrac 18 \min\{ \sqrt{C} n , 8s n \} \right)  \le \exp(-sn)\, .
	\end{align*}
	We can conclude with a union bound over the at most $2^{sn}$ choices for the $s$ subsets $W_i$ with $i=1,\dots,s$ that the lemma holds.
\end{proof}

\begin{proof}[Proof of Lemma~\ref{lem:sublinear_square_paths}]
	Let $k \ge 2$ and $t \ge 1$ be integers.
	For $k=2$ the result follows from the proof of~\cite[Theorem~2.4]{triangle_paper} with slight modifications, so we can assume $k \ge 3$.
	We let $0<\gamma< (16 k t)^{-1}$ and $C \ge 2^{k+9} kt$.
	Further let $p \ge C (\log n)^{1/(2k-3)} n^{-(k-1)/(2k-3)}$, $0 \le m \le \gamma n$, and let $G$ be an $n$-vertex graph with vertex set $V$,  minimum degree $\delta(G) \ge m$ and maximum degree $\Delta(G) \le \gamma n$.
	
	We distinguish two cases.
	If $m \le (\log n)^{2/(2k-3)} n^{(2k-4)/(2k-3)}$, we only need Janson's inequality and we will greedily find $tm+t$ copies of $P_{k+1}^2$, using only edges from the random graph $G(n,p)$.
	Let $V' \subseteq V$ be the set of vertices used in this greedy construction.
	As long as we have not found $tm+t$ copies of $P_{k+1}^2$, we have $|V'| \le (tm+t)(k+1)$ and thus $|V \setminus V'| \ge n/2$.
	We let $\{ H_i\}_{i \in \cI}$ be the family of copies of $P_{k+1}^2$ with vertices in $V \setminus V'$ and note $|\cI| \ge 2^{-k-2} n^{k+1}$.
	Then, using the notation of Lemma~\ref{lem:janson}, we observe that the expected number of these copies appearing as subgraphs of $G(n,p)$ is 
	\[\EE[X]=|\cI| p^{2k-1} \ge 2^{-k-2} n^{k+1} p^{2k-1} \ge C (\log n)^{(2k-1)/(2k-3)} n^{(2k-4)/(2k-3)} \ge 32kt m \log n \, . \]
	On the other hand, we have
	\begin{align*}
		\Delta[X] &= \sum_{H_i \sim H_j} p^{e(H_i) + e(H_j) - e(H_i \cap H_j)} \le \sum_{r=2}^{k} O(n^{2(k+1)-r} p^{2(2k-1)-(2r-3)}) \\
		& \le \EE^2[X] \sum_{r=2}^{k} O \left( p^{3-2r} n^{-r}  \right) \le  \EE^2[X] o \left( n^{-1} \right) \, ,
	\end{align*}
	where in the first inequality we split the sum according to the value of $r=v(H_i \cap H_j)$ and used that then $e(H_i \cap H_j) \le 2r-3$.
	Then with Lemma~\ref{lem:janson} we get that the probability that there is no copy of $P_{k+1}^2$ is bounded from above by $\exp(- \EE[X]/8)\le n^{-4ktm}$.
	We conclude with a union bound over the at most $\binom{n}{(k+1)(tm+t)} \le n^{3ktm}$ possible choices for $V'$ that we can a.a.s.~find $tm+t$ copies of $P^2_{k+1}$ in $G(n,p)$.
	
	For $m \ge (\log n)^{2/(2k-3)} n^{(2k-4)/(2k-3)}$ we need to use the edges of $G$.
	We will find copies of $P^2_{k+1}$, where all edges incident to one vertex come from $G$ and the remaining edges come from $G(n,p)$, where we will need to distinguish between $k=3$ and $k \ge 4$.
	First, we greedily obtain a spanning bipartite subgraph $G' \subseteq G$ of minimum degree $\delta(G') \ge m/2$ by taking a partition of $V(G)$ into sets $A$ and $B$ such that $e_G(A,B)$ is maximised and letting $G'=G[A,B]$.
	Indeed, a vertex of degree less than $m/2$ can be moved to the other class to increase $e_G(A,B)$.
	W.l.o.g.~we assume $|B|\ge n/2\ge |A|$.
	Moreover, we have $|A| \ge m/(4 \gamma)$, as otherwise with $e(A,B) \ge nm/4$ there is a vertex of degree at least $\gamma n$, a contradiction.
	
	Then we observe that given any sets $A' \subseteq A$ and $B' \subseteq B$ such that $|A'| \le m/(16 \gamma)$ and $|B'| \le n/4$, we also have $e(A \setminus A', B \setminus B')\ge nm/16$.
	Otherwise, from $e(A,B \setminus B') \ge |B \setminus B'| m/2 \ge nm/8$, it would follow that $e(A',B\setminus B') \ge nm/16$ and thus, since $|A'| \le m/(16 \gamma)$, we would have a vertex of degree at least $\gamma n$ in $A'$, a contradiction to the maximum degree of $G$.
	
	We will greedily find $tm+t$ copies of $P_{k+1}^2$ with one vertex in $A$ and $k$ vertices in $B$.
	Let $A' \subseteq A$ and $B' \subseteq B$ be the set of vertices used in this greedy construction.
	As long as we have not found $tm+t$ copies of $P_{k+1}^2$, we have $|A'| \le tm+t \le m/(16 \gamma)$ and $|B'| \le k(tm+t) \le n/4$, and thus $e(A \setminus A', B \setminus B')\ge nm/16$.
	Therefore, using that $|A|\le n/2$, there is a vertex $v \in A \setminus A'$ with degree at least $m/8$ into $B \setminus B'$.
	We let $B^*$ be a set of $m/8$ neighbours of $v$ in $B \setminus B'$.
	
	When $k=3$, we will find a path on three vertices in $B^*$ in the random graph, which will give, together with the three edges of $G$ between $v$ and those vertices, a copy of $P_4^2$.
	We argue as follows.
	If such a path does not appear, then there are less than $m$ edges of $G(n,p)$ in $B^*$.
	However the expected number of random edges within $B^*$ is at least $p\binom{m/8}{2} \ge 8kt m \log n$ and therefore, by Lemma~\ref{lem:chernoff}, with probability at least $1-n^{-4kt m}$ there are more than $m$ edges of $G(n,p)$ in $B^*$.
	We conclude by union bound over the at most
	\[ \binom{|A|}{tm+t} \binom{|B|}{k(tm+t)} \le n^{3ktm}\]
	choices for $A'$ and $B'$.
	
	For $k \ge 4$ we let $B_1, \dots, B_4$ be pairwise disjoint sets of size $m/32$ in $B^*$.
	Moreover, we let $B_5,\dots,B_k$ be pairwise disjoint sets of size $n/(4k)$ in $B \setminus B'$, each disjoint from $B_1, \dots, B_4$.
	This is possible as $|B \setminus B'| \ge n/4$.
	
	\begin{claim}
		\label{claim:almost_path}
		With probability at least $1-n^{-\omega(m)}$ there exists vertices $b_1,\dots,b_k$ with $b_i \in B_i$ for $i=1,\dots,k$ such that in $G(n,p)$ we have the edges $b_ib_{i+1}$ for $i=1,\dots,k-1$ and $b_{i}b_{i+2}$ for $i=3,\dots,k-2$.
	\end{claim}
	
	Observe that, together with $v$ and the edges $vb_i$ for $i=1,\dots,4$, this gives a copy of $P_{k+1}^2$ with vertices $b_1,b_2,v,b_3,\dots,v_k$.
	As there are at most $n^{O(m)}$ choices for $A'$ and $B'$, by a union bound and Claim~\ref{claim:almost_path}, we a.a.s.~find $tm+t$ copies of $P_{k+1}^2$.
	It remains to prove the claim.
	
	\begin{claimproof}[Proof of Claim~\ref{claim:almost_path}]
		
		We denote by $\{H_i \}_{i \in \cI}$ the graphs on $k$ vertices $b_1,\dots,b_k$ with $b_i \in B_i$ for $i=1,\dots,k$ and edges $b_ib_{i+1}$ for $i=1,\dots,k-1$ and $b_{i}b_{i+2}$ for $i=3,\dots,k-2$.
		Then, using the notation of Lemma~\ref{lem:janson}, the expected number of those graphs appearing in $G(n,p)$ is 
		\begin{align*}
			\EE[X]&=|\cI| p^{2k-5} \ge \Omega( m^{4} n^{k-4} p^{2k-5})  \\
			&\ge \Omega \left(m (\log n)^{\tfrac{6}{2k-3}+\tfrac{2k-5}{2k-3}} n^{(k-4)+\tfrac{3(2k-4)}{2k-3} - \tfrac{(2k-5)(k-1)}{2k-3}} \right) = \omega(m \log n) \, ,
		\end{align*}
		where we used the bounds on the sizes of the sets $B_i$ for $i=1,\dots,k$ in the first inequality and the bounds on $m$ and $p$ in the second inequality.
		On the other hand we get 
		\begin{align*}
			\Delta[X] &= \sum_{H_i \sim H_j} p^{e(H_i) + e(H_j) - e(H_i \cap H_j)} \le \sum_{r,s} O(m^{8-r} n^{2k-8-s} p^{2(2k-5)-(2s+\min\{2r-3,r-1 \})}) \\
			&\le \EE^2[X] \sum_{r,s} O(m^{-r}n^{-s} p^{-2s-\min\{2r-3,r-1 \}}) \le \EE^2[X] O(n^{-2}p^{-1}) \le \EE^2[X] o(n^{-1}) \, ,
		\end{align*}
		where we split the sum according to the value of $r$ and $s$, with $0 \le r \le 4$, $0 \le s \le k-4$ and $2 \le r+s \le k-1$, where $r$ and $s$ are the number of common vertices of $H_i$ and $H_j$ in $B_1,\dots,B_4$ and $B_5,\dots,B_k$, respectively.
		In the first inequality we used that $e(H_i \cap H_j) \le 2s+\min\{2r-3,r-1 \}$, and in the third inequality we used that $O(m^{-r}n^{-s} p^{-2s-\min\{2r-3,r-1 \})})$ is maximised for $r=0$ and $s=2$ with the given bounds on $m$ and $p$.
		The claim follows by Lemma~\ref{lem:janson}, as in the application above.
	\end{claimproof}
	
\end{proof}

\begin{proof}[Proof of Lemma~\ref{lem:auxF}]
	To  prove~\ref{claim:min_degree}, without loss of generality, it suffices to show that the degree of every vertex in $U_1$  is at least $(1-h\eps)m^{h-1}$.
	Fix any $u_1 \in U_1$ and set $N_1=N_G(u_1, V)$. Notice that since $(V,U_1)$ is $(\eps,d)$-super-regular, we have $|N_1| \ge d |V|\ge \eps|V|$.
	Since $(V,U_2)$ is $(\eps,d)$-super-regular, there are at least $(1-\eps)m$ vertices $u_2 \in U_2$ such that the set $N_2=N_G(u_2,N_1)$ of neighbours of $u_2$ in $N_1$ has size at least $(d-\eps)|N_1|\ge (d-\eps)d|V| \ge\eps |V|$.
	Continuing in the same way, by applying Lemma~\ref{lem:MDL} to the $(\eps,d)$-super-regular pair $(V,U_j)$ for $j=3,\dots,h$, we get that there are at least $((1-\eps)m)^{j-1}$ choices of $(u_2,\dots,u_j) \in U_2 \times \dots \times U_j$ such that the vertices $u_1,u_2,\dots,u_j$ have at least $(d-\eps)^{j-1}|N_1|\ge (d-\eps)^{j-1}d|V|\ge (d-\eps)^{h-1}d|V| \ge\eps |V|$ common neighbours in the set $V$.
	Since $(d-\eps)^{h-1}|N_1| \ge \tfrac12 d^{h}n$ and $((1-\eps)m)^{h-1} \ge (1-h\eps)m^{h-1}$, the first part of the lemma follows.
	
	Without loss of generality, it suffices to prove~\ref{claim:degree_god} for $U_1$.
	If $|X| \ge 2 \eps nd^{1-h}$, then, by applying Lemma~\ref{lem:MDL}, for all but at most $\eps m$ vertices $u_1 \in U_1$, the set $N_1=N(u_1,X)$ of neighbours of $u_1$ in $X$ is of size at least $(d-\eps)|X|$.
	Fix any such $u_1$ and proceed in the same way as in the proof of~\ref{claim:min_degree}.
	We get that there are at least $((1-\eps)m)^{h-1} \ge (1-h\eps)m^{h-1}$ choices of $(u_2,\dots,u_h) \in U_2 \times \dots \times U_h$ such that the vertices $u_1,u_2,\dots,u_h$ have at least $(d-\eps)^{h-1}|N_1| \ge (d-\eps)^h |X| \ge \tfrac12 d^h|X|$ common neighbours in the set $X$, and the second part of the lemma follows.
\end{proof}

\begin{proof}[Proof of Lemma~\ref{lem:H_copies}]
	Given any graph $H$ on $h \ge 2$ vertices and any $\delta>0$, we fix $\eps>0$ with $\eps < 2^{-4h-24} h^{-8} \delta^{4h}$, $\delta'=2^{-3} h^{-1} \eps^{1/4}$ and $C$ large enough for the inequalities indicated below to hold.
	Observe that the maximum degree of $F$ is $m^{h-1}$ and, by Lemma~\ref{lem:auxF}\ref{claim:min_degree}, the minimum degree of $F$ is at least $(1-h\eps)m^{h-1}$.
	Therefore
	\[\EE[e(\tilde F)] = e(F) p^{e(H)} = (1 \pm h\eps) m^h p^{e(H)} \]
	and
	\begin{align*} 
		\Var[e(\tilde F)] &= O_{h,\delta} \left( \sum_{H' \subseteq H, e(H')>0} m^{2v(H) - v(H')} \left( p^{2e(H)-e(H')} - p^{2e(H)}  \right) \right) \\
		& = O_{h,\delta} \left( \EE[e(\tilde F)]^2 \sum_{H' \subseteq H, e(H')>0}  m^{-v(H')} p^{-e(H')} \right) \\
		&=  O_{h,\delta} \left( \EE[e(\tilde F)]^2 \sum_{H' \subseteq H, e(H')>0}  n^{-v(H')} p^{-e(H')} \right) = O_{h,\delta} \left( \EE[e(\tilde F)]^2 C^{-1} n^{-1} \right)\, ,
	\end{align*}
	where we used that $n^{-v(H')}p^{-e(H')} \le C^{-e(H')} n^{-v(H')+e(H')/m_1(H)} \le C^{-e(H')} n^{-1}$ in the last step.
	Using Chebyshev's inequality, we have
	\begin{align*}
		\PP\big[ e(\tilde F) \not= (1 \pm \eps) \EE[e(\tilde F)]\big] = O_{h,\delta,\eps} \left( \frac{\Var[e(\tilde F)]}{\EE[e(\tilde F)]^2}  \right)
		= O_{h,\delta,\eps}(C^{-1} n^{-1})  \, , 
	\end{align*}
	and thus a.a.s.
	\begin{equation}
		\label{eq_e(F)}
		e(\tilde F) = (1 \pm \eps) \EE[e(\tilde F)] = (1 \pm \eps) (1 \pm h\eps) m^h p^{e(H)}.
	\end{equation}
	Similarly as above, given $U_i' \subseteq U_i'$ of size at least $\delta m$ for $i=1,\dots,h$, we have
	\[\EE[e(\tilde F')] =e(F') p^{e(H)} = (1 \pm h\tfrac{\eps}{\delta}) \prod_{i=1}^h |U_i'| p^{e(H)} = \Omega_{h,\delta,\eps}(n^h p^{e(H)}) = \Omega_{h,\delta,\eps}(C n) \]
	and $\Delta[e(\tilde F')] = O_{h,\delta,\eps}(\EE[e(\tilde F')]^2 C^{-1} n^{-1})$.
	Then with Lemma~\ref{lem:janson} we have
	\begin{align*}
		\PP\big[e(\tilde F') < (1-\eps) \EE[e(\tilde{F'})]\big] \le
		\exp\left(- \frac{\eps^2 \EE[e(\tilde F')]^2}{2\Delta[e(\tilde F')]+2\EE[e(\tilde F')]}\right) \le \exp (-h n)\, ,
	\end{align*}
	where the last inequality holds for large enough $C$.
	We conclude with a union bound that a.a.s.
	\begin{equation}
		\label{eq_e(tildeF)}
		e(\tilde F') \ge  (1-\eps) \EE[e(\tilde{F'})] \ge (1-\eps) (1 - h\tfrac{\eps}{\delta}) \prod_{i=1}^h |U_i'| p^{e(H)} \ge (1-\sqrt{\eps}) \prod_{i=1}^h |U_i'| p^{e(H)}
	\end{equation}
	for all choices of $U_i' \subseteq U_i'$ of size at least $\delta m$ for $i=1,\dots,h$ and using the choice of $\eps$. 
	This proves the lower bound of~\eqref{eq:copies_F}.
	Note that~\eqref{eq_e(F)} and~\eqref{eq_e(tildeF)} hold also with $\delta$ replaced by $\delta'$.
	
	Next we upper bound $e(\tilde F')$ by taking $e(\tilde F)$ and subtracting those edges of $\tilde F$ that are not in $\tilde F'$, i.e.~the edges that contain at least one vertex $v_i$ that belongs to $U \setminus U_i'$.
	We only need a lower bound on their number and we will see that it is enough to lower bound those for which $|U_i \setminus U_i'| \ge \delta' m$ (which can be done using~\eqref{eq_e(tildeF)}), and simply ignore the others.
	For this we let $J \subseteq [h]$ be the set of those indices $j \in [h]$ such that $|U_j'| \le (1-\delta')m$, and for any $\emptyset \neq I \subseteq J$ we let $F_I$ be the subgraph of $F$ induced by the sets $U_i \setminus U_i'$ for $i \in I$ and $U_i'$ for $i \not\in I$.
	If $J = \emptyset$, then the inequality $e(\tilde F') \le e(\tilde F)$ already gives the desired upper bound on $e(\tilde F')$.
	Otherwise, using~\eqref{eq_e(F)} and~\eqref{eq_e(tildeF)}, we get
	\[e(\tilde F') \le e(\tilde F) - \sum_{\emptyset \neq I \subseteq J} e(\tilde F_{I})
	\le (1+\eps) \EE[e(\tilde{F})] - \sum_{\emptyset \neq I \subseteq J} (1-\eps) \EE[e(\tilde{F_I})] \,, \]
	that we can further upper bound by
	\begin{align*} 	
		(1+\eps) (1 + h\eps) \prod_{i=1}^h |U_i| p^{e(H)} - \sum_{\emptyset \neq I \subseteq J} \left[ (1-\eps) (1 - h\tfrac{\eps}{\delta'}) \prod_{i \in I} |U_i \setminus U_i'| \prod_{i \not\in I} |U_i'| p^{e(H)} \right] \\
		\le \prod_{j \not\in J} |U_j| \prod_{j \in J} |U_j'| p^{e(H)} + 2 h \eps \prod_{i=1}^h |U_i|  p^{e(H)} + 2 h \frac{\eps}{\delta'} \sum_{\emptyset \neq I \subseteq J} \left( \prod_{i \in I} |U_i \setminus U_i'| \prod_{i \not\in I} |U_i'| p^{e(H)} \right) \\
		\le  (1-\delta')^{|J|-h} \prod_{i=1}^h |U_i'| p^{e(H)} + 2 h \eps \delta^{-h} \prod_{i=1}^h |U_i'|  p^{e(H)} + 2 h \frac{\eps}{\delta'} 2^{|J|} \delta^{-|J|}  \prod_{i=1}^h |U_i'| p^{e(H)} \\
		\le \left(1+ 2 h \delta' + 2 h \eps \delta^{-h} + 2^{h+1} h \delta^{-h} \frac{\eps}{\delta'} \right) \prod_{i=1}^h |U_i'| p^{e(H)} \le (1+\sqrt{\eps} ) \prod_{i=1}^h |U_i'| p^{e(H)} \, .
	\end{align*}
	To get the second line, we used $(1+\eps)(1+h\eps) \le 1+2h\eps$, $(1-\eps)(1-h\tfrac{\eps}{\delta'}) \ge 1-2h\tfrac{\eps}{\delta'}$ and 
	\[ \prod_{i=1}^h |U_i| - \sum_{\emptyset \neq I \subseteq J} \prod_{i \in I} |U_i \setminus U_i'| \prod_{i \not\in I} |U_i'|= \prod_{j \not\in J} |U_j| \prod_{j \in J} |U_j'|\]
	as we are left with those edges that have vertices in $U_j \setminus U_j'$ for $j \not\in J$ and in $U_j'$ for $j \in J$.
	To get to the third line, we used that for $j \not\in J$ we have $|U_j'| \ge (1-\delta')m=(1-\delta')|U_j|$ and that for each $i$ we have $|U_i \setminus U_i'| \le |U_i| \le \delta^{-1} |U_i'|$.
	In the last estimate we use the bound on $\eps$ and the choice of $\delta'$.
	This finishes the proof of~\eqref{eq:copies_F}.
	
	For~\eqref{eq:copies_F_good}, we repeat essentially the same argument we used for the lower bound of~\eqref{eq:copies_F}.
	Observe that from~\ref{claim:degree_god} of Lemma~\ref{lem:auxF}, if $|X| \ge 2\eps n d^{1-h}$, all but at most $\eps m$ vertices from each $U_i$ have degree at least $(1-h\eps)m^{h-1}$ in $F_X$.		
	Therefore
	\begin{equation}
		\label{eq_e(F)good}
		\EE[e(\tilde F_X')] =e(F_X') p^{e(H)} = (1 \pm h\tfrac{\eps}{\delta}) \prod_{i=1}^h |U_i'| p^{e(H)} = \Omega_{h,\delta,\eps}(Cn)
	\end{equation}	
	and again $\Delta[e(\tilde F_X')] = O_{h,\eps,\delta}(\EE[e(\tilde F_X')]^2 C^{-1} n^{-1})$.
	Then with Lemma~\ref{lem:janson} we have
	\begin{align*}
		\PP[e(\tilde F_X') < (1-\eps) \EE[e(\tilde F_X')]  ] \le 
		\exp\left(- \frac{\eps^2 \EE[e(\tilde F_X')]^2}{2\Delta[e(\tilde F_X')]+2\EE[e(\tilde F_X')]}\right) \le \exp (-h n)\, ,
	\end{align*}
	where the last inequality holds for large enough $C$.
	Then with a union bound over all choices of $U_1',\dots, U_h'$ we conclude that $e(\tilde F_X') \ge  (1-\eps) \EE[e(\tilde F_X')]$ with probability $1-e^{-n}$ and, using~\eqref{eq_e(F)good}, we finish the proof.
\end{proof}

\end{document}